\documentclass[10pt]{article}
\usepackage[utf8]{inputenc}
\usepackage{verbatim,amsmath,amssymb,amsfonts,amsthm,amscd,graphicx,psfrag,epsfig,mathrsfs,stmaryrd}
\usepackage{mathrsfs}
\usepackage{tikz-cd}
\usepackage[title]{appendix}
\usepackage{multirow}

\usepackage{mathtools}

\newcommand{\floor}[1]{\left \lfloor #1 \right \rfloor}

\usepackage[alphabetic]{amsrefs}
\usepackage{geometry}
\usepackage{bbm}
\usepackage{bm}
\usepackage{tikz-cd,hyperref} 
\usepackage[]{subcaption}
\usetikzlibrary{positioning,calc,arrows,decorations.pathreplacing,decorations.markings,patterns}
\usepackage{ifthen}
\usepackage[mathscr]{euscript}
\usepackage{hhline}
\tikzcdset{scale cd/.style={every label/.append style={scale=#1},
    cells={nodes={scale=#1}}}}

\makeatletter
\newcommand{\doublewidetilde}[1]{{%
  \mathpalette\double@widetilde{#1}%
}}
\newcommand{\double@widetilde}[2]{%
  \sbox\z@{$\m@th#1\widetilde{#2}$}%
  \ht\z@=.9\ht\z@
  \widetilde{\box\z@}%
}
\makeatother

\usepackage{blkarray,stmaryrd}

\newtheorem{theorem}{Theorem}[section]

\newtheorem{theorem-definition}[theorem]{Theorem-Definition}
\newtheorem{theorem-construction}[theorem]{Theorem-Construction}
\newtheorem{lemma-definition}[theorem]{Lemma--Definition}
\newtheorem{lemma-construction}[theorem]{Lemma--Construction}
\newtheorem{lemma}[theorem]{Lemma}
\newtheorem{proposition}[theorem]{Proposition}
\newtheorem{corollary}[theorem]{Corollary}
\newtheorem{conjecture}[theorem]{Conjecture}
\theoremstyle{definition}
\newtheorem{definition}[theorem]{Definition}
\newtheorem{remark}[theorem]{Remark}
\newtheorem{example}[theorem]{Example}

\newcommand{\old}[1]{}

\newcommand{\Z}{{\mathbb Z}}

\newcommand{\R}{{\mathbb R}}
\newcommand{\C}{{\mathbb C}}

\newcommand{\Q}{{\mathbb Q}}
\newcommand{\T}{{\mathbb T}}

\newcommand{\spectralcurve}{{\mathcal C}}
\renewcommand{\P}{{\mathbb P}}

\makeatletter
\newcommand{\extp}{\@ifnextchar^\@extp{\@extp^{\,}}}
\def\@extp^#1{\mathop{\bigwedge\nolimits^{\!#1}}}
\makeatother
\newcommand\restr[2]{{
  \left.\kern-\nulldelimiterspace 
  #1 
  \vphantom{\big|} 
  \right|_{#2} 
  }}
\definecolor{calpolypomonagreen}{rgb}{0, 0.6, 0.2}

\newcommand{\lms}{\longmapsto}
\newcommand{\lra}{\longrightarrow}

\newcommand{\ra}{\rightarrow}
\newcommand{\divebw}{Y}
\newcommand{\be}{\begin{equation}}
\newcommand{\ee}{\end{equation}}
\newcommand{\bt}{\begin{theorem}}

\newcommand{\et}{\end{theorem}}
\newcommand{\bd}{\begin{definition}}
\newcommand{\ed}{\end{definition}}
\newcommand{\bp}{\begin{proposition}}
\newcommand{\ep}{\end{proposition}}

\newcommand{\bl}{\begin{lemma}}
\newcommand{\el}{\end{lemma}}
\newcommand{\bc}{\begin{corollary}}
\newcommand{\ec}{\end{corollary}}
\newcommand{\bcon}{\begin{conjecture}}
\newcommand{\econ}{\end{conjecture}}
\newcommand{\la}{\label}
\newcommand{\w}{{\rm w}}
\newcommand{\bw}{{\rm b}}

\newcommand{\dd}{{\bf D}}
\newcommand{\dimer}{{\mathrm{m}}}

\def\deg{\operatorname{deg}}
\def\div{\operatorname{div}}

\def\spec{\operatorname{Spec}}

\DeclareMathOperator{\coker}{coker}

\DeclareMathOperator{\cone}{cone}
\newcommand{\red}[1]{{\color{red} #1}}

\tikzset{>=latex}
\tikzset{mid arrow/.style={postaction={decorate,decoration={
				markings,
				mark=at position .5 with {\arrow{latex}}
	}}},
	mid rarrow/.style={postaction={decorate,decoration={
				markings,
				mark=at position .5 with {\arrow{latex reversed}}
	}}},
}

\tikzset{qvert/.style={draw,black,circle,fill=gray,minimum size=5pt,inner sep=0pt}  } 
\tikzset{bvert/.style={draw,circle,fill=black,minimum size=5pt,inner sep=0pt}  }  
\tikzset{wvert/.style={draw,circle,fill=white,minimum size=5pt,inner sep=0pt}  } 
\tikzset{fvert/.style={text=blue}  } 
\tikzset{sqvert/.style={draw,black,rectangle,fill=black,minimum size=5pt,inner sep=0pt}  } 
\tikzset{lvert/.style={draw,circle,fill=black,minimum size=4pt,inner sep=0pt}  } 

\usetikzlibrary{positioning}
\usetikzlibrary{arrows}
\usetikzlibrary{calc}
\tikzset{shifted path/.style args={from #1 to #2 by #3}{insert path={
			let \p1=($(#1.east)-(#1.center)$),
			\p2=($(#2.east)-(#2.center)$),\p3=($(#1.center)-(#2.center)$),
			\n1={veclen(\x1,\y1)},\n2={veclen(\x2,\y2)},\n3={atan2(\y3,\x3)} in
			(#1.{\n3+180+asin(#3/\n1)}) to (#2.{\n3-asin(#3/\n2)})
}}}

\newcommand{\Addresses}{{
  \bigskip
  \footnotesize

  \textsc{University of Michigan, Department of Mathematics, 2844 East Hall, 530 Church Street, Ann Arbor, MI 48109-1043, USA}\par\nopagebreak
  \textit{E-mail address}: \texttt{georgete at umich.edu}

 \medskip
\textsc{Yale University, 12 Hillhouse, New Haven, CT 06511, USA}\par\nopagebreak
  \textit{E-mail address}: \texttt{alexander.goncharov at yale.edu}
  
   \medskip
  \textsc{Yale University, 12 Hillhouse, New Haven, CT 06511, USA}\par\nopagebreak
  \textit{E-mail address}: \texttt{richard.kenyon at yale.edu}

}}

\AtBeginDocument{
   \def\MR#1{}
}

\begin{document}

\title{The inverse spectral map for dimers}

\author{T. George  \and A. B. Goncharov  \and R. Kenyon }

\maketitle
\begin{abstract}
    In 2015, Vladimir Fock proved that the spectral transform, associating to an element of a dimer cluster integrable system its spectral data, is birational by constructing an inverse map using theta functions on Jacobians of spectral curves. We provide an alternate construction of the inverse map that involves only rational functions in the spectral data.
\end{abstract}
\tableofcontents

\section{Introduction}

The planar dimer model is a classical statistical mechanics model, involving the study of the 
set of \emph{dimer covers} (perfect matchings) of a planar, edge-weighted graph. In the 1960s,
Kasteleyn \cites{Kast61,Kast63} and Temperley and Fisher \cite{TF61} showed how to compute the (weighted) number of dimer covers
of planar graphs using the determinant of a signed adjacency matrix now known as the 
\emph{Kasteleyn matrix}. 

In mathematics the dimer model was popularized with the papers \cites{EKLP1,EKLP2} on the ``Aztec diamond" and later with results on the local statistics \cite{K97}, conformal invariance \cite{K00}, and limit shapes \cite{CKP},  connections with algebraic geometry \cites{KOS, KO}, cluster varieties and integrability \cite{GK12},
and string theory \cite{HK}. 

While the dimer model can be considered from a purely combinatorial point of view,
it also has a rich integrable structure, first described in \cite{GK12}.
The integrable structure on dimers on graphs on the torus 
was found to generalize many well-known integrable systems, see for example
\cite{FM} and \cite{AGR}. What is especially important is that the related integrable system is of cluster nature, and this allows one to immediately quantize it, getting a quantum integrable system. 

From the point of view of classical mechanics, 
associated to the dimer model on a bipartite graph on a torus 
(or equivalently a periodic bipartite planar graph) is a Poisson variety with a Hamiltonian integrable
system. Underlying this system is an algebraic curve $\spectralcurve = \{P(z,w)=0\}$ 
(called the \emph{spectral curve}) and a divisor on this curve--essentially a set of $g$ distinct points 
$\{(p_1,q_1),\dots,(p_g,q_g)\}$ on $\spectralcurve$. 
This is the \emph{spectral data} associated to the model.
It was shown in \cite{KO} that the map from the weighted graph to the spectral data was bijective,
from the space of ``face weights" (see below) to the moduli space of genus-$g$ 
curves and effective degree-$g$ divisors on the open spectral curve $\spectralcurve^\circ$.
Subsequently Fock \cite{Fock} constructed the inverse spectral map (from the spectral data to the face weights), describing it in terms of theta functions over the spectral curve. The special case of genus $0$ was described earlier in \cites{Kenyon.isoradial, KO} and an explicit construction in the case of genus $1$ was more recently given in \cite{BCdT}. Positivity of Fock's inverse map was studied in \cite{BCdT1}.    

In the current paper, we show that the inverse map can be given an explicit \emph{rational} expression in terms of the divisor points $(p_i,q_i) \in \spectralcurve^\circ$ and the points of $\spectralcurve$ at toric infinity.
An exact statement is given in Theorem \ref{Th2.3} below. 

While Fock's construction is very natural and interacts nicely with  positivity, it involves theta functions. Our construction gives the inverse map as ratios of certain determinants in the spectral data and can be explicitly computed using computer algebra. 
We briefly describe our construction now. The spectral data is defined via a matrix $K=K(z,w)$ called the Kasteleyn matrix, whose rows are indexed by white vertices, columns by black vertices, and whose entries are Laurent polynomials in $z$ and $w$. Let  us consider  the adjugate matrix of $K$:
\[
Q={Q(z,w)=}K^{-1}\det K.
\]
 The matrix $Q$ is important when studying the probabilistic aspects of the dimer model (on the 
lift of the graph on the torus to the plane): the edge occupation variables form a determinantal process whose kernel is given by the Fourier coefficients of $Q/P$, as discussed in \cite{KOS}. In the present work,
we have a different use for $Q$: finding (a column of) the matrix $Q$ from the spectral data allows us to reconstruct the face weights and
thereby invert the spectral transform. 

The points $(p_i,q_i)\in \spectralcurve$ are defined to be the points where a column of $Q$, 
corresponding to a {fixed} white vertex ${\bf w}$,  vanishes. We show that entries in the ${\bf w}$-column of $Q$, which are Laurent polynomials, can be reconstructed from the spectral data by solving a linear system of equations. Some of the linear equations are easy to describe: for {any} black vertex $\bw$, we have $Q_{\bw {\bf w}}(p_i,q_i)=0$ for $i=1,\dots,g$, which are $g$ linear equations in the coefficients of {the Laurent polynomial} $Q_{\bw {\bf w}}$. However, these equations are usually not sufficient to determine the coefficients of $Q_{\bw {\bf w}}$. We find additional equations from the vanishing of $Q_{\bw {\bf w}}$ at certain points at infinity of the spectral curve $\spectralcurve$, and show that these equations  determine $Q_{\bw {\bf w}}$ uniquely, up to a non-zero constant. We then give a procedure to reconstruct the weights from the ${\bf w}$-column of $Q$. 

{A key construction in our approach is the extension of the Kasteleyn matrix $K$ to a map of vector bundles on a toric stack, for which we make crucial use of the classification of line bundles on toric stacks and the computation of their cohomology developed in \cite{BH09}. Toric stacks already appear implicitly in the context of the spectral transform in \cite{KO} and explicitly \cite{TWZ18}.}

The article is organized as follows. In Section \ref{sec:background} we review the dimer cluster integrable system and the spectral transform. In Section \ref{sec2}, we state Theorem \ref{Vbwthm}, which is our main result, and describe the reconstruction procedure. We work out two detailed examples in Section \ref{sec:example}. Sections \ref{smallpolysection}, \ref{extensionsection} and \ref{laurentsection} contain proofs of our results. In Appendix~\ref{A}, we review results from toric geometry. In Appendix~\ref{B}, we provide explicit combinatorial descriptions for some of our constructions. These are useful for computations.

\vskip 2mm
{\bf Acknowledgments}. The work of A.G. was supported by the NSF grants DMS-1900743, DMS-2153059. 
Work of R. K. was supported by NSF grant DMS-1940932 and the Simons Foundation grant 327929. We thank the referee for their careful reading of the manuscript and their numerous suggestions.

\section{Background}\la{sec:background}

For further information about the material in this section see \cite{GK12}.

\subsection{Dimer models}

Let $\Gamma$ be a bipartite graph on the torus $\T\cong S^1\times S^1$ such that 
  the connected components of the complement of $\Gamma$---the faces---are contractible. We denote by $B(\Gamma)$ and $W(\Gamma)$ the black and white vertices of $\Gamma$, by $V(\Gamma)$ the  vertices, and by $E(\Gamma)$ the edges of $\Gamma$. When the graph is clear from context, we will usually abbreviate these to $B,W, V$ and $E$.

A \textit{dimer model on the torus} is a pair $(\Gamma, [{wt}])$, where $\Gamma$ is a bipartite graph on the torus as above and
$[wt] \in H^1(\Gamma,\C^\times)$ (Here and throughout the paper, $\C^\times$ denotes the group of nonzero complex numbers under multiplication).
 For a loop $L$ and a cohomology class $[wt]$, we denote by  $[wt]([L])$ the pairing between the cohomology and  the  homology.  
We  orient edges from their black vertex to their white vertex.
The cohomology class $[wt]$ can be represented by a cocycle $wt$ which, using this 
 orientation, can be identified with a 
$\C^\times-$valued function on the edges of $\Gamma$ called an \emph{edge weight}.

The edge weight is well-defined modulo multiplication by coboundaries, which are functions on edges $e=\bw \w$
given by $ f(\w)f(\bw)^{-1}$  for functions $f:V(\Gamma)\to\C^\times$. 
Note that the weight of a loop is not the product of its edge weights,
but the ``alternating product" of its edge weights: edges oriented against the orientation of the loop are
multiplied with exponent $-1$.

A \textit{dimer cover} or \textit{perfect matching} $\dimer$ of $\Gamma$ is a subset of $E(\Gamma)$ such that each vertex of $\Gamma$ is incident to exactly one edge in $\dimer$.
Let $\mathscr M$ denote the set of dimer covers of $\Gamma$. If we fix a reference dimer cover $\dimer_0$, we get a function
\begin{align*}
    \pi_{\dimer_0}: \mathscr M &\ra H_1(\T,\Z)\\
    \dimer &\mapsto [\dimer-\dimer_0].
\end{align*}
Here   $\dimer-\dimer_0$ is the $1$-chain which assigns $1$ to (oriented) edges of $\dimer$ and $-1$ to (oriented) 
edges of $\dimer_0$, so $\dimer-\dimer_0$ is a union of oriented cycles and doubled edges, whose homology
class is $[\dimer-\dimer_0]$.

The \emph{Newton polygon} of $\Gamma$ is the polygon
\[
N(\Gamma) :=\text{Convex-hull}(\pi_{\dimer_0} (\mathscr M)) \subset H_1(\T,\R)
\]
{defined modulo translation by $H_1(\T,\Z)$. Changing the reference dimer cover from $\dimer_0$ to $\dimer_0'$ results in a translation of the polygon by $[\dimer_0-\dimer_0']$, so the Newton polygon does not depend on the choice.}

We assume that $\Gamma$ is such that $N(\Gamma)$ has interior. This is a nondegeneracy condition on $\Gamma$. {(When $N$ has empty interior, the graph $\Gamma$ is equivalent under certain elementary
transformations to a graph whose lift to $\R^2$ is disconnected, that is, has noncontractible faces;
such a graph breaks into essentially one-dimensional components, and there is no integrable system.)}

\subsection{Zig-zag paths and the Newton polygon}

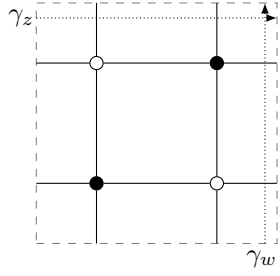
\begin{figure}
\centering
	\begin{tikzpicture}[scale=0.8,baseline={([yshift=-.7ex]current bounding box.center)}]  
	
		\draw[dashed, gray] (0,0) rectangle (4,4);
	
		\coordinate[bvert] (w2) at (1,1);
			\coordinate[bvert] (w1) at (3,3);
			
			\coordinate[wvert] (b1) at (1,3);
			\coordinate[wvert] (b2) at (3,1);
			
			\draw[-]
				(b1) 
				edge  (w1) edge  (w2) edge (1,4) edge (0,3)
				;
			\draw[-] (b2) 
				edge  (w2) edge  (w1) edge (4,1) edge (3,0)
				;
			\draw[-]	(w1) edge (3,4) edge (4,3);
	\draw[-]	(w2) edge (1,0) edge (0,1);
	\draw[->,densely dotted] (3.8,0) -- (3.8,4);
		\draw[->,densely dotted] (0,3.75) -- (4,3.75);
		\node (no) at (-0.25,3.75) {$\gamma_z$};
		\node (no) at (3.75,-0.25) {$\gamma_w$};

	\end{tikzpicture}
\caption{The fundamental rectangle $R$, along with the cycles $\gamma_z,\gamma_w$.}\label{figgzgw}
\end{figure}
A \emph{zig-zag path} in $\Gamma$ is a closed path that turns maximally right at each black vertex and maximally left at each white vertex. {The \textit{medial graph} of $\Gamma$ is the graph $\Gamma^\times$ that has a vertex $v_e$ at the mid-point of each edge $e$ of $\Gamma$ and an edge between $v_{e}$ and $v_{e'}$ whenever $e$ and $e'$ occur consecutively around a face of $\Gamma$. Note that by construction, each vertex of $\Gamma^\times$ has degree $4$. A zig-zag path in $\Gamma$ corresponds to a cycle in $\Gamma^\times$ that goes straight through each degree four vertex, i.e., at every vertex, the outgoing edge of the cycle is the one that is opposite the incoming one (see Figure~\ref{figmedial}). Hereafter, when we say zig-zag path, we mean the corresponding cycle in the medial graph. }

Let $\widetilde \Gamma$ be the biperiodic graph on the plane given by the lift of $\Gamma$ to the universal cover of $\T$. The bipartite graph $\Gamma$ is said to be \textit{minimal} if the lift
of any zig-zag path does not self-intersect, and lifts of any two zig-zag paths do not have ``parallel bigons'',
where by \emph{parallel bigon} we mean two consecutive intersections where both paths are oriented in the same direction from one to the next.
For a minimal bipartite graph $\Gamma$ on the torus, the Newton polygon has an alternative description in terms of 
the zig-zag paths of $\Gamma$. 
Namely, since $\Gamma$ is embedded in $\T$, each zig-zag path $\alpha$ has a non-zero homology class $[\alpha] \in H_1(\T,\Z)$. The polygon $N(\Gamma)$ is the unique convex integral polygon defined modulo translation in $H_1(\T,\Z)$ whose integral primitive edge vectors in counterclockwise order around $N$ are given by the vectors $[\alpha]$ for all zig-zag paths~$\alpha$.

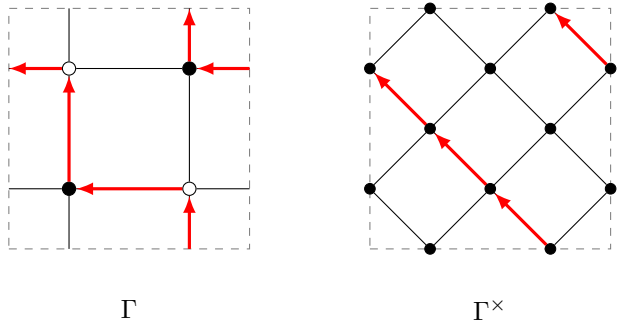
\begin{figure}
	\centering
	
	\begin{tikzpicture}[scale=0.8]
\begin{scope}
	\node[](no) at (2,-1) {$\Gamma$};
	\draw[dashed, gray] (0,0) rectangle (4,4);
	
	\coordinate[bvert] (w2) at (1,1);
	\coordinate[bvert] (w1) at (3,3);
	
	\coordinate[wvert] (b1) at (1,3);
	\coordinate[wvert] (b2) at (3,1);
	
	\draw[-]
	(b1) 
	edge  (w1) edge  (w2) edge (1,4) edge (0,3)
	;
	\draw[-] (b2) 
	edge  (w2) edge  (w1) edge (4,1) edge (3,0)
	;
	\draw[-]	(w1) edge (3,4) edge (4,3);
	\draw[-]	(w2) edge (1,0) edge (0,1);
	\draw[->,red,very thick] (3,0)--(b2);
	\draw[->,red,very thick](b2)--(w2);
	\draw[->,red,very thick]
	(w2)--(b1);
	\draw[->,red,very thick]
	(b1)--(0,3);
	\draw[->,red,very thick]
	(4,3)--(w1);
	\draw[->,red,very thick]
	(w1)--(3,4)
	;
\end{scope}		

\begin{scope}[shift={(6,0)}]
	\node[](no) at (2,-1) {$\Gamma^\times$};
	\draw[dashed, gray] (0,0) rectangle (4,4);
	
	\coordinate[lvert] (01) at (0,1);
	\coordinate[lvert] (03) at (0,3);
	
		\coordinate[lvert] (41) at (4,1);
	\coordinate[lvert] (43) at (4,3);
	\coordinate[lvert] (10) at (1,0);
	\coordinate[lvert] (14) at (1,4);
	
	\coordinate[lvert] (12) at (1,2);
	\coordinate[lvert] (32) at (3,2);
		\coordinate[lvert] (30) at (3,0);
			\coordinate[lvert] (34) at (3,4);
	
	\coordinate[lvert] (21) at (2,1);
	\coordinate[lvert] (23) at (2,3);
	
	\draw[] (10) -- (43)
	(01) -- (34)
	(03) -- (14)
	(30)--(41)
	(10)--(01)
	(30)--(03)
	(41)--(14)
	(43)--(34)
	;
\draw[->, red, very thick]	(30)--(21); 
\draw[->, red, very thick]	(21)--(12); 
\draw[->, red, very thick]	(12)--(03); 
\draw[->, red, very thick]	(43)--(34); 

\end{scope}
		
	\end{tikzpicture}	
	\caption{{A zig-zag path in a graph $\Gamma$ and the corresponding cycle in the medial graph $\Gamma^\times$.}}\label{figmedial}
\end{figure}

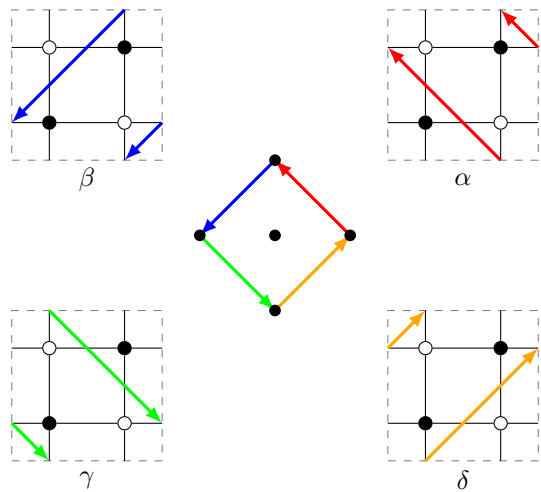
\begin{figure}
\begin{center}
	\begin{tikzpicture}[baseline={([yshift=-12ex]current bounding box.center)}] 
 \def\lw{1.5};
	\begin{scope}
    \begin{scope}[decoration={
    markings,
    mark=at position 1 with {\arrow{>}}}
    ] 
	\draw[postaction={decorate},red,very thick] (1,0)--(0,1);
		\draw[postaction={decorate},blue,very thick] (0,1)--(-1,0);
		\draw[postaction={decorate},green,very thick] (-1,0)--(0,-1);
		\draw[postaction={decorate},orange!70!yellow,very thick] (0,-1)--(1,0);
\end{scope}
		
				\draw[fill=black] (0,0) circle (2pt);
		
		\draw[fill=black] (0,1) circle (2pt);
		\draw[fill=black] (1,0) circle (2pt);
		\draw[fill=black] (0,-1) circle (2pt);
			\draw[fill=black] (-1,0) circle (2pt);
\end{scope}			
\begin{scope}[scale=0.5,shift={(3,2)}]
\node[](no) at (2,-0.5) {$\alpha$};
    		\draw[dashed, gray] (0,0) rectangle (4,4);
	
		\coordinate[bvert] (w2) at (1,1);
			\coordinate[bvert] (w1) at (3,3);
			
			\coordinate[wvert] (b1) at (1,3);
			\coordinate[wvert] (b2) at (3,1);
			
			\draw[-]
				(b1) 
				edge  (w1) edge  (w2) edge (1,4) edge (0,3)
				;
			\draw[-] (b2) 
				edge  (w2) edge  (w1) edge (4,1) edge (3,0)
				;
			\draw[-]	(w1) edge (3,4) edge (4,3);
	\draw[-]	(w2) edge (1,0) edge (0,1);
	\draw[->,red,very thick] (3,0)--(0,3);
	\draw[->,red,very thick](4,3)--(3,4);

\end{scope}	
\begin{scope}[scale=0.5,shift={(-7,2)}]

\node[](no) at (2,-0.5) {$\beta$};
    		\draw[dashed, gray] (0,0) rectangle (4,4);
	
		\coordinate[bvert] (w2) at (1,1);
			\coordinate[bvert] (w1) at (3,3);
			
			\coordinate[wvert] (b1) at (1,3);
			\coordinate[wvert] (b2) at (3,1);
			
			\draw[-]
				(b1) 
				edge  (w1) edge  (w2) edge (1,4) edge (0,3)
				;
			\draw[-] (b2) 
				edge  (w2) edge  (w1) edge (4,1) edge (3,0)
				;
			\draw[-]	(w1) edge (3,4) edge (4,3);
	\draw[-]	(w2) edge (1,0) edge (0,1);
	\draw[->,blue,very thick] (3,4)--(0,1);
	\draw[->,blue,very thick](4,1)--(3,0);
\end{scope}	

\begin{scope}[scale=0.5,shift={(3,-6)}]
\node[](no) at (2,-0.5) {$\delta$};
    		\draw[dashed, gray] (0,0) rectangle (4,4);
	
		\coordinate[bvert] (w2) at (1,1);
			\coordinate[bvert] (w1) at (3,3);
			
			\coordinate[wvert] (b1) at (1,3);
			\coordinate[wvert] (b2) at (3,1);
			
			\draw[-]
				(b1) 
				edge  (w1) edge  (w2) edge (1,4) edge (0,3)
				;
			\draw[-] (b2) 
				edge  (w2) edge  (w1) edge (4,1) edge (3,0)
				;
			\draw[-]	(w1) edge (3,4) edge (4,3);
	\draw[-]	(w2) edge (1,0) edge (0,1);
	\draw[->,orange!70!yellow,very thick] (1,0)--(4,3);
	\draw[->,orange!70!yellow,very thick](0,3)--(1,4);

\end{scope}	
\begin{scope}[scale=0.5,shift={(-7,-6)}]
\node[](no) at (2,-0.5) {$\gamma$};
    		\draw[dashed, gray] (0,0) rectangle (4,4);
	
		\coordinate[bvert] (w2) at (1,1);
			\coordinate[bvert] (w1) at (3,3);
			
			\coordinate[wvert] (b1) at (1,3);
			\coordinate[wvert] (b2) at (3,1);
			
			\draw[-]
				(b1) 
				edge  (w1) edge  (w2) edge (1,4) edge (0,3)
				;
			\draw[-] (b2) 
				edge  (w2) edge  (w1) edge (4,1) edge (3,0)
				;
			\draw[-]	(w1) edge (3,4) edge (4,3);
	\draw[-]	(w2) edge (1,0) edge (0,1);
	\draw[->,green,very thick] (1,4)--(4,1);
	\draw[->,green,very thick](0,1)--(1,0);
\end{scope}

	\end{tikzpicture}

\end{center}
	\caption{Zig-zag paths and Newton polygon {for the bipartite graph in Figure~\ref{figgzgw}.}}
	\label{fig:npzz}
\end{figure}

\begin{example}\la{eg:zz}
Consider the fundamental domain for the square lattice shown in Figure \ref{figgzgw}, and let $\gamma_z,\gamma_w$ be cycles generating $H_1(\T,\Z)$ as shown there. We will write homology classes in $H_1(\T,\Z)$ in the basis $(\gamma_z,\gamma_w)$. There are four zig-zag paths labeled $\alpha,\beta,\gamma$ and $\delta$ with homology classes  $(-1,1),(-1,-1),(1,-1)$ and $(1,1)$ respectively (Figure  \ref{fig:npzz}), and therefore the Newton polygon is 
\[
\text{Convex-hull}\{(1,0),(0,1),(-1,0),(0,-1)\}.
\]
\end{example}

\subsection{The cluster variety assigned to a Newton polygon }

\begin{figure}
	\centering
		\begin{tikzpicture}[scale=0.5]
		\node[](no) at (0,0){$\longleftrightarrow$};
		\node[](no) at (0,-3){spider move};
		\begin{scope}[shift={(3,0)},rotate=0]
			\def\r{2};
			\coordinate[wvert] (n1) at (0:\r);
			\coordinate[wvert] (n2) at (0+90:\r);
			\coordinate[wvert] (n3) at (0+180:\r);
			\coordinate[wvert] (n4) at (0+270:\r);
			
			\coordinate[bvert] (b1) at (0:0.5*\r);
			\coordinate[bvert] (b2) at (0+180:0.5*\r);
			
			\draw[-]
			(n1) --  (b1) --  (n2)--(b2)--(n4) -- (b1)
			(b2)--(n3)
			;
			\coordinate[] (t1) at (15:\r);
			\coordinate[] (t2) at (120-45:\r);
			\coordinate[] (t3) at (150-45:\r);
			\coordinate[] (t4) at (210-45:\r);
			\coordinate[] (t5) at (240-45:\r);
			\coordinate[] (t6) at (300-45:\r);
			\coordinate[] (t7) at (330-45:\r);
			\coordinate[] (t8) at (30-45:\r);

		\end{scope}
		
		\begin{scope}[shift={(-3,0)},rotate=90]
			\def\r{2};
			\coordinate[wvert] (n1) at (0:\r);
			\coordinate[wvert] (n2) at (0+90:\r);
			\coordinate[wvert] (n3) at (0+180:\r);
			\coordinate[wvert] (n4) at (0+270:\r);
			\coordinate[bvert] (b1) at (0:0.5*\r);
			\coordinate[bvert] (b2) at (0+180:0.5*\r);
			
			\draw[-]
			(n1)--(b1)--(n2)--(b2)--(n4)--(b1)
			(b2)--(n3)
			;
			\coordinate[] (t1) at (15:\r);
			\coordinate[] (t2) at (120-45:\r);
			\coordinate[] (t3) at (150-45:\r);
			\coordinate[] (t4) at (210-45:\r);
			\coordinate[] (t5) at (240-45:\r);
			\coordinate[] (t6) at (300-45:\r);
			\coordinate[] (t7) at (330-45:\r);
			\coordinate[] (t8) at (30-45:\r);

		\end{scope}
	
	\begin{scope}[shift={(15,0)}]
\node[](no) at (0,0){$\longleftrightarrow$};
		\node[](no) at (0,-3){contraction-uncontraction move};
\begin{scope}[shift={(3,0)},rotate=0]
	\def\r{2};
	\coordinate[] (n1) at (0:\r);
	\coordinate[] (n2) at (30:\r);
	\coordinate[] (n3) at (-30:\r);
	\coordinate[] (n4) at (160:\r);
	\coordinate[] (n6) at (200:\r);
	
	\coordinate[wvert] (w1) at (0:0);

	\path[] (w1) edge (n1) edge (n2) edge (n3) 
	(w1) edge (n4) edge (n6);
	
\end{scope}

\begin{scope}[shift={(-3,0)}]
	\def\r{2};
	\coordinate[] (n1) at (0:\r);
	\coordinate[] (n2) at (30:\r);
	\coordinate[] (n3) at (-30:\r);
	\coordinate[] (n4) at (160:\r);
	\coordinate[] (n6) at (200:\r);
	
	\coordinate[wvert] (w1) at (0:0.5*\r);
	\coordinate[wvert] (w2) at (180:0.5*\r);
	\coordinate[bvert] (b1) at (0:0);
	
	\path[] (w1) edge (n1) edge (n2) edge (n3) edge (w2)
	(w2) edge (n4) edge (n6);
\end{scope}	
\end{scope}
	\end{tikzpicture}
	\caption{{The elementary transformations.}}\label{fig:et}
\end{figure}
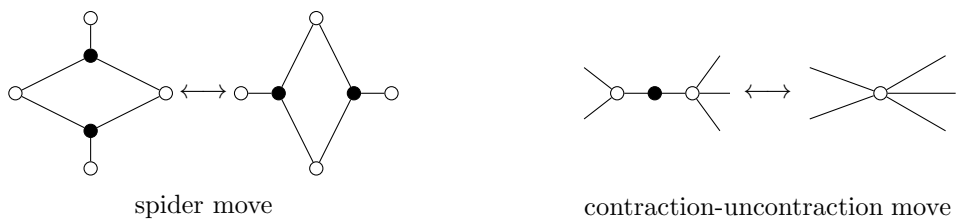

For a convex integral polygon $N \subset H_1(\T,\R)$ defined modulo translation, consider the family of minimal bipartite graphs $\Gamma$ with Newton polygon $N(\Gamma)=N$. Any two graphs $\Gamma_1, \Gamma_2$ in the family are related by certain \textit{elementary transformations}; {see Figure~\ref{fig:et}}. An elementary transformation $\Gamma_1 \to \Gamma_2$ gives rise to a birational map $H^1(\Gamma_1, \C^\times) \dashrightarrow H^1(\Gamma_2, \C^\times)$. Gluing the tori $H^1(\Gamma, \C^\times)$  by these maps, we obtain a space $\mathcal X_N$, called the \textit{dimer cluster Poisson  variety}. 
It carries a canonical Poisson structure. The Poisson center is generated by {the loop weights
of the zig-zag paths}. The space $\mathcal X_N$ is the phase space of the cluster 
integrable system. See details in \cite{GK12}.

\subsection{Some notation}
Let $\Sigma$ denote the normal fan of $N$ {(see Section~\ref{A2} and Figures~\ref{fig:hex} and \ref{fig:sqoct})} so that the set of rays $\Sigma(1)=\{\rho\}$ of $\Sigma$ is in bijection with the set of edges of $N$. We denote the edge of $N$ whose inward normal is directed along the ray $\rho$  by $E_\rho$, and the primitive vector along $\rho$ by $u_\rho$.

{Let ${\rm M}:=H_1(\T,\Z){\cong \Z^2}~\text{and}~{\rm M}^\vee:=\rm{Hom}_\Z(\rm M,\Z){\cong \Z^2}
$ be dual lattices and let $\langle *,* \rangle : {\rm M} \times {\rm M}^\vee \ra \Z$ denote the duality pairing}. Let us consider the algebraic torus with lattice of characters ${\rm M}$:
$$
{\mathrm T}:= \mathrm{Hom}_\Z({\rm M},\C^\times) \cong (\C^\times)^2.
$$
Let ${\rm M}_\R$ (resp. ${\rm M}^\vee_\R$) denote ${\rm M} \otimes_\Z \R$ (resp. ${\rm M}^\vee \otimes_\Z \R$), so that $N \subset {\rm M}_\R$ and $\Sigma \subset {\rm M}^\vee_\R$.

An elementary transformation $\Gamma_1 \ra \Gamma_2$ induces a canonical bijection between zig-zag paths in $\Gamma_1$ and zig-zag paths in $\Gamma_2$. Therefore, the set of zig-zag paths is canonically associated with $N$. We denote the set of zig-zag paths by $Z$, and for an edge $E_\rho$ of $N$, we denote  by $Z_\rho$ the set of zig-zag paths $\alpha$ such that the primitive vector $[\alpha]$ is contained in $E_\rho$.

\subsection{The Kasteleyn matrix}

Let $R$ be a fundamental rectangle for $\T$, so that $\T$ is obtained by gluing together opposite sides of $R$. Let $\gamma_z,\gamma_w$ be the oriented sides of $R$ generating $H_1(\T,\Z)$, as  shown in Figure \ref{figgzgw}.  Let $z$ (resp. $w$) denote the character $\chi^{\gamma_w}$ (resp. $\chi^{\gamma_z}$), so the coordinate ring of ${\mathrm T}$ is $\C[z^{\pm 1},w^{\pm 1}]$. 

Let $(*,*)_\T$ 
be the intersection pairing on $H_1(\T,\Z)$. 
For $z,w\in\C^\times$ we multiply edge weights on edges crossing $\gamma_z$ by $z^{\pm1}$ and
those crossing $\gamma_w$ by $w^{\pm1}$, with the sign determined by the orientation.
Precisely, we multiply by 
\be \la{edgeph}
\phi(e):=z^{(e,\gamma_w)_\T} w^{(e,-\gamma_z)_\T}, 
\ee 
Here 
$(e,*)_\T:=(l_e,*)_\T$ is the intersection index  with the oriented loop  $l_e$ obtained by concatenating $e = {\rm b}{\rm w}$ with an oriented path contained in $R$ from ${\rm w}$ to ${\rm b}$. 
\old{
\red{Let $H_1(\T,\Z)^\vee:=\text{Hom}_\Z(H_1(\T,\Z),\Z)$ be the dual lattice of $H_1(\T,\Z)$.}
There is an isomorphism $T:=H_1(\T,\Z)^\vee \otimes \C^\times \cong (\C^\times)^2$, defined as follows.   
For each edge $e$ of $\Gamma$, we associate a character, that is a group homomorphism $T \ra \C^\times$:
\be \la{edgeph}
\varphi(e)=z^{( e,\gamma_z )} w^{(e,\gamma_w)},
\ee
where $( e,\gamma_z )$ is the intersection index  of the edge $e$ and   $\gamma_z$, and similarly $( e,\gamma_w )$.  Explicitly we fix an embedding of $\Gamma$ in the fundamental rectangle. Isotoping edges if necessary, we may assume that each edge of $\Gamma$ intersects $\gamma_z$ and $\gamma_w$ only finitely many times. For an edge $e$, let $I_1, ..., I_n$ be the intersection points of $e$ with $\gamma_z$. We define $( e,\gamma_z ):=\sum_{i=j}^n (e,\gamma_z)_{I_j}$, where $(e,\gamma_z)_{I_j} \in \{-1,0,1\}$ is the local intersection index, where we orient $e$ from its black vertex to its white vertex.\\
}

A \emph{Kasteleyn sign} is a cohomology class $[\epsilon] \in H^1(\Gamma,\C^\times)$ such that for any loop $L$ in $\Gamma$, $[\epsilon]([L])$ is $-1$ (resp., $1$) if the number of edges in $L$ is $0$ mod $4$ (resp., $2$ mod $4$). Given edge weights $wt$ and $\epsilon$ representing $[wt]$ and $[\epsilon]$ respectively, one defines the {\it Kasteleyn matrix} $K {=K(z,w)}$, whose columns and rows are parameterized by $\bw\in B $ and $\w \in W$ respectively:
\begin{align}\label{Kastdet}
    K_{\w,\bw}&=\sum_{e \in E \text{ incident to } \bw,\w} wt(e) \epsilon(e)\phi(e).
\end{align}
It describes  a map of free   $\C[z^{\pm 1}, w^{\pm 1}]$-modules, called the {\it Kasteleyn  operator}:
\begin{align}\label{Kastdet1}
    K:~&\C[z^{\pm 1}, w^{\pm 1}]^{B} \ra  \C[z^{\pm 1}, w^{\pm 1}]^{W},\\
 &\delta_{\bw} \lms    \sum_{\w \in W} K_{\w, \bw} \delta_{\w}.
\end{align}

\begin{theorem}[Kasteleyn 1963, \cite{Kast63}]\label{Kastthm}
Fix a dimer cover $\dimer_0$, and let $\phi(\dimer_0)=\prod_{e \in \dimer_0} \phi(e)$. Then,
\[
\frac{1}{{ wt}(\dimer_0) \epsilon(\dimer_0) \phi(\dimer_0) } \det K= \sum_{\dimer \in {\mathscr M}} \mathrm{sign}([\dimer-\dimer_0]) [{wt}]([\dimer-\dimer_0]) \chi^{[\dimer-\dimer_0]},
\]
where $\mathrm{sign}({[\dimer-\dimer_0]}) \in \{\pm 1\}$ is a sign that depends only on the homology class $[\dimer-\dimer_0]$ and $[\epsilon]$.
\end{theorem}

The  \textit{characteristic polynomial} is the Laurent polynomial 
\[
P(z,w):=\frac{1}{{wt}(\dimer_0) \epsilon(\dimer_0) \phi(\dimer_0) } \det K.\]
Its vanishing locus $\spectralcurve^\circ:=\{P(z,w)=0\} \subset (\C^\times)^2$ is called the \textit{(open part of the) spectral curve}. Theorem $\ref{Kastthm}$ implies that $N$ is the Newton polygon of $P(z,w)$.  
Although the definition of the Kasteleyn matrix uses cocycles representing the cohomology classes $wt$ and $\epsilon$, the spectral curve does not depend on these choices.
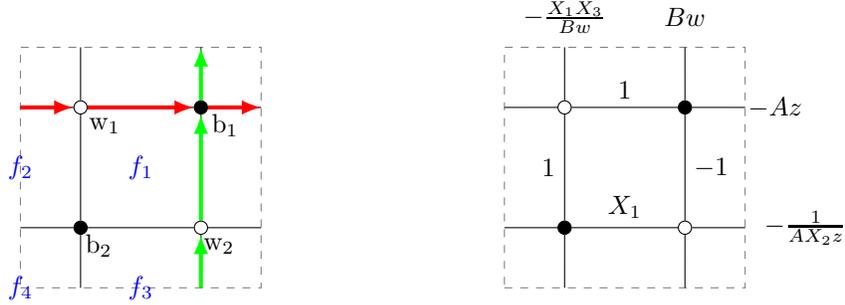
\begin{figure}
\centering
	\begin{tikzpicture}[scale=0.8]  
\clip (-0.5,-0.5) rectangle (6,6);
\def\lw{1.5};
		\draw[dashed, gray] (0,0) rectangle (4,4);
	
		\coordinate[bvert] (w2) at (1,1);
			\coordinate[bvert] (w1) at (3,3);
			
			\coordinate[wvert] (b1) at (1,3);
			\coordinate[wvert] (b2) at (3,1);
			
			\draw[-]
				(b1) 
				edge  (w1) edge  (w2) edge (1,4) edge (0,3)
				;
			\draw[-] (b2) 
				edge  (w2) edge  (w1) edge (4,1) edge (3,0)
				;
			\draw[-]	(w1) edge (3,4) edge (4,3);
	\draw[-]	(w2) edge (1,0) edge (0,1);
	\node[](no) at (1.4,2.7){$\w_1$};
	\node[](no) at (3.3,0.7){$\w_2$};
	\node[](no) at (3.4,2.7){$\bw_1$};
	\node[](no) at (1.3,0.7){$\bw_2$};
	
	\node[blue](no) at (2,2){$f_1$};
	\node[blue](no) at (0,2){$f_2$};
	\node[blue](no) at (2,0){$f_3$};
	\node[blue](no) at (0,0){$f_4$};
	\draw[->,green,line width = \lw] (3,0)--(b2);
	\draw[->,green,line width = \lw](b2)--(w1);
	\draw[->,green,line width = \lw]
	(w1)--(3,4);
	
		\draw[->,red,line width = \lw] (0,3)--(b1);
			\draw[->,red,line width = \lw] (b1)--(w1);
			\draw[->,red,line width = \lw] (w1)--(4,3);

	\end{tikzpicture}
	\hspace{10mm} 	\begin{tikzpicture}[scale=0.8]
		\clip (-0.5,-0.5) rectangle (6,6); 
		\draw[dashed, gray] (0,0) rectangle (4,4);
	
		\coordinate[bvert] (w2) at (1,1);
			\coordinate[bvert] (w1) at (3,3);
			
			\coordinate[wvert] (b1) at (1,3);
			\coordinate[wvert] (b2) at (3,1);
			
			\draw[-]
				(b1) 
				edge node[above] {$1$} (w1) edge node[left] {$1$} (w2) edge (1,4) edge (0,3)
				;
			\draw[-] (b2) 
				edge node[above] {$X_1$} (w2) edge node[right] {${-}1$} (w1) edge (4,1) edge (3,0)
				;
			\draw[-]	(w1) edge (3,4) edge (4,3);
	\draw[-]	(w2) edge (1,0) edge (0,1);
\node (no) at (1,4.5) {${-}\frac{X_1 X_3}{B{w}}$};
\node (no) at (3,4.5) {$B{w}$};
\node (no) at (5,1) {${-}\frac{1}{ A X_2 {z}}$};
\node (no) at (4.5,3) {$-{A {z}}$};

	\end{tikzpicture}
\caption{Shown on the {left} is a labeling of vertices and faces of $\Gamma$, and two cycles $a$ (red) and $b$ (green) in $\Gamma$ that generate $H_1(\T,\Z)$. Shown on the right is a cocycle representing $[wt]$, along with $\epsilon$ and $\phi$. {The signs are due to $\epsilon$, the $z,w$ due to $\phi$, and other weights are $wt$.} }
\label{figds}
\end{figure}
\begin{example}
Let $a$ and $b$ be the two cycles in $\Gamma$ shown on the {left} of Figure \ref{figds} whose projections to $\T$ generate $H_1(\T,\Z)$. Let $[wt] \in H^1(\Gamma,\C^\times)$ and let $A:=[wt]([a]), B:=[wt]([b])$. For $i=1,2,3$, let $X_i$ denote the $[wt]([\partial f_i])$, where $\partial f_i$ denotes the boundary of the face $f_i$ (the weight of the fourth face is determined by the fact that the product of all face weights is $1$).  Then $(X_1,X_2,X_3,A,B)$ generate the coordinate ring of $H^1(\Gamma,\C^\times)$. A cocycle representing $[wt]$ is shown on the {right} of Figure \ref{figds}, along with $\epsilon$ and $\phi$. The Kasteleyn matrix and the spectral curve are:
\begin{align}
K&=\begin{blockarray}{ccc}
{\bw_1}& \bw_2 \\
\begin{block}{(cc)c}
  1-A z & 1-\frac{X_1 X_3}{B w} & \w_1\\
  -1+Bw & X_1-\frac{1}{A X_2 z} & \w_2\\
\end{block}
\end{blockarray},\nonumber \\
P(z,w)&=\left(1  + X_1 + \frac{1}{X_2} + X_1 X_3\right)- B w - \frac{X_1 X_3}{B w} - \frac{1}{A X_2 z} - A X_1 z. \label{spectralcurve2}
\end{align}
\end{example}

\subsection{The toric surface assigned to a Newton polygon} \la{sec:toricvariety}

In this section, we collect some notation regarding toric varieties, and refer the reader to the {Appendices \ref{toricv} and \ref{A2}} for more details. A convex integral polygon $N \subset {\rm M}_\R$ determines a compactification $X_N$ of the complex torus ${\rm T}$ called a {\it toric surface}, and  a  divisor $D_N$ supported on the boundary $X_N- {\rm T}$,  
 so that Laurent polynomials with Newton polygon $N$ extend naturally to sections of the  coherent sheaf $\mathcal O_{X_N}(D_N)$ {(for background on the coherent sheaf associated to a divisor, see for example \cite{CLS}*{Chapter 4}).}  

Denote by $|D_N|$ the projective 
 space of  non-zero global sections of the  coherent sheaf $\mathcal O_{X_N}(D_N)$, considered modulo a multiplicative constant. Assigning to a section its vanishing locus, we {see elements} of 
$|D_N|$ {as} curves in $X_N$ whose restrictions to ${\rm T}$ are defined by Laurent polynomials with Newton polygon contained in $N$. 

The genus $g$ of the  generic curve in $|D_N|$ is equal to the number of interior lattice points in $N$. {Recall that the edges $\{E_\rho\}$ of $N$ are in bijection with the rays $\{\rho\}$ of $\Sigma$}. Each edge $E_\rho$ of $N$ determines a projective line $D_\rho$ {which we call a \textit{line at infinity}} of $X_N$, and
$$
X_N - {\rm T} = \bigcup_{\rho \in \Sigma(1)}D_\rho.
$$
The divisor $D_N$ is given by 
\be \la{DN}
D_N = \sum_{\rho \in \Sigma(1)} a_\rho D_\rho,
\ee
  where $a_\rho \in \Z$ are such that 
\be \la{eq:Npoly}
N=\bigcap_{\rho \in \Sigma(1)}\{m \in {\rm M}_\R: \langle m , u_\rho \rangle \geq -a_\rho\}.
\ee
 The lines $D_\rho$ intersect according to the combinatorics of $N$: {precisely, for $\rho_1,\rho_2 \in \Sigma(1)$, the intersection $D_{\rho_1} \cap D_{\rho_2}$ is empty if $E_{\rho_1} \cap E_{\rho_2}$ is empty and a point if $E_{\rho_1} \cap E_{\rho_2}$ is a vertex of $N$.} The intersection index of a  generic curve in $|D_N|$ with the line $D_\rho$   is equal to  the number $|E_\rho|$ 
 of primitive integral vectors in the edge $E_\rho$. The points of intersection are called \textit{points at infinity}.
 Let $\spectralcurve \in |D_N|$ denote the compactification of the open spectral curve $\spectralcurve^\circ$, i.e., $\spectralcurve$ is the closure of $\spectralcurve^\circ$ in $X_N$. $\spectralcurve$ is called the \textit{spectral curve}.
 
 \subsection{Casimirs}\la{sec:cas}
 Let $\alpha$ be a zig-zag path $\alpha= \bw_1 \ra \w_1 \ra \bw_2 \ra \cdots \ra \w_d \ra \bw_1$ in $Z_\rho$. We define the \emph{Casimir} $C_\alpha$ by 
\[
C_\alpha:=(-1)^d [\epsilon]([\alpha]) [wt]([\alpha]).
\]

 The Casimirs determine points at infinity of $\spectralcurve$ as follows: since $[\alpha]$ is primitive and 
 $\langle u_\rho,[\alpha] \rangle=0$, we can extend it to a basis $(x_1,x_2)$ of $\rm M$ with $[\alpha]=x_1$ and $\langle x_2,u_\rho \rangle =1$. The affine open variety in $X_N$ corresponding to the cone $\rho$ is 
 \[
 U_\rho=\spec \C[x_1^{\pm 1 },x_2] \cong \C^\times \times \C,
 \]
 and $D_\rho \cap U_\rho$ is defined by $x_2=0$, and so the character $x_1^{-1}=\chi^{-[\alpha]}$ is a coordinate on the dense open torus $\C^\times=D_\rho \cap U_\rho$ in $D_\rho$. Therefore, the equation
\be \la{eq:casca}
\chi^{-[\alpha]} (\nu_\rho(\alpha))= C_\alpha,
\ee
defines a point $\nu_\rho(\alpha)$ in $D_{\rho}$. In other words, the point is defined as the unique point on the line at infinity such that the monomial $z^iw^j$, where $-[\alpha]=(i,j)$, evaluates to $C_\alpha$. We will prove later (see (\ref{caspoint})) that these are precisely the points at infinity of $\spectralcurve$.  

\begin{example}
Consider the fundamental domain of the square lattice, whose zig-zag paths were listed in Example \ref{eg:zz} and Figure \ref{fig:npzz}. The Casimirs are
\begin{align}\la{cassq}
    C_\alpha  = -\frac B {A X_1}, \ \ \ 
    C_\beta = -\frac{1}{A B X_2 }, \ \ \ 
    C_\gamma  = -\frac{A X_1X_2X_3}{B}, \ \ \ 
    C_\delta = -\frac{AB}{X_3}.\end{align}
    Let us denote the normal ray in $\Sigma$ of a zig-zag path $\omega$ by $\rho(\omega)$, so $u_{\rho(\alpha)}=(-1,-1)$ etc. We choose $x_2=\chi^{(0,-1)}$ so that $\langle (0,-1), u_{\rho(\alpha)} \rangle$=1. Then we have $U_{\rho(\alpha)}=\spec\C[x_1=z^{-1}w,x_2=w^{-1}]$ and $D_{\rho(\alpha)} \subset U_{\rho(\alpha)}$ is given by $x_2=0$. 
In this case, $D_N=D_{\rho(\alpha)}+D_{\rho(\beta)}+D_{\rho(\gamma)}+D_{\rho(\delta)}$ and $P(z,w)$ is a global section of $\mathcal O_{X_N}(D_N)$. We trivialize $\mathcal O_{X_N}(D_N)$ over $U_{\rho(\alpha)}$ as follows:
\begin{align*}
    \restr{\mathcal O_{X_N}(D_N)}{U_{\rho(\alpha)}}=\{t \in \C[z^{\pm 1},w^{\pm 1}]: \restr{\div~t}{U_{\rho(\alpha)}}+D_{\rho(\alpha)} \geq 0\} &\cong \mathcal O_{U_{\rho(\alpha)}}\\
    t &\mapsto {tx_2}
\end{align*}
Then making the change of variables $z=\frac{1}{x_1x_2}$ and $w=\frac{1}{x_2}$, and multiplying by $x_2$, the portion of the spectral curve $\spectralcurve$ in $U_\rho$ is cut out by
\[
\left(1  + X_1 + \frac{1}{X_2} + X_1 X_3\right)x_2- {B}  - \frac{X_1 X_3}{B }x_2^2 - \frac{x_1x_2^2}{A X_2 } -   \frac{A X_1}{x_1},
\]
so that $\spectralcurve \cap D_{\rho(\alpha)}$ is given by 
\[
-B-   \frac{A X_1}{x_1}=0.
\]
Therefore, $\nu(\alpha)$ is given by $\frac{z}{w}=\frac{1}{x_1}=C_\alpha$, which agrees with (\ref{eq:casca}). The table below lists the points at infinity for each of the zig-zag paths.
{
\be \def\arraystretch{1.5}
\begin{array}{|cccc|}
\hline
 \text{Zig-zag path} & \text{Homology class} & \text{Basis $x_1,x_2$} & \text{Point at infinity}\\
 \hline
\alpha & (-1,1) & (-1,1),(0,-1) &x_1=\frac{1}{C_\alpha},x_2=0  \\ [0.5ex]
\hline
\beta & (-1,-1) & (-1,-1),(0,-1) &x_1=\frac{1}{C_\beta},x_2=0 \\[0.5ex]
\hline
\gamma & (1,-1)& (1,-1),(0,1)&x_1=\frac{1}{C_\gamma},x_2=0  \\[0.5ex]
\hline
\delta & (1,1) & (1,1),(0,1)&x_1=\frac{1}{C_\delta},x_2=0  \\[0.5ex]
\hline
\end{array}\la{zzpathtable}
\ee
}
\end{example}

\subsection{The spectral transform} 

Our next goal is to define the spectral transform, which plays the key role in this paper. We present two equivalent definitions of the spectral transform. The first is the original definition of Kenyon and Okounkov \cite{KO}, and it is the one which we use in computations. However, it depends on the choice of the distinguished white vertex ${\bf w}$. The second is  more  invariant, and does not require choosing a  distinguished {white} vertex ${\bf w}$.

Recall that {for each edge $E_\rho$ of $N$, we have} $\# Z_\rho = \# \spectralcurve \cap D_\rho$, but there is no canonical bijection between these sets. 
We define a \textit{parameterization of the points at infinity by zig-zag paths} to be a {choice} of bijections $\nu=\{\nu_\rho\}_{\rho \in \Sigma(1)}$, where 
\be \la{ZA}
\nu_\rho : Z_\rho \xrightarrow[]{\sim} \spectralcurve\cap D_\rho.
\ee

For a curve $\spectralcurve \in|D_N|$, we denote by $\operatorname{Div}_\infty(\spectralcurve)$ the abelian group of divisors on $\spectralcurve$ supported at the \textit{infinity}, that is at  $\spectralcurve \cap D_N$.

Compactifications of the Kasteleyn operator will play a important role in this paper.
 The main ingredient in the construction of these compactifications is a combinatorial object called the \textit{discrete Abel map} introduced by Fock \cite{Fock} that encodes intersections with zig-zag paths. Let $\Gamma$ be a minimal bipartite graph in $\T$ with Newton polygon $N$ and spectral curve $\spectralcurve$. The {discrete Abel map}    
\begin{align*}
{\bf d}:B \cup W \cup F  &\ra \operatorname{Div}_\infty(\spectralcurve)
\end{align*}
assigns to each vertex and face of $\Gamma$ a divisor at infinity. 
It is defined uniquely up to a constant by the requirement that for a path $\gamma$ from $x$ to $y$, contained in  the fundamental domain $R$,  where $x$ and $y$ are either vertices or faces of $\Gamma$, we have
\begin{align*}
{\bf d}(y)-{\bf d}(x)&=\sum_{\rho \in \Sigma(1)}\sum_{\alpha \in Z_\rho} (\alpha,\gamma)_R \nu_\rho(\alpha).
\end{align*}
Here $(\alpha,\gamma)_R$ is the intersection index in $R$, i.e., the signed number of intersections of $\alpha$ with $\gamma$. Since we require $\gamma$ to be contained in $R$, this is well-defined, independent of the choice of path $\gamma$. Locally, the rule is as follows: 
\begin{enumerate}
    \item If $\bw$ is a black vertex incident to a face $f$, and $\bw$ and $f$ are separated by $\alpha \in Z_\rho$, then ${\bf d}(\bw)={\bf d}(f)+\nu_\rho(\alpha)$.
    \item If $\w$ is a white vertex incident to a face $f$, and $\w$ and $f$ are separated by $\alpha \in Z_\rho$, then ${\bf d}(\w)={\bf d}(f)-\nu_\rho(\alpha)$.
\end{enumerate}

We normalize ${\bf d}$, setting the value of ${\bf d}$ at certain face $f_0$ of $\Gamma$ to be $0$. Then for any black vertex $\bw$, face $f$, and white vertex $\w$ of $\widetilde \Gamma$ we have: 
\be
{\rm deg} ~{\bf d}(\bw) =1, \ \ {\rm deg}~ {\bf d}({f}) =0, \ \ {\rm deg}~ {\bf d}(\w) =-1. 
\ee
\begin{remark}
 Only differences of the form ${\bf d}(y)-{\bf d}(x)$ will appear in our constructions later, so the choice of normalization does not play a role. 
\end{remark}
\begin{example}\la{eg:dam}
Let us compute the discrete Abel map ${\bf d}$ for the square lattice in Figure \ref{figds}. We normalize ${\bf d}(f_1)=0$. Then we have
\[
{\bf d}(\bw_1)=\nu_{\rho(\gamma)}(\gamma), \quad {\bf d}(\bw_2)=\nu_{\rho(\alpha)}(\alpha),\quad {\bf d}(\w_1)=-\nu_{\rho(\beta)}(\beta),\quad {\bf d}(\w_2)=-\nu_{\rho(\delta)}(\delta),
\]
where $\nu$ is shown in table (\ref{zzpathtable}).
\end{example}
\vskip 2mm
{\bf Definition 1.}  A \textit{divisor spectral data} related to a Newton polygon $N$ is a 
triple $(\spectralcurve,S,\nu)$ where $\spectralcurve \in |D_N|$ is a genus $g$ curve on the toric surface $X_N$, $S$ is a degree $g$ effective divisor in $\spectralcurve^\circ$, and $\nu=\{\nu_\rho\}$ are parameterizations of the divisors  $D_\rho \cap \spectralcurve$, see (\ref{ZA}). Denote by  $\mathcal S_N$  the moduli space parameterizing the divisor spectral data on $N$. 
Let us  fix a {\it distinguished white vertex} ${\bf w}$ of $\Gamma$. Then there is a rational map (here {and in the sequel}, $\dashrightarrow$ means a rational map), called the \textit{spectral transform}, defined by Kenyon and Okounkov \cite{KO}, 
\begin{align} 
    \kappa_{\Gamma, {\bf w}}:  
      \mathcal X_N \dashrightarrow \mathcal S_N \la{SM}
\end{align}
defined on the dense open subset $H^1(\Gamma,\C^\times)$ of $\mathcal X_N$ by $[wt] \mapsto (\spectralcurve,S,\nu)$ as follows:
\begin{enumerate}
    \item $\spectralcurve$ is the spectral curve.
    \item For generic $[wt]$, $\spectralcurve$ is a smooth curve and $\text{coker } K$ is the pushforward of a line bundle on $\spectralcurve^\circ$. Let $s_{\bf w}$ be the section of $\text{coker } K$ given by the ${\bf w}$-entry of the cokernel map. $S$ is defined to be the divisor of this section. In Corollary \ref{cordeg}, we show that $S$ has degree $g$.   Then $S$ is the set of $g$ points in $\spectralcurve^\circ$ where the ${\bf w}$-column of the adjugate matrix $Q={Q(z,w)=}K^{-1} \text{det}K$ vanishes. 
    \item The parameterization of points at infinity by zig-zag paths $\nu$ is defined as follows:  $\nu_\rho(\alpha)$ is the point in $\spectralcurve\cap D_\rho$ satisfying $\chi^{-[\alpha]}= C_\alpha$ (see Section \ref{sec:cas}). We call $\nu_\rho(\alpha)$ the \textit{point at infinity associated to} $\alpha$. 
    \end{enumerate}     
\vskip 2mm 

{\bf Definition 2.}     
A \textit{line bundle spectral data} related to a Newton polygon $N$ is a 
triple $(\spectralcurve,{\cal L},\nu)$ where $\spectralcurve\in |D_N|$ is a genus $g$ curve on the toric surface $X_N$, ${\cal L}$ is a degree $g-1$ line bundle on $\spectralcurve$, and $\nu$ is a parameterization of points at infinity by zig-zag paths. Denote by  $\mathcal S'_N$  the moduli space parameterizing the line bundle spectral data on $N$. 

The spectral transform is a rational map 
\begin{align*}
	\kappa_{\Gamma, {\bf d}}:  
	\mathcal X_N & \dashrightarrow \mathcal S'_N
\end{align*}
defined on the dense open subset $H^1(\Gamma,\C^\times)$ of $\mathcal X_N$ by $[wt] \mapsto (\spectralcurve,\mathcal L,\nu)$, where:
\begin{enumerate}
	\item $\spectralcurve$ is the spectral curve.
	\item Let $\restr{K}{\spectralcurve^\circ}$ denote the restriction of the Kasteleyn matrix to $\spectralcurve^\circ$. The discrete Abel map  ${\bf d}$ determines an extension $\overline K$ of $\restr{K}{\spectralcurve^\circ}$ to a morphism of locally free sheaves on $\spectralcurve$; see Section~\ref{extensionsection}. The coherent sheaf $\mathcal L$ is   defined as the cokernel of $\overline K$. When $\spectralcurve$ is a smooth curve, which happens for generic $[wt]$, $\mathcal  L$ is a line bundle.  The convention deg ${\bf d}(\w)=-1$ implies that $\text{deg}~\mathcal L=g-1$; see  Proposition~\ref{degl}. 
	
	\item The parameterizations of the divisors $D_N \cap \spectralcurve$ are defined by associating to a zig-zag path $\alpha$ the point at infinity $\nu_\rho(\alpha)$.
\end{enumerate}

Since $\rho$ is determined by $\alpha$, we will use the simpler notation $\nu(\alpha){:=\nu_{\rho}(\alpha)}$ hereafter.  
 
The two types of spectral data are equivalent. Given  {a degree $g$ effective divisor} $S$, we have (Proposition~\ref{degl}) 
\be
\mathcal L \cong \mathcal O_{\spectralcurve}\left( S+{\bf d}({\bf w}) \right).
\ee
On the other hand, given a line bundle $\mathcal L$ and a white vertex ${\bf w}$, we can recover $S$ as follows. Consider the Abel-Jacobi map
\begin{align*}
A^g:\text{Sym}^g \spectralcurve &\ra \text{Jac}(\spectralcurve),\\
E &\mapsto \mathcal L \otimes \mathcal O_\spectralcurve(E + {\bf d}({\bf w})).
\end{align*} 
Then $A^g$ is birational by the Abel-Jacobi theorem \cite{Beau}*{Corollary~4.6}. We obtain $S=(A^g)^{-1}(\mathcal O_\spectralcurve)$.

\begin{example}
We compute the spectral transform for our running example of the square lattice. Let us take the distinguished white vertex to be ${\bf w}=\w_1$.
\begin{align}
Q&=\begin{blockarray}{ccc}
\w_1& \w_2 \\
\begin{block}{(cc)c}
 X_1-\frac{1}{A X_2 z}  & -1+\frac{X_1 X_3}{B w} & \bw_1\\
  1-Bw & 1-A z & \bw_2\\
\end{block}
\end{blockarray}.\la{qbwds}
\end{align}
Solving $Q_{{\bw_1} {\bf w}}{(p,q)}=Q_{{\bw_2} {\bf w}}{(p,q)}=0,$ we get 
\be \la{pqds}
p=\frac{1}{A X_1 X_2}, \quad  q=\frac 1 B.
\ee
Therefore, the spectral transform is:
\begin{align*}
\kappa_{\Gamma,{\bf w}}:H^1(\Gamma,\C^\times) &\dashrightarrow \mathcal S_N\\
(X_1,X_2,X_3,A,B) &\mapsto (\spectralcurve,(p,q),\nu),
\end{align*}
where $\spectralcurve=\{P(z,w)=0\}$ with $P(z,w)$ as is in (\ref{spectralcurve2}), $S=(p,q)$ is a single point ({the genus $g=1$ since $N$ in Figure~\ref{fig:npzz} has one interior lattice point}) and $\nu$ is as shown in table (\ref{zzpathtable}).
\end{example}

\section{The main theorem} \la{sec2}

Below we introduce functions ${\rm V}_{\bw \w}$ on the moduli space $ {\cal S}_{N}$ of spectral data, relying on results in the remaining Sections \ref{smallpolysection}, \ref{extensionsection}, \ref{laurentsection}. They are defined for any pair {$(\bw,\w) \in B \times W$} of black and white vertices, and defined as the solution to a system of linear equations ${\mathbb V}_{\bw \w}$. 

The main result of the paper is the following.
\begin{theorem}\label{Vbwthm}
For the distinguished {white} vertex ${\bf w}$, the pull-back of the function ${\rm V}_{{\bw}{\bf w}}$ under the spectral map coincides, up to a  multiplicative  constant, with the $\bw {\bf w}$ matrix element 
 ${Q}_{\bw {\bf w}}$ of the adjugate matrix 
 ${Q} :=K^{-1}\det K $ of the Kasteleyn matrix $K$. 
 That is, 
 \be
 {Q}_{\bw {\bf w}} = c \cdot  \kappa_{\Gamma,{\bf w}}^*({\rm V}_{\bw {\bf w}}),
 \ee
 {where $c$ depends on $\bw$ (and ${\bf w}$).}
 \end{theorem}
 
As an application of this result, we get an explicit description of the inverse to the spectral map (\ref{SM}); see Section~\ref{Sec2.3}. 

The next few sections discuss the structure of the system of linear equations $\mathbb V_{\bw \w}$. 
Detailed examples are given in Section~\ref{sec:example}. 

\subsection{The matrix \texorpdfstring{$\mathbb V_{\bw \w}$}{Vbw}} 

The system of linear equations $\mathbb V_{\bw \w}$ is in the variables $ (a_m)_{m \in N_{\bw \w } \cap {\rm M}}$ where $N_{\bw \w} \subset {\rm M}_\R$ is a convex polygon, introduced in Section~\ref{S2.1.1}, and called the \textit{small Newton polygon}. {There is one system for every pair $(\bw,\w) \in B \times W$}. {The system $\mathbb V_{\bw \w}$ is of the form $(\text{matrix}) (a_m)=0$; we also denote this matrix by $\mathbb V_{\bw \w}$.} Therefore, the columns of the matrix $\mathbb{V}_{\bw \w}$ are indexed by the lattice points $N_{\bw \w } \cap {\rm M}$. By Corollary~\ref{Cor3.5},  
the polygon $N_{\bw \w}$ is the Newton polygon of the Laurent polynomial ${Q}_{\bw \w}$.

The equations in $\mathbb V_{\bw \w}$, i.e., the rows of the matrix $\mathbb V_{\bw \w}$ are defined in Section~\ref{S2.1.2}. There are two types:
\begin{enumerate}
     \item There is a row for each of the points $(p_1,q_1),\dots, (p_g,q_g)$ of the divisor $S$ on the spectral curve. The entry of the row in column $m \in N_{\bw \w } \cap {\rm M}$ is $\chi^m(p_i,q_i)$. 
    \item The remaining rows correspond to  certain zig-zag paths $\alpha$. The entries in the row corresponding to $\alpha$ are certain monomials in $C_\alpha$.  
 \end{enumerate}

Let us proceed to the precise definition of the matrix $\mathbb V_{\bw\w}$.

\subsubsection{Columns of the matrix \texorpdfstring{$\mathbb V_{\bw \w}$}{Vbw}} 

We now describe the small Newton polygons, whose lattice points correspond to columns of $\mathbb V_{\bw \w}$.

\paragraph{Rational Abel map \texorpdfstring{$\dd$}{dd}.}

{ {Recall the set $\{D_{\rho}\}$ of lines at infinity of the toric surface~$X_N$.}} Consider the $\Q$-vector space $\operatorname{Div}_{\rm T}^{\Q}(X_N)$ of \textit{$\Q$-divisors at infinity}, defined as the $\Q$-vector space with a basis given by the divisors $D_{\rho}$:
$$
\operatorname{Div}_{\rm T}^{\Q}(X_N) := \bigoplus_{\rho \in \Sigma(1)}\Q{D_\rho}. 
$$

 We define a \textit{rational Abel map}
\[
 {\dd}:V \ra \operatorname{Div}_{\rm T}^{\Q}(X_N)
\]
 which  assigns to each vertex $\rm v$  of 
the graph $\Gamma$ a $\Q$-divisor at infinity $\dd(\rm v)$ as follows:
 \begin{enumerate}
    \item  Normalize ${\dd}({\bf w})=0$. As in the case of ${\bf d}$, the choice of normalization plays no role, and we can replace $0$ with any $\Q$-divisor.
    \item For any path $\gamma$ contained in $R$ from ${\rm v}_1$ to ${\rm v}_2$,
\begin{align*}
{\dd}({\rm v}_2)-{\dd}({\rm v}_1)&=\sum_{\rho \in \Sigma(1)}\sum_{\alpha \in Z_\rho} {\frac{(\alpha,\gamma)_R}{|E_\rho|}} D_\rho, 
\end{align*}
where $(\cdot,\cdot)_R$ is the intersection index in $R$, i.e., the signed number of intersections of $\alpha$ with~$\gamma$. 
\end{enumerate}
The following lemma follows from definitions.
\begin{lemma}\la{lem:DAM}
Let $\alpha, \beta$ be the zig-zag paths through ${e=\bw \w}$, with $\alpha \in Z_\sigma, \beta \in Z_\rho$. Then, we have 
\be \la{DAM}
{\dd}(\w)-\dd(\bw)=-\frac{1}{|E_{{\sigma}}|}D_{{\sigma}}-\frac{1}{|E_{{\rho}}|}D_{{\rho}}-\div  \phi(e) 
\ee {where $\phi(e)$ is {the character} defined in (\ref{edgeph}) {and $\div \phi(e)$ denotes its (Weil) divisor as in (\ref{divchar}).}}
\end{lemma}

 \paragraph{Small Newton polygons.}\la{S2.1.1} 

Recall the divisor $D_N$ at infinity of $X_N$, see (\ref{DN}). Given an edge $e=\bw \w$,  we define a $\Q$-divisor at infinity  
\be \la{DE}
\divebw_{\bw \w}:=D_N-\dd(\w)+\dd(\bw)-\sum_{\rho \in \Sigma(1)}\sum_{\alpha \in Z_\rho : \bw \in \alpha } \frac{1}{|E_{\rho}|}D_{\rho}. 
\ee
Here the double sum is over all zig-zag paths $\alpha$ passing through $\bw$ and $|E_\rho|$ denotes the number of integral primitive vectors in $E_\rho$ as in Section~\ref{sec:toricvariety}. We define $b_\rho \in \Q$ as the multiplicities of the projective lines at infinity $D_\rho$ in the divisor $\divebw_{\bw \w} $: 
\be \la{br}
\divebw_{\bw \w} =  \sum_{\rho \in \Sigma(1)} b_\rho D_\rho.
\ee

 
\bd \la{SNP1} The {\it small Newton polygon $N_{\bw \w}$} is the polygon defined by the formula
\be\label{nbwdefinition}
N_{\bw \w}= \bigcap_{\rho \in \Sigma(1)} \{m \in {\rm M}_\R : \langle m, u_\rho \rangle \geq -b_\rho\}.
\ee
{There is a canonical bijection between divisors $D$ in $\operatorname{Div}_{\rm T}^{\Q}(X_N)$ and convex polygons {$P$} with rational intercepts (see Proposition \ref{pro:globalsec} for its importance in toric geometry): 
\be\la{equivalence}
D=\sum_{\rho \in \Sigma(1) }a_\rho D_\rho ~~ \xleftrightarrow[]{} ~~ P=\bigcap_{\rho \in \Sigma(1)}\{m \in {\rm M}_\R: \langle m , u_\rho \rangle \geq -a_\rho\}, \qquad a_\rho \in \Q.
\ee
Therefore,} $N_{\bw \w}$ is the polygon associated to the divisor $\divebw_{\bw \w}$ in (\ref{equivalence}). 
\ed
The polygon $N_{\bw \w}$ may  not be integral. We will consider only integral points in it. The convex hull of the integral points in $N_{\bw \w}$ contains the Newton polygon of $Q_{\bw \w}$ (Corollary \ref{Cor3.5}).

\begin{example}\la{eg:smallply1}
We compute the small polygons for the square lattice in Figure \ref{figds}. Recall that we chose ${\bf w}=\w_1$. Since there is only one zig-zag path in each homology direction, the rational Abel map ${\bf D}$ is obtained from ${\bf d}$ by replacing the point at infinity with the corresponding line at infinity, so from Example \ref{eg:dam}, we have
\[
{\bf D}(\bw_1)=D_{\rho(\gamma)}, \quad {\bf D}(\bw_2)=D_{\rho(\alpha)},\quad {\bf D}(\w_1)=-D_{\rho(\beta)},\quad {\bf D}(\w_2)=-D_{\rho(\delta)}.
\]
We have $D_N=D_{\rho(\alpha)}+D_{\rho(\beta)}+D_{\rho(\gamma)}+D_{\rho(\delta)}$, using which we compute
\begin{align*}
\divebw_{\bw_1 \w_1}&= (D_{\rho(\alpha)}+D_{\rho(\beta)}+D_{\rho(\gamma)}+D_{\rho(\delta)}) -(-D_{\rho(\beta)})+D_{\rho(\gamma)}-(D_{\rho(\alpha)}+D_{\rho(\beta)}+D_{\rho(\gamma)}+D_{\rho(\delta)})\\
&= D_{\rho(\beta)}+D_{\rho(\gamma)},\\
\divebw_{\bw_2 \w_1}&= (D_{\rho(\alpha)}+D_{\rho(\beta)}+D_{\rho(\gamma)}+D_{\rho(\delta)}) -(-D_{\rho(\beta)})+D_{\rho(\alpha)}-(D_{\rho(\alpha)}+D_{\rho(\beta)}+D_{\rho(\gamma)}+D_{\rho(\delta)})\\
&= D_{\rho(\alpha)}+D_{\rho(\beta)}.
\end{align*}
 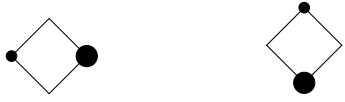
\begin{figure}[t!]
	\centering
	\begin{tikzpicture}
		\draw[fill=black] (0,0) circle (4pt);
		
		\draw[fill=black] (-1,0) circle (2pt);
		
		\draw[] (0,0) -- (-0.5,0.5)--(-1,0) --(-0.5,-0.5)--(0,0);
		
	\end{tikzpicture} \hspace{20mm}
	\begin{tikzpicture}
		\draw[fill=black] (0,0) circle (4pt);
		
		\draw[fill=black] (0,1) circle (2pt);

		\draw[] (0,0) -- (0.5,0.5)--(0,1) --(-0.5,0.5)--(0,0);
		
	\end{tikzpicture}
	\caption{The two small polygons in Example \ref{eg:smallply1}. The big black dot denotes the origin, while the other black dots are integral points.}
	\label{fig:smallpoly1}
\end{figure}
Therefore,
\begin{align*}
    N_{\bw_1 \w_1}&= \{-i-j \geq 0\} \cap \{i-j\geq -1\} \cap \{i+j \geq -1\} \cap \{-i+j \geq 0\},\\
    N_{\bw_2 \w_1}&= \{-i-j \geq -1\} \cap \{i-j\geq -1\} \cap \{i+j \geq 0\} \cap \{-i+j \geq 0\},
\end{align*}
see Figure \ref{fig:smallpoly1}. Note that the convex hulls of the lattice points are the Newton polygons of $Q_{\bw_1 \w_1}$ and $Q_{\bw_2\w_1}$ in (\ref{qbwds}). 
\end{example}

\subsubsection{Rows of the matrix \texorpdfstring{$\mathbb V_{\bw \w}$}{Vbw}}  \la{S2.1.2}

Recall that the variables in $\mathbb V_{\bw \w}$ are $ (a_m)_{m \in N_{\bw {\bf w} } \cap {\rm M}}$. {We identify a Laurent polynomial $F = \sum_{m \in {\rm M}} b_m \chi^m$ with its vector of coefficients $(b_m)_{m \in {\rm M}}$.} The equations in $\mathbb V_{\bw \w}$ are of two types:
\begin{enumerate}
    \item For each $1 \leq i \leq g$, we have the linear equations 
  \be \la{eq1}
    \sum_{m \in N_{\bw {\w} } \cap {\rm M}} a_m \chi^m(p_i,q_i) =0, 
  \ee
  so the entry of the corresponding row of $\mathbb V_{\bw \w}$ in column $m$ is $\chi^m(p_i,q_i)$.
    \item Recall the notation $\floor{x}$ for the largest integer $n$ such that $n \leq x$.   
    
    Given  a $\Q$-divisor $D = \sum_{\rho \in \Sigma(1)} b_\rho D_\rho$, we define a divisor with integral coefficients
  \[
    \floor{D}:=\sum_{\rho \in \Sigma(1)} \floor{b_\rho} D_\rho. 
 \]
Recall the divisor $\divebw_{{\bw}{\w}}$ in (\ref{DE}). {For a divisor $D$ at infinity, let $\restr{D}{\spectralcurve}$ denote the divisor corresponding to the intersection of $D$ with $\spectralcurve$. Precisely, if $D= \sum_{\rho \in \Sigma(1)} a_\rho D_\rho$, then $\restr{D}{\spectralcurve} := \sum_{\rho \in \Sigma(1)} a_\rho \sum_{\alpha \in Z_\rho} \nu(\alpha)$}. We have a linear equation for every zig-zag path $\alpha$ such that $\nu(\alpha)$ appears in 
    \be \la{contpointsatinf}
    -\restr{D_N}{\spectralcurve}+{\bf d}({\w})-{\bf d}(\bw)+\sum_{\alpha \in Z} \nu(\alpha) + \restr{\floor{\divebw_{{\bw }{\w}}}}{\spectralcurve}.
    \ee

    Suppose $\alpha \in Z_\rho$ is a zig-zag path that contributes an equation. We extend $[\alpha]$ to a basis $(x_1,x_2)$ of $\rm M$, where $x_1:= [\alpha]$ and $\langle x_2,u_\rho \rangle =1$, so that for any $m \in {\rm M}$, we can write
\[
    \chi^m =  x_1^{b_m} x_2^{c_m}, ~~~~b_m,c_m \in \Z.
    \]
    Let $N_{\bw {\w}}^\rho$ be the set of lattice points in $N_{\bw {\w}}$ closest 
to the edge $E_{\rho}$ of $N$ i.e., the set of points in $N_{\bw {\w}}$ that minimize the functional $\langle *,u_\rho \rangle$. 
   Then the equation associated with $\alpha$ is 
\be \la{Cas}
    \sum_{m \in N^\rho_{\bw {\w} } \cap {\rm M}}a_m C_\alpha^{-b_m}=0. 
\ee
So the entry in column ${m \in N^\rho_{\bw {\w} } \cap {\rm M}}$ is the monomial $C_\alpha^{-b_m}$, and the entries in the other columns are $0$. Choosing a different basis vector $x_2$ leads to the same equation multiplied by a monomial in $C_\alpha$.
\end{enumerate}

\begin{remark}\la{rem::prim}
	When all the sides of the Newton polygon are primitive, we call the Newton polygon \textit{simple}. In this case, we have $[\divebw_{\bw {\bf w}}]=\divebw_{\bw {\bf w}}$ and ${\bf d}({\bf w})-{\bf d}(\bw) = \restr{(\dd({\bf w})-\dd(\bw))}{\spectralcurve}$. Then Formula~(\ref{contpointsatinf}) simplifies considerably to
	\be \la{eq:simplepoly}
	\sum_{\alpha \in Z: \bw \notin \alpha} \nu(\alpha).
	\ee
	So for a simple Newton polygon the Casimir rows of the matrix $\mathbb V_{\bw \w}$, i.e., the rows providing equations (\ref{Cas}),  are parameterized by the zig-zag paths $\alpha$ which do not contain the vertex ${\rm b}$. 
\end{remark}

{

\subsubsection{The functions \texorpdfstring{${\rm V}_{\bw \w}$}{Vbw}} 

The number of rows of $\mathbb V_{\bw \w}$ is at least as large as the number of columns minus one, but not necessarily equal. However, Proposition~\ref{prop:uniq} shows that there is a unique solution to $\mathbb V_{\bw \w}$ up to a multiplicative constant. Therefore, 
\be \la{DV}
{\rm V}_{\bw \w}:=\sum_{m \in N_{\bw \w } \cap {\rm M}} a_m \chi^m,
\ee
is uniquely defined up to a multiplicative constant (where $(a_m)_{m \in N_{\bw \w } \cap {\rm M}}$ is a solution to $\mathbb V_{\bw \w}$). Only {ratios of the values of} these functions that are independent of the multiplicative constant appear in the inverse map, see Section~\ref{Sec2.3}.
}

\begin{remark} \la{remind}
 When the equations in $\mathbb V_{\bw \w}$ are linearly independent (so there is exactly one less equation than the number of variables), we can {prepend} to $\mathbb V_{\bw \w}$ the equation $\sum_{m \in N_{\bw \w}  \cap {\rm M}} a_m \chi^m$ to get a square matrix, which we denote by $\mathbb V_{\bw \w}^\chi$. Then the function ${\rm V}_{\bw \w}$ is the determinant:
 $$
{\rm V}_{\bw \w} =  \mathrm{det}~\mathbb V_{\bw \w}^\chi.
 $$ 
 Indeed, given an $(n-1)\times n$ matrix $(a_{ij})$, the system of linear equations $\sum^n_{j=1}a_{ij}x_j=0$ has a solution 
 given by the signed maximal minors $A_j$ of the matrix $A$: 
\[
 x_j = (-1)^jA_j.
\]
  Here $A_j$ is the determinant of the matrix obtained by deleting the $j$-th column of $A$.  Therefore, the determinant of the augmented matrix 
  $\mathbb V_{\bw \w}^\chi$ recovers the expression ${\rm V}_{\bw \w} $ in (\ref{DV}).
\end{remark}

\begin{example}
We compute the linear system of equations $\mathbb V_{\bw {\bf w}}$ for the square lattice in Figure \ref{figds} with ${\bf w}=\w_1$. Since both black vertices are contained in every zig-zag path, the formula (\ref{eq:simplepoly}) is $0$, so there are no equations of type 2 in $\mathbb V_{\bw {\bf w}}$ for $\bw \in B$. Therefore,
\[
{\mathbb V}_{\bw_1 {\bf w}}=\begin{pmatrix}
1& p^{-1}
\end{pmatrix}, \quad {\mathbb V}_{\bw_2 {\bf w}}=\begin{pmatrix}
1& q
\end{pmatrix}.
\]
By Remark \ref{remind}, we get
\be \la{v:eg}
{\rm V}_{\bw_1 {\bf w}}=\begin{vmatrix}
1 & z^{-1}\\
1& p^{-1}
\end{vmatrix}, \quad {\rm V}_{\bw_2 {\bf w}}=\begin{vmatrix}
1 & w\\
1& q
\end{vmatrix}.
\ee
Using (\ref{pqds}), we have 
\begin{align*}
\kappa_{\Gamma,{\bf w}}^*({\rm V}_{\bw_1 {\bf w}})&=A X_1 X_2-\frac{1}{z}=A X_2  Q_{\bw_1 {\bf w}},\\
\kappa_{\Gamma,{\bf w}}^*({\rm V}_{\bw_2 {\bf w}})&=\frac 1 B-w=\frac 1 B  Q_{{\bw_2} {\bf w}},
\end{align*}
verifying the conclusion of Theorem \ref{Vbwthm}.
\end{example}

\subsection{Reconstructing weights via functions \texorpdfstring{${\rm V}_{\bw {\bf w}}$}{vbw}.} \la{Sec2.3}

Take a white vertex $\w$ and a zig-zag path $\alpha$ containing $\w$. The pair $(\w, \alpha)$ determines a {\it wedge} $W:=\bw \xrightarrow[]{e} \w \xrightarrow[]{e'} \bw'$,
 where $\w$ is a white vertex incident to the vertices $\bw, \bw'$ such that $\bw \w \bw'$ is  a part of $\alpha$. Recall $\phi(e)$ from (\ref{edgeph}), and   the Kasteleyn sign $\epsilon(e)$.
  We assign to this wedge the ratio
\be \la{RV}
r_W:=-\frac{\epsilon(e')\phi(e'){\rm V}_{\bw' {\bf w}}}{\epsilon(e)\phi(e){\rm V}_{\bw {\bf w}}}(\nu(\alpha)).
\ee
Note that we use the distinguished white vertex ${\bf w}$ in the expression rather than ${\w}$. The expression is in fact independent of $\w$, as {we will see in the proof of Theorem~\ref{Th2.3} below.} 

{
\begin{remark}\la{remark:factor}
	 The ratio on the right is a rational function on the curve. We evaluate the ratio at the point at infinity of the spectral curve $\nu(\alpha)$ corresponding to the zig-zag path $\alpha$, see (\ref{ZA}). To do this, we first extend $[\alpha]$ to a basis $(x_1,x_2)$ of $\rm M$ with $[\alpha]=x_1$ and $\langle x_2,u_\rho \rangle =1$ , as explained in Section~\ref{sec:cas}. Then $\nu(\alpha)$ is given by $\frac{1}{x_1}=C_\alpha, x_2=0$. The numerator and denominator in (\ref{RV}) vanish to the same order in $x_2$ by Corollary~\ref{divQbw} below, so after factoring out and canceling the highest power of $x_2$ in the numerator and denominator, we can evaluate at ${x_1}=\frac{1}{C_\alpha}, x_2=0$ to get a well-defined number.
\end{remark}
}

Let $L=\bw_1 \to \w_1 \to \bw_2 \to \dots \to \bw_\ell=\bw_1$ be an oriented loop on $\Gamma$. It is a concatenation of
wedges $W_i:=\bw_{i-1} \w_i \bw_i, i=1,\dots,\ell$ (with $i$ taking values {cyclic modulo $\ell$}) provided by the white vertices. Denote by $\alpha_i$ the zig-zag path assigned to the wedge $W_i$.  We define a cohomology class $[\omega]$ by
\begin{equation}\label{mon}
\begin{split}
[\omega]([L]):=& \prod_{i =1}^\ell r_{W_i}.
\end{split}
\end{equation}

\bl
The product (\ref{mon}) does not depend on the ambiguities of {the multiplicative constants in} the involved functions ${\rm V}_{{{\rm b}{\bf w}}}$.
\el

\begin{proof} For each 
	black vertex $\bw_i$ in $L$, ${\rm V}_{\bw_i {\bf w}}$ appears twice in (\ref{mon}), once each in the numerator and denominator, and so the multiplicative constants cancel out.  
\end{proof}

 \bt \la{Th2.3} The cohomology classes $[wt]$ and $\kappa_{\Gamma,{\bf w}}^*[\omega]$ are equal.
\et 
{\begin{proof}
	Let $\bw \xrightarrow[]{e}\w \xrightarrow[]{e'} \bw'$ be a wedge with zig-zag path $\alpha \in Z_\rho$.  The restriction of the characteristic polynomial $\restr{P(z,w)}{D_\rho}$ is the partition function of those dimers whose homology class in $N$ lies on $E_\rho$. From the explicit construction of external dimers in \cite{GK12} (that is, dimers whose homology classes are in $\partial N$), we have that each dimer with homology class in $E_\rho$ uses exactly one of the edges $e$ or  $e'$. Since $Q_{\bw \w}(z,w)$ is the partition function of dimers with the vertices $\bw,\w$ removed, we have 
	\[
	\restr{P}{D_\rho}=wt (e) \epsilon{(e)} \phi(e) \restr{Q_{\bw \w}}{D_\rho}+wt(e')  \epsilon{(e')} \phi(e') \restr{Q_{\bw' \w}}{D_\rho}.
	\]
	Since $\nu(\alpha)$ is on the spectral curve, $P(\nu(\alpha))=0$, from which we get
	\be \la{eq:altprod}
	\frac{wt(e) }{wt(e') } =-\frac{\epsilon(e' )\phi(e')Q_{\bw' \w}}{\epsilon(e)\phi(e) Q_{\bw \w}}(\nu(\alpha)).
	\ee
	We have $\text{corank}(K)=1$ at smooth points of $\spectralcurve$. Note that $\restr{KQ}{\spectralcurve}=0$. Therefore, for generic $wt$, since $\spectralcurve$ is smooth, $Q$ is a rank $1$ matrix given by  
\[	Q=\ker K^* \otimes \coker K. \]
This implies that
\[
	\frac{Q_{\bw \w}}{Q_{\bw' \w}}(\nu(\alpha))=\frac{Q_{{\bw}{\bf w}}}{Q_{{\bw}'{\bf w}}}(\nu(\alpha)). \qedhere
\]
	
\end{proof}
}

\begin{example}
Consider the cycle $a$ in Figure~\ref{figds} {given by} {the red horizontal path}. We write it as the concatenation of the two wedges {$W_1$ and $W_2$ represented by $(\w_1,\delta)$ and $(\w_1,\gamma)$ respectively.} From Table~(\ref{zzpathtable}), we know that in the basis $x_1=z w,x_2=w$, the point $\nu(\delta)$ is given by $x_1=\frac{1}{C_\delta},x_2=0$. Using (\ref{v:eg}), and making the substitution $z=\frac{x_1}{x_2},w=x_2$, we get
\begin{align*}
r_{W_1}&=- \frac{-1\cdot w^{-1}\cdot V_{\bw_2 {\bf w}}}{-1 \cdot z \cdot  V_{\bw_1 {\bf w}}}(\nu(\delta))\\
&=-\frac{1}{zw} \frac{q-w}{p^{-1}-z^{-1}}(\nu(\delta))\\
&=\frac{-(q-x_2)}{x_1 p^{-1}-x_2}\left(\frac{1}{C_\delta},0\right)\\
&=-{pq}{C_\delta}.
\end{align*}
 Similarly, from table (\ref{zzpathtable}) we know that in the basis $x_1=\frac z w,x_2=w$, the point $\nu(\gamma)$ is given by $x_1=\frac{1}{C_\gamma},x_2=0$. Using (\ref{v:eg}), and making the substitution $z={x_1}{x_2},w=x_2$, we get
\begin{align*}
r_{W_2}&=- \frac{1\cdot 1\cdot V_{\bw_1 {\bf w}}}{-1 \cdot w^{-1} \cdot  V_{\bw_2 {\bf w}}}(\nu(\gamma))\\
&=w \frac{p^{-1}-z^{-1}}{q-w}(\nu(\gamma))\\
&=x_2 \frac{p^{-1}-\frac{1}{x_1 x_2}}{q-x_2}\left(\frac{1}{C_\gamma},0\right)\\
&=-\frac{C_\gamma}{q}.
\end{align*}
\end{example}
Therefore, $[\omega]([a])={p C_\gamma}{C_\delta}$, and using (\ref{cassq}) and (\ref{pqds}), we have 
\begin{align*}
\kappa_{\Gamma,{\bf w}}^*[\omega]([a])&=\left(\frac{1}{AX_1X_2} \right)\cdot \left(-\frac{AX_1 X_2 X_3}{B} \right)\cdot \left({-\frac{AB}{X_3}}\right)\\
&=A. 
\end{align*}

\section{Examples}\la{sec:example}

In this section, we work out two detailed examples.

\subsection{Primitive genus \texorpdfstring{$2$}{2} example}

\begin{figure}
\centering

	\begin{tikzpicture}  
	\begin{scope}[scale=1.5,shift={(0,3)}]
		\draw[dashed, gray] (0,0) --(1,1.732)--(5,1.732)--(6,0)--(5,-1.732)--(1,-1.732)--(0,0);

\coordinate[bvert] (b11) at (2*1+1-1,1.732*1-0.5773);
        \coordinate[wvert] (w11) at (2*1+1-2,1.732*1-1.732+0.5773);
        \coordinate[bvert] (b21) at (2*2+1-1,1.732*1-0.5773);
        \coordinate[wvert] (w21) at (2*2+1-2,1.732*1-1.732+0.5773);
        \coordinate[wvert] (w31) at (2*3+1-2,1.732*1-1.732+0.5773);
        
\coordinate[wvert] (w'11) at (2*1+1-1,-1.732*1+0.5773);
        \coordinate[bvert] (b'11) at (2*1+1-2,-0.5773);
        \coordinate[wvert] (w'21) at (2*2+1-1,-1.732*1+0.5773);
        \coordinate[bvert] (b'21) at (2*2+1-2,-1.732*1+1.732-0.5773);
        \coordinate[bvert] (b'31) at (2*3+1-2,-1.732*1+1.732-0.5773);   
        
    \coordinate[] (b01) at (1-1+0.5,1.732*1-1.732+0.866);  
        \coordinate[] (b41) at (6-0.5,1.732*1-1.732+0.866);
    \coordinate[] (w'01) at (1-1+0.5,-1.732*1+1.732-0.866); 
            \coordinate[] (w'41) at (6-0.5,-1.732*1+1.732-0.866);
    \draw[](b01)--(w11)--(b11)--(w21)--(b21)--(w31)--(b41);
        \draw[](w'01)--(b'11)--(w'11)--(b'21)--(w'21)--(b'31)--(w'41)
        (b'11)--(w11)
        (b'21)--(w21)
        (b'31)--(w31)
        
        (b11)--(2,1.732)
        (b21)--(4,1.732)
        
        (w'11)--(2,-1.732)
        (w'21)--(4,-1.732)
        
        ;
  \draw[red,very thick,->] (b01)--(b41);
    \draw[red,very thick,->] (w'01)--(w'41);
\draw[green, very thick,<-] (b01)--(2,-1.732);
\draw[green, very thick,<-](2,1.732)--(4,-1.732);
\draw[green, very thick,<-]
(4,1.732)--(w'41)
;
\draw[blue,very thick,->] (2,1.732)--(w'01);
\draw[blue,very thick,<-](2,-1.732)--(4,1.732);
\draw[blue,very thick,<-](4,-1.732)--(b41)
;  
        
\end{scope}
\begin{scope}[shift={(12,6)},scale=0.7]
\draw[text=black,red,very thick,->] (-1,-1)--node[black,below]{$[\gamma]$}(1,0);
\draw[text=black,green,very thick,->] (1,0)--node[black,right]{$[\alpha]$}(0,2);
\draw[text=black,blue,very thick,->] (0,2)--node[black,left]{$[\beta]$}(-1,-1);
	\draw[fill=black] (0,0) circle (4pt);
		
		\draw[fill=black] (0,1) circle (2pt);
		\draw[fill=black] (1,0) circle (2pt);
		\draw[fill=black] (0,2) circle (2pt);
			\draw[fill=black] (-1,-1) circle (2pt);
			\node[](no) at (0,-2){$N$};
			
\end{scope}	

\begin{scope}[shift={(12,2)},scale=0.7]

\draw[red,very thick,->] (0,0)--(-1,2);
\draw[green,very thick,->] (0,0)--(-2,-1);

\draw[blue,very thick,->] (0,0)--(3,-1);
	\draw[fill=black] (0,0) circle (2pt);
	\node[](no) at (1.5,0) {$\sigma_\beta$};
		\node[](no) at (-1.5,-0.3) {$\sigma_\alpha$};
			\node[](no) at (-0.3,1.5) {$\sigma_\gamma$};
				\node[](no) at (0,-1.5){$\Sigma$};
\end{scope}	
\end{tikzpicture}
\caption{A hexagonal graph, its Newton polygon {$N$} and normal fan {$\Sigma$}, with zig-zag paths and rays labeled.}\label{fig:hex}
\end{figure}
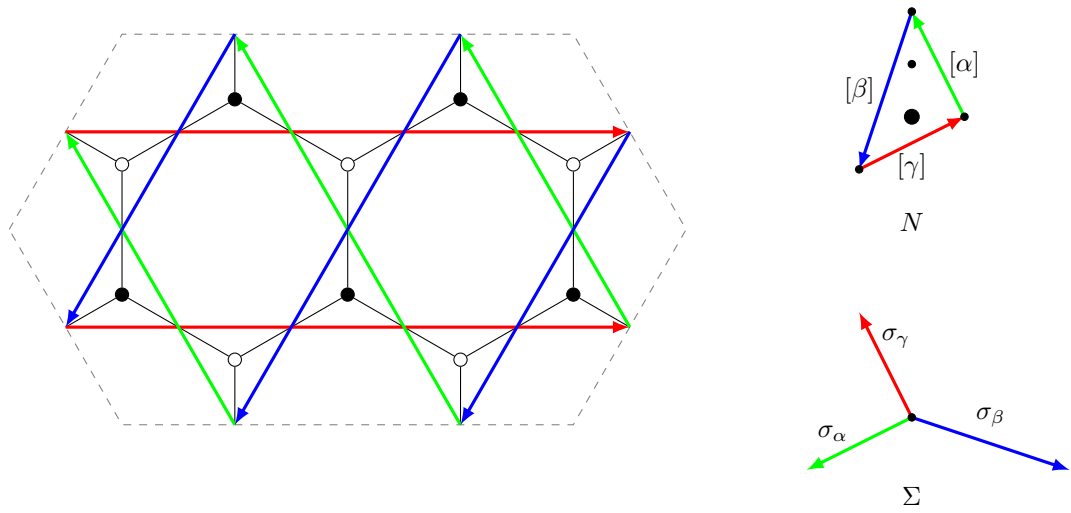

\begin{figure}
\centering

	\begin{tikzpicture}[scale = 1.5]  
		\draw[dashed, gray] (0,0) --(1,1.732)--(5,1.732)--(6,0)--(5,-1.732)--(1,-1.732)--(0,0);

\coordinate[bvert,label=below:${\bw_1}$] (b11) at (2*1+1-1,1.732*1-0.5773);
        \coordinate[wvert,label=right:${{\bf w}=\w_1}$] (w11) at (2*1+1-2,1.732*1-1.732+0.5773);
        \coordinate[bvert,label=below:${\bw_2}$] (b21) at (2*2+1-1,1.732*1-0.5773);
        \coordinate[wvert,label=above:${\w_2}$] (w21) at (2*2+1-2,1.732*1-1.732+0.5773);
        \coordinate[wvert,label=left:${\w_3}$] (w31) at (2*3+1-2,1.732*1-1.732+0.5773);
        
\coordinate[wvert,label=above:${\w_4}$] (w'11) at (2*1+1-1,-1.732*1+0.5773);
        \coordinate[bvert,label=below:${\bw_3}$] (b'11) at (2*1+1-2,-0.5773);
        \coordinate[wvert,label=above:${\w_5}$] (w'21) at (2*2+1-1,-1.732*1+0.5773);
        \coordinate[bvert,label=below:${\bw_4}$] (b'21) at (2*2+1-2,-1.732*1+1.732-0.5773);
        \coordinate[bvert,label=below:${\bw_5}$] (b'31) at (2*3+1-2,-1.732*1+1.732-0.5773);   
        
    \coordinate[] (b01) at (1-1+0.5,1.732*1-1.732+0.866);  
        \coordinate[] (b41) at (6-0.5,1.732*1-1.732+0.866);
    \coordinate[] (w'01) at (1-1+0.5,-1.732*1+1.732-0.866); 
            \coordinate[] (w'41) at (6-0.5,-1.732*1+1.732-0.866);
    \draw[](b01)--(w11)--(b11)--node[below]{$\frac{1}{X_2}$}(w21)--node[below]{$X_3$}(b21)--(w31)--node[above]{$\frac{1}{AB X_1 zw}$}(b41);
        \draw[](w'01)--(b'11)--(w'11)--(b'21)--(w'21)--(b'31)--node[below]{$Az$}(w'41)
        (b'11)--(w11)
        (b'21)--(w21)
        (b'31)--(w31)
        
        (b11)--node[left]{$X_1 X_4 B w$}(2,1.732)
        (b21)--node[left]{$B w$}(4,1.732)
        
        (w'11)--(2,-1.732)
        (w'21)--(4,-1.732)
        
        ;
        \def\lw{1.5};
\draw[red,very thick,->] (b01)--(w11);
\draw[red,very thick,->] (w11)--(b'11);
\draw[red,very thick,->](b'11)--(w'11);
\draw[red,very thick,->](w'11)--(b'21);
\draw[red,very thick,->](b'21)--(w'21);
\draw[red,very thick,->](w'21.45)--(b'31.195);
\draw[red,very thick,->](b'31)--(w'41);

\draw[green,very thick,->] (4,-1.732)--(w'21);
\draw[green,very thick,->] (w'21.20)--(b'31.220);
\draw[green,very thick,->] (b'31)--(w31);
\draw[green,very thick,->] (w31)--(b21);
\draw[green,very thick,->] (b21)--(4,1.732);

\node[](no) at (4,-1.732-0.5){$b$};
\node[](no) at (1-1+0.5-0.5,1.732*1-1.732+0.866){$a$};
\node[blue](no) at (0,0) {$f_1$};

\node[blue](no) at (2,0) {$f_2$};
\node[blue](no) at (4,0) {$f_3$};
\node[blue](no) at (6,0) {$f_4$};
\node[blue](no) at (3,1.732) {$f_5$};

\end{tikzpicture}
\caption{Labeling of the vertices and faces of $\Gamma$, a cocycle {representing $[wt]$ and $\phi$}, where $X_i=[wt]([\partial f_i]),A=[wt]([a]),B=[wt](b)$, and $a$ and $b$ are the red and green cycles respectively. {The Kasteleyn sign $\epsilon$ is $1$ for all edges. If no weight or $\phi$ is indicated for an edge, it means that it is $1$.}}\label{fig:damhex}
\end{figure}
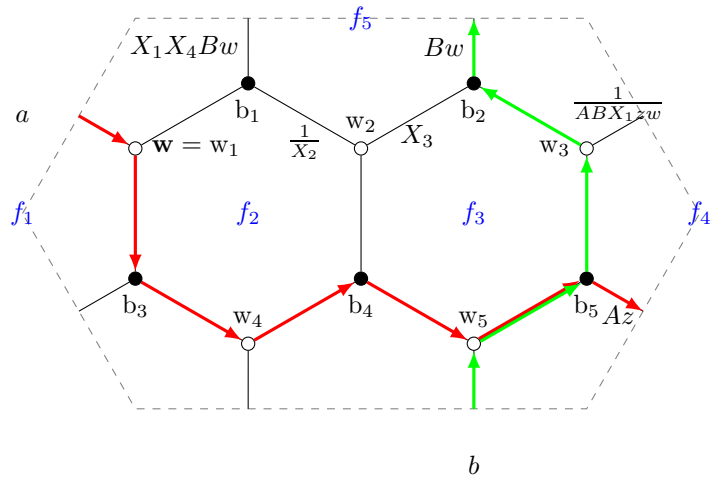

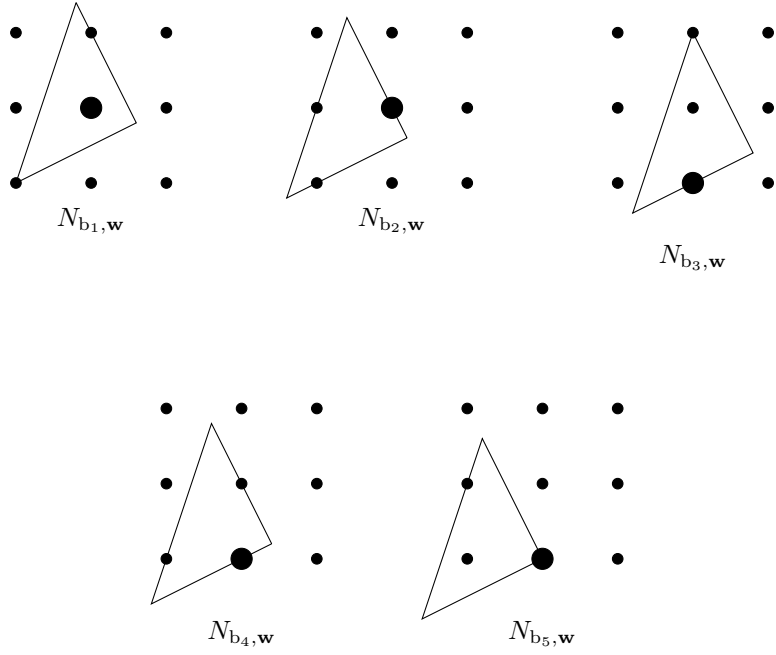
\begin{figure}
	\centering
	
	\begin{tikzpicture}  
		\begin{scope}[shift={(10,6)}]
			\draw[](-1/5,7/5)--(-1,-1)--(3/5,-1/5)--(-1/5,7/5);
			\draw[fill=black] (0,0) circle (4pt);
			
			\draw[fill=black] (0,1) circle (2pt);
			\draw[fill=black] (1,0) circle (2pt);
			\draw[fill=black] (0,-1) circle (2pt);
			\draw[fill=black] (-1,0) circle (2pt);
			\draw[fill=black] (1,1) circle (2pt);
			\draw[fill=black] (1,-1) circle (2pt);
			\draw[fill=black] (-1,-1) circle (2pt);
			\draw[fill=black] (-1,1) circle (2pt);
			\node[](no) at (0,-1.5){$N_{\bw_1,{\bf w}}$};
		\end{scope}	
		
		\begin{scope}[shift={(14,6)}]
			\draw[](1/5,-2/5)--(-3/5,6/5)--(-7/5,-6/5)--(1/5,-2/5);
			\draw[fill=black] (0,0) circle (4pt);
			
			\draw[fill=black] (0,1) circle (2pt);
			\draw[fill=black] (1,0) circle (2pt);
			\draw[fill=black] (0,-1) circle (2pt);
			\draw[fill=black] (-1,0) circle (2pt);
			\draw[fill=black] (1,1) circle (2pt);
			\draw[fill=black] (1,-1) circle (2pt);
			\draw[fill=black] (-1,-1) circle (2pt);
			\draw[fill=black] (-1,1) circle (2pt);
			\node[](no) at (0,-1.5){$N_{\bw_2,{\bf w}}$};
		\end{scope}	
		
		\begin{scope}[shift={(18,5)}]
			\draw[](0,2)--(4/5,2/5)--(-4/5,-2/5)--(0,2
			);
			\draw[fill=black] (0,0) circle (4pt);
			\draw[fill=black] (0,1) circle (2pt);
			\draw[fill=black] (1,0) circle (2pt);
			\draw[fill=black] (0,2) circle (2pt);
			\draw[fill=black] (-1,0) circle (2pt);
			\draw[fill=black] (1,1) circle (2pt);
			\draw[fill=black] (1,2) circle (2pt);
			\draw[fill=black] (-1,2) circle (2pt);
			\draw[fill=black] (-1,1) circle (2pt);
			\node[](no) at (0,-1){$N_{\bw_3,{\bf w}}$};
		\end{scope}	
		
		\begin{scope}[shift={(12,0)}]
			\draw[](2/5,1/5)--(-2/5,9/5)--(-6/5,-3/5)--(2/5,1/5);
			\draw[fill=black] (0,0) circle (4pt);
			
			\draw[fill=black] (0,1) circle (2pt);
			\draw[fill=black] (1,0) circle (2pt);
			\draw[fill=black] (0,2) circle (2pt);
			\draw[fill=black] (-1,0) circle (2pt);
			\draw[fill=black] (1,1) circle (2pt);
			\draw[fill=black] (1,2) circle (2pt);
			\draw[fill=black] (-1,2) circle (2pt);
			\draw[fill=black] (-1,1) circle (2pt);
			\node[](no) at (0,-1){$N_{\bw_4,{\bf w}}$};
		\end{scope}	
		
		\begin{scope}[shift={(16,0)}]
			\draw[](0,0)--(-4/5,8/5)--(-8/5,-4/5)--(0,0);
			\draw[fill=black] (0,0) circle (4pt);
			
			\draw[fill=black] (0,1) circle (2pt);
			\draw[fill=black] (1,0) circle (2pt);
			\draw[fill=black] (0,2) circle (2pt);
			\draw[fill=black] (-1,0) circle (2pt);
			\draw[fill=black] (1,1) circle (2pt);
			\draw[fill=black] (1,2) circle (2pt);
			\draw[fill=black] (-1,2) circle (2pt);
			\draw[fill=black] (-1,1) circle (2pt);
			\node[](no) at (0,-1){$N_{\bw_5,{\bf w}}$};
		\end{scope}	
	\end{tikzpicture}
	\caption{The small polygons for the hexagonal graph.}\label{fig:hexsp}
\end{figure}

Consider the hexagonal graph $\Gamma$ with Newton polygon $N$ and {normal} fan $\Sigma$ as shown in Figure \ref{fig:hex}. 
{We label the vertices of $\Gamma$ as in Figure~\ref{fig:damhex}.} We label the zig-zag paths by $\alpha,\beta,\gamma$, and denote the ray of $\Sigma$ dual to $\tau \in \{\alpha,\beta,\gamma\}$ by $\sigma_\tau$. 

We can take $X_i=[wt]([\partial f_i]), i=1,\dots,4$, and $A=[wt]([a]), B=[wt]([b])$ as coordinates on $H^1(\Gamma,\C)$ (see Figure \ref{fig:damhex}).

The Casimirs are
\begin{align}\la{cashexx}
    C_\alpha= -\frac{B^2 X_1 X_2 X_4}{A}, \quad C_\beta=- \frac{X_3}{A B^3 X_1^2 X_4},\quad C_\gamma = \frac{A^2 B X_1}{X_2 X_3}.
\end{align}
The Kasteleyn matrix is
\[
K=\begin{blockarray}{cccccc}
{\bw_1}& \bw_2&\bw_3&\bw_4&\bw_5& \\
\begin{block}{(ccccc)c}
  1&0&1&0&Az & \w_1\\
 \frac{1}{X_2}&X_3&0&1&0 & \w_2\\
 0&1&\frac{1}{A B X_1 z w}&0&1&\w_3\\
 X_1 X_4 B w &0&1&1&0&\w_4\\
 0&Bw&0&1&1&\w_5\\
\end{block}
\end{blockarray}
\]
Let $P(z,w)=\det K$ and $\spectralcurve=\overline{\{P(z,w)=0\}}$.
The spectral transform is $\kappa_{\Gamma,{\bf w}}=(\spectralcurve,S,\nu) \in \mathcal S_N$, where {since the interior of $N$ contains two lattice points, the divisor} $S=(p_1,q_1)+(p_2,q_2)$ {is a sum of two points, where} 
\begin{align}
   p_1&=-\frac{\sqrt{(-B {X_1} {X_2} {X_3} {X_4}-B {X_1} {X_2} {X_4}-B)^2-4 B^2 {X_1} {X_2}
   {X_4}}+B {X_1} {X_2} {X_3} {X_4}-B {X_1} {X_2} {X_4}+B}{2 A B {X_1}}, \nonumber\\
   q_1&=\frac{-\sqrt{(-B {X_1} {X_2} {X_3} {X_4}-B {X_1} {X_2} {X_4}-B)^2-4 B^2 {X_1} {X_2}
   {X_4}}+B {X_1} {X_2} {X_3} {X_4}+B {X_1} {X_2} {X_4}+B}{2 B^2 {X_1} {X_2} {X_4}}, \nonumber\\
     p_2&=-\frac{-\sqrt{(-B {X_1} {X_2} {X_3} {X_4}-B {X_1} {X_2} {X_4}-B)^2-4 B^2 {X_1} {X_2}
   {X_4}}+B {X_1} {X_2} {X_3} {X_4}-B {X_1} {X_2} {X_4}+B}{2 A B {X_1}}, \nonumber\\
   q_2&=\frac{\sqrt{(-B {X_1} {X_2} {X_3} {X_4}-B {X_1} {X_2} {X_4}-B)^2-4 B^2 {X_1} {X_2}
   {X_4}}+B {X_1} {X_2} {X_3} {X_4}+B {X_1} {X_2} {X_4}+B}{2 B^2 {X_1} {X_2} {X_4}}. \la{eq:pqpq}
\end{align}
The points at infinity are given by the following table:
{
\be \def\arraystretch{1.5}
\begin{array}{|cccc|}
	\hline
	\text{Zig-zag path} & \text{Homology class} & \text{Basis $x_1,x_2$} & \text{Point at infinity}\\
	\hline
	\alpha & (-1,2) & (-1,2),(0,-1) &x_1=\frac{1}{C_\alpha},x_2=0 \\ [0.5ex]
	\hline
	\beta & (-1,-3) & (-1,-3),(0,-1) &x_1=\frac{1}{C_\beta},x_2=0 \\[0.5ex]
	\hline
	\gamma & (2,1)& (2,1),(-1,0)&x_1=\frac{1}{C_\gamma},x_2=0 \\[0.5ex]
	\hline
\end{array} \la{zzpathtable3}
\ee
}

The discrete Abel map $\dd$ is given by
\begin{align*}
    \dd({\bf w})&=0,& 
    \dd(\bw_1)&=D_\beta+D_\gamma,&
    \dd(\bw_2)&=-D_\alpha+2D_\beta+D_\gamma,\\
    \dd(\bw_3)&=D_\alpha+ D_\beta,&
    \dd(\bw_4)&=2D_\beta,&
    \dd(\bw_5)&=-D_\alpha+3D_\beta,
\end{align*}
and $D_N=2D_\alpha+2 D_\beta+D_\gamma$. Since $\dd({\bf w})=0$ and every black vertex {$\bw$} is contained in every zig-zag path, we have 
\begin{align*}
    \divebw_{\bw {\bf w}}&=2 D_\alpha + 2 D_\beta + D_\gamma +\dd(\bw)-D_\alpha-D_\beta-D_\gamma \\
    &=\dd(\bw)+D_\alpha+D_\beta.
\end{align*}
Using this, we compute
\begin{align*}
    \divebw_{\bw_1 {\bf w}}&=D_\alpha+2 D_\beta+D_\gamma,&
    \divebw_{\bw_2 {\bf w}}&=3D_\beta+D_\gamma,&
    \divebw_{\bw_3 {\bf w}}&=2D_\alpha+2D_\beta,\\
    \divebw_{\bw_4 {\bf w}}&=D_\alpha+3D_\beta,&
    \divebw_{\bw_5 {\bf w}}&=4D_\beta.&&
\end{align*}
The small polygons are shown in Figure~\ref{fig:hexsp}. Since the Newton polygon $N$ is primitive, we are in the setting of Remark~\ref{rem::prim}. Every zig-zag path contains every black vertex, so the expression (\ref{eq:simplepoly}) is $0$. Therefore, there are no equations of type 2 in the linear system $\mathbb V_{\bw {\bf w}}$ for any black vertex $\bw$. Since $g=2$, we have two equations of type $1$ for every black vertex $\bw$. Moreover, we note that each of the small polygons in Figure~\ref{fig:hexsp} contains exactly three lattice points, so by Remark~\ref{remind}, we get
\begin{align*}
    {\rm V}_{\bw_1 {\bf w}}&=\begin{vmatrix}
1 & w & z^{-1} w^{-1}\\
1& q_1 & p_1^{-1} q_1^{-1}\\
1& q_2 & p_2^{-1} q_2^{-1}
\end{vmatrix},& {\rm V}_{\bw_2 {\bf w}}&=\begin{vmatrix}
1 & z^{-1} & z^{-1} w^{-1}\\
1& p_1^{-1} & p_1^{-1} q_1^{-1}\\
1& p_2^{-1} & p_2^{-1} q_2^{-1}
\end{vmatrix},&{\rm V}_{\bw_3 {\bf w}}&=\begin{vmatrix}
1 & w &  w^{2}\\
1& q_1 &  q_1^{2}\\
1& q_2 &  q_2^{2}
\end{vmatrix},\\
{\rm V}_{\bw_4 {\bf w}}&=\begin{vmatrix}
1 & w & z^{-1}\\
1& q_1 & p_1^{-1} \\
1& q_2 & p_2^{-1} 
\end{vmatrix},& {\rm V}_{\bw_5 {\bf w}}&=\begin{vmatrix}
1 & z^{-1}w & z^{-1} \\
1& p_1^{-1}q_1 & p_1^{-1} \\
1& p_2^{-1}q_2 & p_2^{-1} 
\end{vmatrix}.&&
\end{align*}
The boundary of the face $f_2$ is the concatenation of the three wedges {$W_1,W_2$ and $W_3$ represented by  
$(\w_2,\alpha), ({\bf w},\beta)$ and $(\w_4,\gamma)$ respectively}. We compute
\begin{align*}
    r_{W_1}=-\frac{V_{\bw_1 \w_2}}{V_{\bw_4 \w_2}}(\nu(\alpha))=-\frac{\begin{vmatrix}
1 & w & z^{-1} w^{-1}\\
1& q_1 & p_1^{-1} q_1^{-1}\\
1& q_2 & p_2^{-1} q_2^{-1}
\end{vmatrix}}{\begin{vmatrix}
1 & w & z^{-1}\\
1& q_1 & p_1^{-1} \\
1& q_2 & p_2^{-1} 
\end{vmatrix}}(\nu(\alpha)).
\end{align*}
To evaluate at $\nu(\alpha)$, as explained in Remark  \ref{remark:factor}, we extend $[\alpha]=(-1,2)$ to the basis $(x_1,x_2)$ of ${\rm M}$, where $x_1=[\alpha]=(-1,2)$ and $x_2= (0,-1)$. Then $\nu(\alpha)$ is given by $x_1=\frac{1}{C_\alpha},x_2=0$. Expressing $z,w$ in the basis $(x_1,x_2)$ as $z=\frac{1}{x_1x_2^2} ,w=\frac{1}{x_2}$, we get
\begin{align*}
    r_{W_1}&=-\frac{\begin{vmatrix}
1 & \frac{1}{x_2} & x_1 x_2^3\\
1& q_1 & p_1^{-1} q_1^{-1}\\
1& q_2 & p_2^{-1} q_2^{-1}
\end{vmatrix}}{\begin{vmatrix}
1 & \frac{1}{x_2} & x_1 x_2^2\\
1& q_1 & p_1^{-1} \\
1& q_2 & p_2^{-1} 
\end{vmatrix}}\left( \frac{1}{C_\alpha},0\right)= -\frac{\begin{vmatrix}
x_2 & {1} & {x_1}x_2^4\\
1& q_1 & p_1^{-1} q_1^{-1}\\
1& q_2 & p_2^{-1} q_2^{-1}
\end{vmatrix}}{\begin{vmatrix}
x_2 & 1 & {x_1}{x_2}^3\\
1& q_1 & p_1^{-1} \\
1& q_2 & p_2^{-1} 
\end{vmatrix}}\left( \frac{1}{C_\alpha},0\right)= -\frac{\begin{vmatrix}
0 & {1} & 0\\
1& q_1 & p_1^{-1} q_1^{-1}\\
1& q_2 & p_2^{-1} q_2^{-1}
\end{vmatrix}}{\begin{vmatrix}
0 & 1 & 0\\
1& q_1 & p_1^{-1} \\
1& q_2 & p_2^{-1} 
\end{vmatrix}}\\&={-\frac{p_1 q_1 -p_2 q_2}{q_1
q_2 (p_1-p_2)}},
\end{align*}
where we factored out $x_2$ from the numerator and denominator and then evaluated at $(x_1,x_2)=( \frac{1}{C_\alpha},0)$.

For $W_2$, letting $(x_1,x_2)=((-1,-3),(0,-1))$ we have $z=\frac{x_2^3}{x_1 },w=\frac{1}{x_2}$, and $\nu(\beta)$ is given by $x_1=\frac{1}{C_\beta},x_2=0$. Therefore, we get
\begin{align*}
    r_{W_2}&=-\frac{V_{\bw_3 {\bf w}}}{V_{\bw_1 {\bf w}}}(\nu(\beta))=-\frac{\begin{vmatrix}
1 & w &  w^{2}\\
1& q_1 &  q_1^{2}\\
1& q_2 &  q_2^{2}
\end{vmatrix}}{\begin{vmatrix}
1 & w & z^{-1} w^{-1}\\
1& q_1 & p_1^{-1} q_1^{-1}\\
1& q_2 & p_2^{-1} q_2^{-1}
\end{vmatrix}}(\nu(\beta))=-\frac{\begin{vmatrix}
1 & \frac{1}{x_2} &  \frac{1}{x_2^2}\\
1& q_1 &  q_1^{2}\\
1& q_2 &  q_2^{2}
\end{vmatrix}}{\begin{vmatrix}
1 & \frac{1}{x_2} & \frac{x_1}{ x_2^2}\\
1& q_1 & p_1^{-1} q_1^{-1}\\
1& q_2 & p_2^{-1} q_2^{-1}
\end{vmatrix}}\left(\frac{1}{C_\beta},0\right)=-\frac{\begin{vmatrix}
0 & 0&  1\\
1& q_1 &  q_1^{2}\\
1& q_2 &  q_2^{2}
\end{vmatrix}}{\begin{vmatrix}
0 & 0 & \frac{1}{C_\beta}\\
1& q_1 & p_1^{-1} q_1^{-1}\\
1& q_2 & p_2^{-1} q_2^{-1}
\end{vmatrix}}\\&={-C_\beta}.
\end{align*}
Finally, for $W_3$, letting $(x_1,x_2)=((2,1),(-1,0))$ we have $z=\frac{1}{x_2 },w=x_1{x_2^2}$, and $\nu(\gamma)$ is given by $x_1=\frac{1}{C_\gamma},x_2=0$. Therefore, we get
\begin{align*}
    r_{W_3}&=-\frac{V_{\bw_4 {\bf w}}}{V_{\bw_3 {\bf w}}}(\nu(\gamma))=-\frac{\begin{vmatrix}
1 & w & z^{-1}\\
1& q_1 & p_1^{-1} \\
1& q_2 & p_2^{-1} 
\end{vmatrix}}{\begin{vmatrix}
1 & w &  w^{2}\\
1& q_1 &  q_1^{2}\\
1& q_2 &  q_2^{2}
\end{vmatrix}}(\nu(\gamma))=-\frac{\begin{vmatrix}
1 & x_1x_2^2 & x_2\\
1& q_1 & p_1^{-1} \\
1& q_2 & p_2^{-1} 
\end{vmatrix}}{\begin{vmatrix}
1 & x_1x_2^2  &  x_1^2 x_2^4\\
1& q_1 &  q_1^{2}\\
1& q_2 &  q_2^{2}
\end{vmatrix}}\left(\frac{1}{C_\gamma},0\right)=-\frac{\begin{vmatrix}
1 & 0 & 0\\
1& q_1 & p_1^{-1} \\
1& q_2 & p_2^{-1} 
\end{vmatrix}}{\begin{vmatrix}
1 & 0  &  0\\
1& q_1 &  q_1^{2}\\
1& q_2 &  q_2^{2}
\end{vmatrix}}\\
&={\frac{p_1 q_1-p_2 q_2}{p_1
p_2 q_1 q_2 (q_1-q_2)}}.
\end{align*}
Putting everything together, we get
\begin{align*}
X_2&=-\frac{\begin{vmatrix}
0 & {1} & 0\\
1& q_1 & p_1^{-1} q_1^{-1}\\
1& q_2 & p_2^{-1} q_2^{-1}
\end{vmatrix}}{\begin{vmatrix}
0 & 1 & 0\\
1& q_1 & p_1^{-1} \\
1& q_2 & p_2^{-1} 
\end{vmatrix}}\frac{\begin{vmatrix}
0 & 0&  1\\
1& q_1 &  q_1^{2}\\
1& q_2 &  q_2^{2}
\end{vmatrix}}{\begin{vmatrix}
0 & 0 & \frac{1}{C_\beta}\\
1& q_1 & p_1^{-1} q_1^{-1}\\
1& q_2 & p_2^{-1} q_2^{-1}
\end{vmatrix}}\frac{\begin{vmatrix}
1 & 0 & 0\\
1& q_1 & p_1^{-1} \\
1& q_2 & p_2^{-1} 
\end{vmatrix}}{\begin{vmatrix}
1 & 0  &  0\\
1& q_1 &  q_1^{2}\\
1& q_2 &  q_2^{2}
\end{vmatrix}}\\
&={\frac{C_\beta (p_1 q_1-p_2 q_2)^2}{p_1
		p_2 q_1^2 q_2^2 (p_1-p_2)
		(q_1-q_2)}},
\end{align*}
{with similar formulas for $X_1,X_3,X_4,A,B$. It may be easily verified that these invert the spectral transform
by} plugging in the formulas (\ref{cashexx}) and (\ref{eq:pqpq}) into the right-hand side and simplifying using computer algebra.

\subsection{Non-primitive example} 

\begin{figure}
\centering

	\begin{tikzpicture}  
	\begin{scope}
		\draw[dashed, gray] (0,0) rectangle (8,8);
\coordinate[bvert] (b1) at (2,7);
			\coordinate[bvert] (b2) at (2,5);
			\coordinate[bvert] (b3) at (5,6);
			\coordinate[bvert] (b4) at (7,6);
			\coordinate[bvert] (b5) at (1,2);
			\coordinate[bvert] (b6) at (3,2);
			\coordinate[bvert] (b7) at (6,3);
			\coordinate[bvert] (b8) at (6,1);
			\coordinate[wvert] (w1) at (1,6);
			\coordinate[wvert] (w2) at (3,6);
			\coordinate[wvert] (w3) at (6,7);
			\coordinate[wvert] (w4) at (6,5);
			\coordinate[wvert] (w5) at (2,3);
			\coordinate[wvert] (w6) at (2,1);
			\coordinate[wvert] (w7) at (5,2);
			\coordinate[wvert] (w8) at (7,2);

\draw[] (0,2)--(b5)--(w5)--(b6)--(w7)--(b7)--(w8)--(8,2)
(b5)--(w6)--(b6)
(w7)--(b8)--(w8)
(w6)--(2,0)
(b8)--(6,0)
(w5)--(b2)--(w2)--(b3)--(w4)--(b7)
(w4)--(b4)--(8,6)
(b4)--(w3)--(b3)
(0,6)--(w1)--(b2)
(w1)--(b1)--(w2)
(b1)--(2,8)
(w3)--(6,8)
;
\draw[orange!70!yellow,very thick,<-] (0,6)to[out=-27,in=207](4,6)to[out=27,in=180-27](8,6);
\draw[green,very thick,->] (2,8)to[out=-90-27,in=117](2,4)to[out=-90+27,in=90-27](2,0);
\draw[blue,very thick,<-] (2,8)to[out=-90+27,in=90-27](2,4)to[out=-90-27,in=90+27](2,0);
\draw[red,very thick,->] (0,6)to[out=27,in=180-27](4,6)to[out=-27,in=180+27](8,6);

\draw[red,very thick,->] (0,2)to[out=-27,in=207](4,2)to[out=27,in=180-27](8,2);
\draw[blue,very thick,<-] (6,8)to[out=-90-27,in=117](6,4)to[out=-90+27,in=90-27](6,0);
\draw[green,very thick,->] (6,8)to[out=-90+27,in=90-27](6,4)to[out=-90-27,in=90+27](6,0);
\draw[orange!70!yellow,very thick,<-] (0,2)to[out=27,in=180-27](4,2)to[out=-27,in=180+27](8,2);

	\node[](no) at (2.5,0.5) {$\gamma_1$};
		\node[](no) at (1.5,0.5) {$\alpha_1$};
			\node[](no) at (0.5,2.5) {$\beta_1$};
			\node[](no) at (0.5,1.5) {$\delta_1$};
		\node[](no) at (6.5,0.5) {$\alpha_2$};
		\node[](no) at (5.5,0.5) {$\gamma_2$};
			\node[](no) at (0.5,5.5) {$\beta_2$};	\node[](no) at (0.5,6.5) {$\delta_2$};

\end{scope}
\begin{scope}[shift={(11,6)},scale=0.7]
\draw[red,very thick,->] (-1,-1)--(0,-1);
\draw[red,very thick,->] (0,-1)--(1,-1);
\draw[green,very thick,->] (-1,1)--(-1,0);
\draw[green,very thick,->] (-1,0)--(-1,-1);
\draw[blue,very thick,->] (1,-1)--(1,0);
\draw[blue,very thick,->] (1,0)--(1,1);
\draw[orange!70!yellow,very thick,->] (1,1)--(0,1);
\draw[orange!70!yellow,very thick,->] (0,1)--(-1,1);
	\draw[fill=black] (0,0) circle (4pt);
		
		\draw[fill=black] (0,1) circle (2pt);
		\draw[fill=black] (1,0) circle (2pt);
		\draw[fill=black] (0,-1) circle (2pt);
			\draw[fill=black] (-1,0) circle (2pt);
			\draw[fill=black] (1,1) circle (2pt);
		\draw[fill=black] (1,-1) circle (2pt);
		\draw[fill=black] (-1,-1) circle (2pt);
			\draw[fill=black] (-1,1) circle (2pt);
				\node[](no) at (0,-2.5) {$N$};
\node[](no) at (2.5,0) {$[\alpha_1]+[\alpha_2]$};		
\node[](no) at (-2.5,0) {$[\gamma_1]+[\gamma_2]$};			
\node[](no) at (0,1.5) {$[\beta_1]+[\beta_2]$};	
\node[](no) at (0,-1.5) {$[\delta_1]+[\delta_2]$};			
\end{scope}	

\begin{scope}[shift={(11,2)},scale=0.7]

\draw[red,very thick,->] (0,0)--(0,1);
\draw[green,very thick,->] (0,0)--(1,0);
\draw[blue,very thick,->] (0,0)--(-1,0);
\draw[orange!70!yellow,very thick,->] (0,0)--(0,-1);
	\draw[fill=black] (0,0) circle (2pt);
	\node[](no) at (1.5,0) {$\sigma_\gamma$};
		\node[](no) at (-1.5,0) {$\sigma_\alpha$};
			\node[](no) at (0,-1.5) {$\sigma_\beta$};	\node[](no) at (0,1.5) {$\sigma_\delta$};
			\node[](no) at (0,-2.5) {$\Sigma$};
\end{scope}	
\end{tikzpicture}
\caption{A square-octagon graph, its Newton polygon {$N$} and normal fan {$\Sigma$}, with zig-zag paths and rays labeled.}\label{fig:sqoct}
\end{figure}
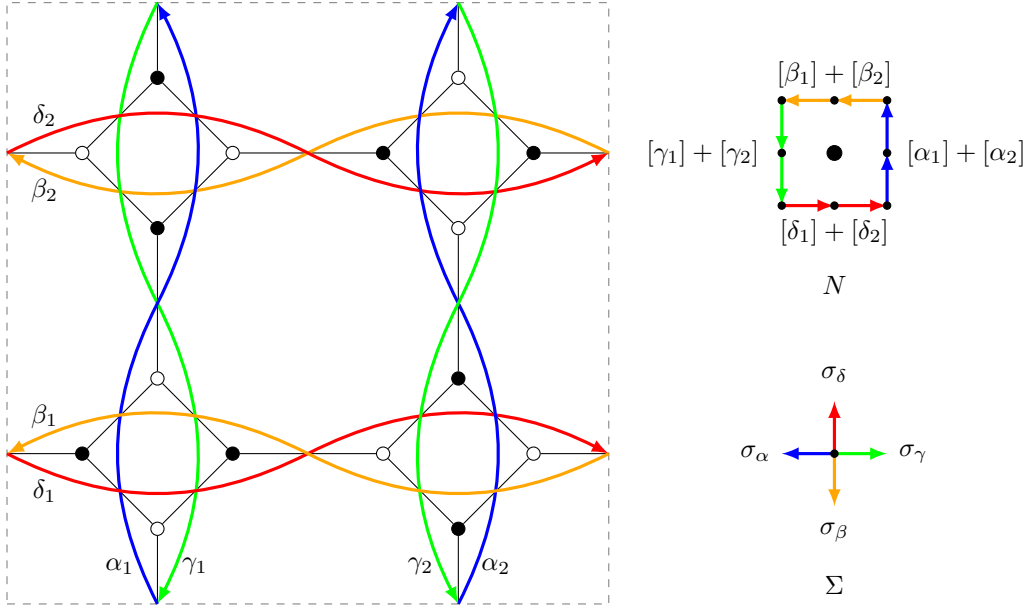

Consider the square-octagon graph $\Gamma$ with Newton polygon $N$ and {normal} fan $\Sigma$ as shown in Figure \ref{fig:sqoct}. {
	We label the vertices of $\Gamma$ as in Figure \ref{fig:damsqoct}.} {We label the rays of $\Sigma$ by $\sigma_\alpha,\sigma_\beta,\sigma_\gamma,\sigma_\delta$ and the two zig-zag paths dual to ray $\sigma_\tau$ by $\{\tau_1,\tau_2\}$, for $\tau \in \{\alpha,\beta,\gamma,\delta\}$. }

We can take $X_i:=[wt]([\partial f_i]), i=1,\dots,7$, and $A:=[wt]([a]), B:=[wt]([b])$ as coordinates on $H^1(\Gamma,\C)$ (see Figure~\ref{fig:damsqoct}). The Casimirs are
\begin{align*}
    C_{\alpha_1}&=X_1 X_3 X_7 B, & C_{\alpha_2}&=\frac{B {X_2} {X_3} {X_4} {X_6} {X_7}}{{X_1} {X_5}}, &
    C_{\beta_1}&= \frac{X_2}{A X_1 X_5},& C_{\beta_2}&= \frac{1}{A X_7},\\
    C_{\gamma_1}&= \frac{X_5}{B X_1 X_3},& C_{\gamma_2}&= \frac{X_6}{B},&
    C_{\delta_1}&=\frac{A X_1}{X_2 X_6},& C_{\delta_2}&= \frac{A X_1 X_5}{X_2 X_3 X_4 X_6 X_7}.
\end{align*}
Since the Newton polygon $N$ has only one interior lattice point, the divisor $S=(p,q)$ consists of a single point. The Kasteleyn matrix is 
\[
K=\begin{pmatrix}
 1 & 1 & 0 & A z & 0 & 0 & 0 & 0 \\
 1 & -{X_7} & 1 & 0 & 0 & 0 & 0 & 0 \\
 0 & 0 & 1 & 1 & 0 & 0 & 0 & \frac{1}{B w} \\
 0 & 0 & \frac{{X_1} {X_5}}{{X_2} {X_3} {X_4} {X_6} {X_7}} & -1 & 0 & 0 & 1 & 0 \\
 0 & \frac{1}{{X_3}} & 0 & 0 & 1 & \frac{1}{{X_5}} & 0 & 0 \\
 B w {X_1} & 0 & 0 & 0 & 1 & -1 & 0 & 0 \\
 0 & 0 & 0 & 0 & 0 & \frac{{X_1}}{{X_2}} & {X_6} & 1 \\
 0 & 0 & 0 & 0 & \frac{1}{A z} & 0 & 1 & -1 \\
\end{pmatrix}
.\]
Let $P(z,w)=\det K$ and $\spectralcurve=\overline{\{P(z,w)=0\}}$.
The spectral transform is $\kappa_{\Gamma,{\bf w}}=(\spectralcurve,S,\nu) \in \mathcal S_N$, where 
\begin{align*}
    p&=-\frac{{X_2} {X_4} {X_6} \left({X_3} {X_5} {X_6} {X_7} \left({X_1}^2 ({X_4}+1)+{X_2}
   {X_4}\right)+{X_1} {X_2} {X_3}^2 {X_4} {X_6}^2 {X_7}^2+{X_1} {X_5}^2\right)}{A ({X_1}
   {X_5}+{X_2} {X_3} {X_4} {X_6} {X_7}) }\\
 &~~~~~\times \frac{1}{\left({X_3} {X_4} {X_6} {X_7} \left({X_1}^2
   {X_5}+{X_2} ({X_5}+1) ({X_6}+1)\right)+{X_1} {X_5} ({X_6}
   ({X_4}+{X_5}+1)+{X_5}+1)\right)},\\
   q&=\frac{{X_5} \left(-{X_3} {X_4} {X_6} {X_7} \left({X_1}^2+{X_2} {X_6}+{X_2}\right)-{X_1}
   {X_5} ({X_6}+1)\right)}{B {X_1} {X_3} {X_7} ({X_1} {X_5} ({X_4}
   {X_6}+{X_6}+1)+{X_2} {X_3} {X_4} {X_6} ({X_6}+1) {X_7})}.
\end{align*}
The table below lists the points at infinity for each of the zig-zag paths:

{
	\be \def\arraystretch{1.5}
	\begin{array}{|cccc|}
		\hline
		\text{Zig-zag path} & \text{Homology class} & \text{Basis $x_1,x_2$} & \text{Point at infinity}\\
\hline
		\alpha_1 & \multirow{2}{*}{(0,1)} &\multirow{2}{*}{ (0,1),(-1,0)}&x_1=\frac{1}{C_{\alpha_1}},x_2=0 \\ 
		\alpha_2 & & & x_1=\frac{1}{C_{\alpha_2}},x_2=0  \\[0.5ex]
		\hline
		\beta_1 & \multirow{2}{*}{(-1,0)} &  \multirow{2}{*}{(-1,0),(0,-1)}&x_1=\frac{1}{C_{\beta_1}},x_2=0  \\
		\beta_2 & &  &x_1=\frac{1}{C_{\beta_2}},x_2=0 \\[0.5ex]
		\hline
	
		\gamma_1 & \multirow{2}{*}{(0,-1)}&  \multirow{2}{*}{(0,-1),(1,0)}&x_1=\frac{1}{C_{\gamma_1}},x_2=0  \\
		\gamma_2 & &  &x_1=\frac{1}{C_{\gamma_2}},x_2=0   \\[0.5ex]
		\hline

		\delta_1 & \multirow{2}{*}{(1,0)} &  \multirow{2}{*}{(1,0),(0,1)}&x_1=\frac{1}{C_{\delta_1}},x_2=0   \\
		\delta_2 &  &  &x_1=\frac{1}{C_{\delta_2}},x_2=0  \\[0.5ex]
		\hline
	\end{array} \la{zzpathtable2}
	\ee
}

\begin{figure}
\centering

	\begin{tikzpicture}  
	\begin{scope}
		\draw[dashed, gray] (0,0) rectangle (8,8);
	
		    \coordinate[bvert,label=below:${\bw_1}$] (b1) at (2,7);
			\coordinate[bvert,label=above:${\bw_2}$] (b2) at (2,5);
			\coordinate[bvert,label=right:${\bw_3}$] (b3) at (5,6);
			\coordinate[bvert,label=left:${\bw_4}$] (b4) at (7,6);
			\coordinate[bvert,label=right:${\bw_5}$] (b5) at (1,2);
			\coordinate[bvert,label=left:${\bw_6}$] (b6) at (3,2);
			\coordinate[bvert,label=below:${\bw_7}$] (b7) at (6,3);
			\coordinate[bvert,label=above:${\bw_8}$] (b8) at (6,1);
			
			\coordinate[wvert,label=below:${\bf w}$] (w1) at (1,6);	
		    
			\coordinate[wvert,label=above:${\w_2}$] (w2) at (3,6);
			\coordinate[wvert] (w3) at (6,7);
			\coordinate[wvert] (w4) at (6,5);
			\coordinate[wvert] (w5) at (2,3);
			\coordinate[wvert] (w6) at (2,1);
			\coordinate[wvert] (w7) at (5,2);
			\coordinate[wvert] (w8) at (7,2);

\draw[] (0,2)--(b5)--(w5)--node[right]{$\frac{1 }{X_5}$}(b6)--node[below]{$\frac{X_1 }{X_2}$}(w7)--node[left]{$X_6$}(b7)--(w8)--node[above]{$\frac{1}{Az}$}(8,2)
(b5)--(w6)--node[right]{$-1$}(b6)
(w7)--(b8)--node[right]{$-1$}(w8)
(w6)--(2,0)
(b8)--(6,0)
(w5)--node[left]{$\frac{1}{X_3}$}(b2)--node[right]{$-X_7$}(w2)--(b3)--node[left]{$U$}(w4)--(b7)
(w4)--node[right]{$-1$}(b4)--node[above]{$Az$}(8,6)
(b4)--(w3)--(b3)
(0,6)--(w1)--(b2)
(w1)--(b1)--(w2)
(b1)--node[left]{$X_1 B w$}(2,8)
(w3)--node[left]{$\frac{1}{Bw}$}(6,8)
;
\draw[red,very thick,->] (0,6)--(w1);
\draw[red,very thick,->] (w1)--(b1);
\draw[red,very thick,->] (b1)--(w2);
\draw[red,very thick,->] (w2)--(b3);
\draw[red,very thick,->] (b3)--(w3);
\draw[red,very thick,->] (w3.325)--(b4.125);
\draw[red,very thick,->] (b4)--(8,6);

\draw[green,very thick,->] (6,0)--(b8);
\draw[green,very thick,->] (b8)--(w8);
\draw[green,very thick,->] (w8)--(b7);
\draw[green,very thick,->] (b7)--(w4);
\draw[green,very thick,->] (w4)--(b4);
\draw[green,very thick,->] (b4.145)--(w3.305);
\draw[green,very thick,->] (w3)--(6,8);
\node[](no) at (6,-0.5){$b$};
\node[](no) at (-0.5,6){$a$};
\node[blue](no) at (0,0) {$f_1$};
\node[blue](no) at (4,0) {$f_2$};
\node[blue](no) at (0,4) {$f_3$};
\node[blue](no) at (4,4) {$f_4$};
\node[blue](no) at (2,2) {$f_5$};
\node[blue](no) at (6,2) {$f_6$};
\node[blue](no) at (2,6) {$f_7$};
\node[blue](no) at (6,6) {$f_8$};

\end{scope}

\end{tikzpicture}
\caption{Labeling of the vertices and faces of $\Gamma$, and a cocycle and Kasteleyn sign, where $X_i=[wt]([\partial f_i]),A=[wt]([a]),B=[wt](b)$ and $U=\frac{X_1 X_5}{X_2X_3X_4X_6 X_7}$. The edges with no weight indicated have weight $1$.}\label{fig:damsqoct}
\end{figure}
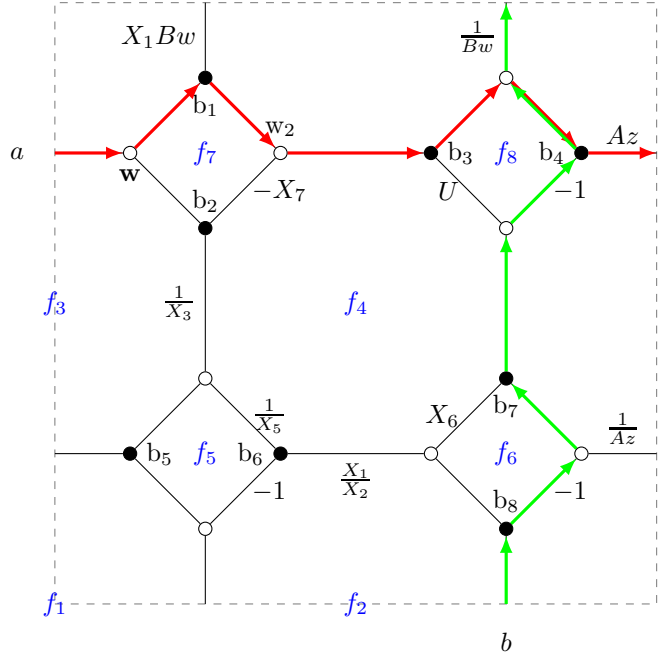

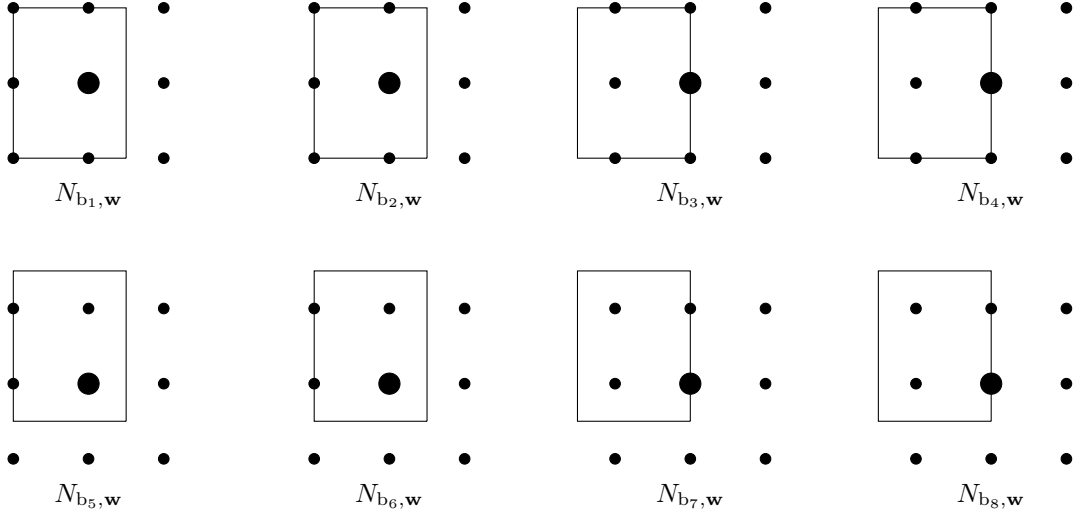
\begin{figure}
\centering

	\begin{tikzpicture}  
\begin{scope}[shift={(10,6)}]
\draw[](0.5,-1)--(0.5,1)--(-1,1)--(-1,-1)--(0.5,-1);
	\draw[fill=black] (0,0) circle (4pt);
		
		\draw[fill=black] (0,1) circle (2pt);
		\draw[fill=black] (1,0) circle (2pt);
		\draw[fill=black] (0,-1) circle (2pt);
			\draw[fill=black] (-1,0) circle (2pt);
			\draw[fill=black] (1,1) circle (2pt);
		\draw[fill=black] (1,-1) circle (2pt);
		\draw[fill=black] (-1,-1) circle (2pt);
			\draw[fill=black] (-1,1) circle (2pt);
			\node[](no) at (0,-1.5){$N_{\bw_1,{\bf w}}$};
\end{scope}

\begin{scope}[shift={(14,6)}]
\draw[](0.5,-1)--(0.5,1)--(-1,1)--(-1,-1)--(0.5,-1);
	\draw[fill=black] (0,0) circle (4pt);
		
		\draw[fill=black] (0,1) circle (2pt);
		\draw[fill=black] (1,0) circle (2pt);
		\draw[fill=black] (0,-1) circle (2pt);
			\draw[fill=black] (-1,0) circle (2pt);
			\draw[fill=black] (1,1) circle (2pt);
		\draw[fill=black] (1,-1) circle (2pt);
		\draw[fill=black] (-1,-1) circle (2pt);
			\draw[fill=black] (-1,1) circle (2pt);
			\node[](no) at (0,-1.5){$N_{\bw_2,{\bf w}}$};
\end{scope}	

\begin{scope}[shift={(18,6)}]
\draw[](0.5-0.5,-1)--(0.5-0.5,1)--(-1-0.5,1)--(-1-0.5,-1)--(0.5-0.5,-1);
	\draw[fill=black] (0,0) circle (4pt);
		
		\draw[fill=black] (0,1) circle (2pt);
		\draw[fill=black] (1,0) circle (2pt);
		\draw[fill=black] (0,-1) circle (2pt);
			\draw[fill=black] (-1,0) circle (2pt);
			\draw[fill=black] (1,1) circle (2pt);
		\draw[fill=black] (1,-1) circle (2pt);
		\draw[fill=black] (-1,-1) circle (2pt);
			\draw[fill=black] (-1,1) circle (2pt);
			\node[](no) at (0,-1.5){$N_{\bw_3,{\bf w}}$};
\end{scope}	

\begin{scope}[shift={(22,6)}]
\draw[](0.5-0.5,-1)--(0.5-0.5,1)--(-1-0.5,1)--(-1-0.5,-1)--(0.5-0.5,-1);
	\draw[fill=black] (0,0) circle (4pt);
		
		\draw[fill=black] (0,1) circle (2pt);
		\draw[fill=black] (1,0) circle (2pt);
		\draw[fill=black] (0,-1) circle (2pt);
			\draw[fill=black] (-1,0) circle (2pt);
			\draw[fill=black] (1,1) circle (2pt);
		\draw[fill=black] (1,-1) circle (2pt);
		\draw[fill=black] (-1,-1) circle (2pt);
			\draw[fill=black] (-1,1) circle (2pt);
			\node[](no) at (0,-1.5){$N_{\bw_4,{\bf w}}$};
\end{scope}	
\begin{scope}[shift={(10 ,2)}]
\draw[](.5,-0.5)--(.5,1.5)--(-1,1.5)--(-1,-0.5)--(.5,-0.5);
	\draw[fill=black] (0,0) circle (4pt);
		
		\draw[fill=black] (0,1) circle (2pt);
		\draw[fill=black] (1,0) circle (2pt);
		\draw[fill=black] (0,-1) circle (2pt);
			\draw[fill=black] (-1,0) circle (2pt);
			\draw[fill=black] (1,1) circle (2pt);
		\draw[fill=black] (1,-1) circle (2pt);
		\draw[fill=black] (-1,-1) circle (2pt);
			\draw[fill=black] (-1,1) circle (2pt);
			\node[](no) at (0,-1.5){$N_{\bw_5,{\bf w}}$};
\end{scope}	
\begin{scope}[shift={(14 ,2)}]
\draw[](.5,-0.5)--(.5,1.5)--(-1,1.5)--(-1,-0.5)--(.5,-0.5);
	\draw[fill=black] (0,0) circle (4pt);
		
		\draw[fill=black] (0,1) circle (2pt);
		\draw[fill=black] (1,0) circle (2pt);
		\draw[fill=black] (0,-1) circle (2pt);
			\draw[fill=black] (-1,0) circle (2pt);
			\draw[fill=black] (1,1) circle (2pt);
		\draw[fill=black] (1,-1) circle (2pt);
		\draw[fill=black] (-1,-1) circle (2pt);
			\draw[fill=black] (-1,1) circle (2pt);
			\node[](no) at (0,-1.5){$N_{\bw_6,{\bf w}}$};
\end{scope}	

\begin{scope}[shift={(18 ,2)}]
\draw[](0,-0.5)--(0,1.5)--(-1.5,1.5)--(-1.5,-0.5)--(0,-0.5);
	\draw[fill=black] (0,0) circle (4pt);
		
		\draw[fill=black] (0,1) circle (2pt);
		\draw[fill=black] (1,0) circle (2pt);
		\draw[fill=black] (0,-1) circle (2pt);
			\draw[fill=black] (-1,0) circle (2pt);
			\draw[fill=black] (1,1) circle (2pt);
		\draw[fill=black] (1,-1) circle (2pt);
		\draw[fill=black] (-1,-1) circle (2pt);
			\draw[fill=black] (-1,1) circle (2pt);
			\node[](no) at (0,-1.5){$N_{\bw_7,{\bf w}}$};
\end{scope}	
\begin{scope}[shift={(22 ,2)}]
\draw[](0,-0.5)--(0,1.5)--(-1.5,1.5)--(-1.5,-0.5)--(0,-0.5);
	\draw[fill=black] (0,0) circle (4pt);
		
		\draw[fill=black] (0,1) circle (2pt);
		\draw[fill=black] (1,0) circle (2pt);
		\draw[fill=black] (0,-1) circle (2pt);
			\draw[fill=black] (-1,0) circle (2pt);
			\draw[fill=black] (1,1) circle (2pt);
		\draw[fill=black] (1,-1) circle (2pt);
		\draw[fill=black] (-1,-1) circle (2pt);
			\draw[fill=black] (-1,1) circle (2pt);
			\node[](no) at (0,-1.5){$N_{\bw_8,{\bf w}}$};
\end{scope}	
\end{tikzpicture}
\caption{The small polygons for the square-octagon graph.}\label{fig:sqoctsp}
\end{figure}
The discrete Abel map $\dd$ is given by $ \dd({\bf w})=0$ and
\begin{align*}
       \dd(\bw_1)&=\frac 1 2 D_\gamma+\frac 1 2 D_\delta,&
    \dd(\bw_2)&=\frac 1 2 D_\beta+\frac 1 2 D_\gamma,\\
    \dd(\bw_3)&=\frac 1 2(- D_\alpha+D_\beta+D_\gamma+D_\delta),&
    \dd(\bw_4)&=-D_\alpha+\frac 1 2 D_\beta+D_\gamma+\frac 1 2 D_\delta \\
    \dd(\bw_5)&=D_\beta,&
    \dd(\bw_6)&=-\frac 1 2 D_\alpha+ D_\beta+\frac 1 2 D_\gamma \\
    \dd(\bw_7)&=-\frac 1 2 D_\alpha+\frac 1 2 D_\beta+D_\gamma,&
    \dd(\bw_8)&=-\frac 1 2 D_\alpha+ D_\beta+D_\gamma-\frac 1 2 D_\delta.
\end{align*}
We have $D_N=D_\alpha+D_\beta+D_\gamma+D_\delta$, using which we compute
\begin{align*}
    \divebw_{\bw_1 {\bf w}}&=\frac 1 2 D_\alpha+ D_\beta+D_\gamma+D_\delta, &
    \divebw_{\bw_2 {\bf w}}&=\frac 1 2 D_\alpha+ D_\beta+D_\gamma+D_\delta,\\
    \divebw_{\bw_3 {\bf w}}&=D_\beta+\frac 3 2 D_\gamma+ D_\delta, &
    \divebw_{\bw_4 {\bf w}}&=D_\beta+\frac 3 2 D_\gamma+ D_\delta,\\
    \divebw_{\bw_5 {\bf w}}&=\frac 1 2 D_\alpha+\frac 3 2 D_\beta+ D_\gamma+\frac 1 2 D_\delta, &
    \divebw_{\bw_6 {\bf w}}&=\frac 1 2 D_\alpha+\frac 3 2 D_\beta+ D_\gamma+\frac 1 2 D_\delta,\\
    \divebw_{\bw_7 {\bf w}}&=\frac 3 2 D_\beta+ \frac 3 2 D_\gamma+\frac 1 2 D_\delta,& 
    \divebw_{\bw_8 {\bf w}}&=\frac 3 2 D_\beta+ \frac 3 2 D_\gamma+\frac 1 2 D_\delta.
\end{align*}
The corresponding small polygons are shown in Figure \ref{fig:sqoctsp}. Therefore, we have 
\[
{\rm V}_{\bw_1 {\bf w}}= a_{(-1,-1)}z^{-1}w^{-1}+a_{(0,-1)} w^{-1}+a_{(1,-1)}zw^{-1}+a_{(-1,0)}z^{-1}+a_{(0,0)}+a_{(1,0)} z,
\]
where the $a_{m}$ satisfy the system of equations $\mathbb V_{\bw_1 {\bf w}}$ that we now determine. We have the equation of type 1:
\[
a_{(-1,-1)}p^{-1}q^{-1}+a_{(0,-1)} q^{-1}+a_{(1,-1)}pq^{-1}+a_{(-1,0)}p^{-1}+a_{(0,0)}+a_{(1,0)} p
=0.
\]
{
	To find the zig-zag paths that contribute equations of type 2, we compute (\ref{contpointsatinf}). We have
\begin{align*}
	-\restr{D_N}{\spectralcurve}&= \nu(\alpha_1) + \nu(\alpha_2) +\nu(\beta_1) + \nu(\beta_2) + \nu(\gamma_1) + \nu(\gamma_2) + \nu(\delta_1) + \nu(\delta_2),\\
	{\bf d}({\bf w})-{\bf d}(\bw_1) &=-\nu(\gamma_1) - \nu(\delta_2),\\
	\restr{\floor{\divebw_{\bw_1 {\bf w}}}}{\spectralcurve}&= \nu(\beta_1) + \nu(\beta_2) + \nu(\gamma_1) + \nu(\gamma_2) + \nu(\delta_1) + \nu(\delta_2),
\end{align*}
using which we get that (\ref{contpointsatinf}) is equal to $\nu(\beta_1)+\nu(\beta_2)+\nu(\gamma_2)+\nu(\delta_1) $, so we have four equations of type 2, one for each of the zig-zag paths $\beta_1,\beta_2,\gamma_2,\delta_1$.	
}

Therefore, we have $5$ equations and $6$ variables, so we are in the setting of Remark \ref{remind} where ${\rm V}_{\bw_1 {\bf w}}=\det\mathbb {V}^\chi_{\bw_1 {\bf w}}$. Computing the equations of type 2, we get
\[
{\rm V}_{\bw_1 {\bf w}}=\begin{vmatrix}
z^{-1}w^{-1}&w^{-1}&zw^{-1}&z^{-1}&1&z\\
p^{-1}q^{-1}&q^{-1}&pq^{-1}&p^{-1}&1&p\\
1&C_{\beta_1}&0&0&0&0\\
1&C_{\beta_2}&0&0&0&0\\
C_{\gamma_2}&0&1&0&C_{\gamma_2}^{-1}&0\\
0&0&0&0&C_{\delta_1}&1
\end{vmatrix}.
\]
In like fashion, {for ${\rm V}_{\bw_2 {\bf w}}$, we have an equation of type 1 and four equations of type 2 for the zig-zag paths $\beta_1, \gamma_2, \delta_1, \delta_2$}. We compute
\[
{\rm V}_{\bw_2 {\bf w}}=\begin{vmatrix}
z^{-1}w&w&z^{-1}&1&z^{-1}w^{-1}&w^{-1}\\
p^{-1}q&q&p^{-1}&1&p^{-1}q^{-1}&q^{-1}\\
1&C_{\beta_1}&0&0&0&0\\
C_{\gamma_2}&0&1&0&C_{\gamma_2}^{-1}&0\\
0&0&0&0&C_{\delta_1}&1\\
0&0&0&0&C_{\delta_2}&1
\end{vmatrix}.
\]
We write the boundary of the face $f_7$ as the concatenation of the two wedges {$W_1$ and $W_2$ represented by
$({\bf w},\gamma_1)$ and $(\bw_2,\alpha_1)$ respectively}. We have
\begin{align*}
    r_{W_1}&=\frac{{\rm V}_{\bw_2 {\bf w}}}{{\rm V}_{\bw_1 {\bf w}}}(\nu(\gamma_1))=\frac{\begin{vmatrix}
C_{\gamma_1}&0&1&0&C_{\gamma_1}^{-1}&0\\
p^{-1}q&q&p^{-1}&1&p^{-1}q^{-1}&q^{-1}\\
1&C_{\beta_1}&0&0&0&0\\
C_{\gamma_2}&0&1&0&C_{\gamma_2}^{-1}&0\\
0&0&0&0&C_{\delta_1}&1\\
0&0&0&0&C_{\delta_2}&1
\end{vmatrix}}{\begin{vmatrix}
C_{\gamma_1}&0&1&0&C_{\gamma_1}^{-1}&0\\
p^{-1}q^{-1}&q^{-1}&pq^{-1}&p^{-1}&1&p\\
1&C_{\beta_1}&0&0&0&0\\
1&C_{\beta_2}&0&0&0&0\\
C_{\gamma_2}&0&1&0&C_{\gamma_2}^{-1}&0\\
0&0&0&0&C_{\delta_1}&1
\end{vmatrix}},
\end{align*}
where to evaluate at $\nu(\gamma_1)$, we use the basis $x_1,x_2$ from table (\ref{zzpathtable2}). Similarly, we compute
\begin{align*}
r_{W_2}&=-\frac{{\rm V}_{\bw_1 {\bf w}}}{{\rm V}_{\bw_2 {\bf w}}}(\nu(\alpha_1))=-\frac{\begin{vmatrix}
0&C_{\alpha_1}^{-1}&0&1&0&C_{\alpha_1}\\
p^{-1}q^{-1}&q^{-1}&pq^{-1}&p^{-1}&1&p\\
1&C_{\beta_1}&0&0&0&0\\
1&C_{\beta_2}&0&0&0&0\\
C_{\gamma_2}&0&1&0&C_{\gamma_2}^{-1}&0\\
0&0&0&0&C_{\delta_1}&1
\end{vmatrix}}{\begin{vmatrix}
0&C_{\alpha_1}^{-1}&0&1&0&C_{\alpha_1}\\
p^{-1}q&q&p^{-1}&1&p^{-1}q^{-1}&q^{-1}\\
1&C_{\beta_1}&0&0&0&0\\
C_{\gamma_2}&0&1&0&C_{\gamma_2}^{-1}&0\\
0&0&0&0&C_{\delta_1}&1\\
0&0&0&0&C_{\delta_2}&1
\end{vmatrix}}.
\end{align*}
It can be verified using computer algebra that $X_7=r_{W_1} r_{W_2}$.
\section{The small polygons}\label{smallpolysection}
In the remaining sections, we prove the results stated in Section \ref{sec2}. In order to invert the spectral transform, we want to first reconstruct the $Q_{{\bw} {\bf w}}$, the entries {of the ${\bf w}$-column} of the
adjugate matrix, from the spectral data. To do this, we need to first find the Newton polygon of the $Q_{{\bw} {\bf w}}$, which we call the small polygons and denote by $N_{{\bw} {\bf w}}$. Explicitly, $N_{{\bw} {\bf w}}$ is the convex hull of homology classes of dimer covers of $\Gamma-\{{\bw,{\bf w}}\}$. However, it appears difficult to describe $N_{{\bw} {\bf w}}$ in a direct combinatorial way. Instead, we will re-express the problem in terms of toric geometry. The key to doing this is an extension of the Kasteleyn matrix, which is a map of trivial sheaves on $\rm T$, to a map of locally free sheaves on a compactification of $\rm T$. We are led to consider a stacky toric surface $\mathscr X_N$ instead of the toric surface $X_N$, because such an extension does not exist on $X_N$ unless the polygon has only primitive sides.

The basics of stacky toric surfaces are recalled in detail in Appendix~\ref{A3}. For the convenience of the reader we reproduce some notation. 

Let $\Sigma$ be the normal fan of $N$. There is a \textit{stacky fan} $\boldsymbol \Sigma=(\Sigma,\beta)$ where
\begin{align*}
    \beta:\mathbb Z^{\Sigma(1)} &\ra {\rm M}^\vee,\\
    \delta_\rho &\mapsto |E_\rho| u_\rho,
\end{align*}
where $u_\rho$ is the primitive normal to $E_\rho$. We identify the set of rays  $\Sigma(1)$  of the fan $\Sigma$ with  the components $D_\rho$ of the divisor at infinity
$$
\rho \leftrightarrow \tau_{\rho}=\R_{\geq 0}u_\rho.
$$
We assign to $\boldsymbol{\Sigma}$ a smooth {\it toric {Deligne-Mumford} stack $\mathscr X_{N}$}, which contains the torus ${\rm T}$ as a dense open subset. 

 We consider the stack rather than the toric surface since we  construct an extension of the Kasteleyn operator to a compactification of the torus ${\rm T}$ in Lemma~\ref{main::lem}. There is no such extension on the toric surface when the Newton polygon is not simple, but there is one on the stack.

\subsection{Extension of the Kasteleyn operator}
Define for each black vertex $\bw$ the line bundle 
\[
\mathcal E_{\bw}:=\mathcal O_{\mathscr X_N}\Bigl({\dd}(\bw)-\sum_{\rho \in \Sigma(1)}\sum_{\alpha \in Z_\rho : \bw \in \alpha } \frac{1}{|E_{\rho}|}D_{\rho} \Bigr),
\]
and for each white vertex $\w$, the line bundle 
\[
\mathcal F_{\w}:=\mathcal O_{\mathscr X_N}({\dd}(\w)).
\]
Let 
$$
\mathcal E:=\bigoplus_{\bw \in B} \mathcal E_\bw, \quad   \quad \mathcal F:=\bigoplus_{\w \in W} \mathcal F_\w.
$$ 
They   are locally free sheaves of the same rank $\# B=\#W$ on $\mathscr X_N$.
\begin{proposition}\label{main::lem}
The Kasteleyn operator $K$ extends to a map of locally free sheaves on $\mathscr X_N$:
\be \la{MK}
\widetilde{K}: \mathcal E \ra \mathcal F.
\ee
\end{proposition}
\begin{proof}
By definition, 
\[
 K_{\w \bw}=\sum_{e \in E(\Gamma)\text{ incident to } \bw,\w} wt(e) \epsilon(e)\phi(e).
\]
We need to show that for any edge $e$ with vertices $\bw,\w$, the character $\phi(e)$ is a global section of 
\[
\mathcal H om_{\mathscr X_N}(\mathcal E_\bw,\mathcal F_\w) \cong \mathcal O_{\mathscr X_N}\Bigl({\dd}(\w)-{\dd}(\bw)+\sum_{\rho \in \Sigma(1)}\sum_{\alpha \in Z_\rho : \bw \in \alpha } \frac{1}{|E_{\rho}|}D_{\rho} \Bigr).
\]
{Let $m \in {\rm M}$ be such that $\phi(e)=\chi^m$ and let $D:= {\dd}(\w)-{\dd}(\bw)+\sum_{\rho \in \Sigma(1)}\sum_{\alpha \in Z_\rho : \bw \in \alpha } \frac{1}{|E_{\rho}|}D_{\rho}$. By Proposition~\ref{pro:globalsec}, $\chi^m$ is a global section of the line bundle $\mathcal O_{\mathscr X_N}(D)$ if and only if $m \in P_D \cap {\rm M}$. Using~(\ref{divisorpolygonbijection}), this is equivalent to showing that for every edge $e=\bw \w$, we have 
\[
\div \phi(e)+{\dd}(\w)-{\dd}(\bw)+\sum_{\rho \in \Sigma(1)}\sum_{\alpha \in Z_\rho : \bw \in \alpha } \frac{1}{|E_{\rho}|}D_{\rho} \geq 0,
\]
where $\div \phi(e)=\sum_{\rho \in \Sigma(1)} \langle m, u_\rho \rangle D_\rho$ as in (\ref{divchar}).}

 Let $\alpha, \beta$ be the zig-zag paths through $e$, with $\alpha \in Z_\sigma, \beta \in Z_\rho$. Then by Lemma \ref{lem:DAM}, we have
\[
{\dd}(\w)-\dd(\bw)=-\frac{1}{|E_{{\sigma}}|}D_{{\sigma}}-\frac{1}{|E_{{\rho}}|}D_{{\rho}}-\div  \phi(e). 
\]
This implies
\be \la{eq:tau}
\div  \phi(e)+\dd(\w)  -  \dd(\bw)+\sum_{\tau \in \Sigma(1)}\sum_{\gamma \in Z_\tau:\bw \in \gamma } \frac{1}{|E_{\tau}|}D_{\tau} 
 = \sum_{\tau \in \Sigma(1)} \sum_{\substack{\gamma \in Z_\tau : \bw \in \gamma \\ \gamma \neq \alpha,\beta}}^{}\frac{1}{|E_{\tau}|}D_{\tau} \geq 0. \qedhere
\ee
 
\end{proof}

{The small polygon $N_{\bw \w}$ is by definition the Newton polygon of $Q_{\bw \w}$. By Proposition~\ref{pro:globalsec}, this is equivalent to saying that $Q_{\bw \w}$ is a global section of a line bundle $\mathcal O_{\mathscr X_N}(\divebw_{\bw \w})$, where $\divebw_{\bw \w}$ is the divisor associated to $N_{\bw \w}$ by the correspondence (\ref{divisorpolygonbijection}). Now that we have shown that $K$ is a global section of $\mathcal Hom_{\mathscr X_N}(\mathcal E, \mathcal F)$, we can take exterior powers to find which line bundle $\mathcal O_{\mathscr X_N}(\divebw_{\bw \w})$ the minor $Q_{\bw \w}$ of $K$ is a global section of.
} 

Taking the determinant of the map (\ref{MK}), we see that $ \mathrm{det} \ \widetilde{K}$ is a global section of the line bundle  
\be \la{SH1}
 \mathcal H om_{\mathscr X_N}\Bigl(\bigwedge_{\bw \in B} \mathcal E_\bw,\bigwedge_{\w \in W} \mathcal F_\w \Bigr) \cong \mathcal O_{\mathscr X_N}\Bigl( \sum_{\w \in W } \dd(\w)-\sum_{\bw \in B }\Bigl(\dd(\bw)-\sum_{\rho \in \Sigma(1)}\sum_{\alpha \in Z_\rho:b \in \alpha } \frac{1}{|E_{\rho}|}D_{\rho}  \Bigr)\Bigr).
\ee

\begin{lemma}\la{detpoly}
  Let $D_N$ be the divisor associated to $N$ by the correspondence (\ref{divisorpolygonbijection}) between divisors and polygons. 
Then one has 
\be\la{tt}
 \sum_{\w \in W } \dd(\w)-\sum_{\bw \in B }\Bigl(\dd(\bw)-\sum_{\rho \in \Sigma(1)}\sum_{\alpha \in Z_\rho:b \in \alpha } \frac{1}{|E_{\rho}|}D_{\rho}  \Bigr)=D_N.
\ee
 Therefore, 
\[
\det\widetilde{K} \in H^0(\mathscr X_N, \mathcal O_{\mathscr X_N}(D_N)).
\]
\end{lemma}

\begin{proof}

 Let $a_\rho$ be the coefficient of $D_\rho$ in $D_N$. Let $(i_1,i_2)$ be a vertex of $P$ {contained in} $E_\rho$ and let $\dimer$ be {the associated} extremal dimer cover. We pair up black and white vertices in the sum according to $\dimer$:
$$
\sum_{e=\bw \w \in  \dimer}\Bigl(\dd(\w)-\dd(\bw)+\sum_{\rho \in \Sigma(1)}\sum_{\alpha \in Z_\rho : \bw \in \alpha} \frac{1}{|E_{\rho}|}D_{\rho}\Bigr).
$$
Now we observe that if $e$ is not contained in any zig-zag path in {$Z_\rho$}, then $D_\rho$ does not appear in the summand, and if $e$ is contained in a zig-zag path associated to $E_\rho$, then $D_\rho$ appears twice but with opposite signs, modulo contributions from intersections of edges with $\gamma_z,\gamma_w$. Therefore, there is no net contribution to the coefficient of $D_\rho$ except for the intersections of edges in $\dimer$ with $\gamma_z,\gamma_w$, which is the same as in
$$
-\sum_{e \in \dimer}\div \phi(e)=-\div z^{i_1} w^{i_2},
$$
which is $a_\rho$. Comparing with (\ref{SH1}),  we see that (\ref{tt}) implies the second statement.
\end{proof}


Now we consider the codimension 1 exterior power, where we remove $\{{ \bw},{\bf w}\}$. Let $\widetilde {Q}$ be the adjugate matrix of $\widetilde K$. Set 
\[
\divebw_{\bw \w}:=D_N-\dd(\w)+\dd(\bw)-\sum_{\rho \in \Sigma(1)}\sum_{\alpha \in Z_\rho : \bw \in \alpha } \frac{1}{|E_{\rho}|}D_{\rho}.
\]
\begin{corollary} \la{Cor3.5}
$\widetilde { Q}_{\bw \w} \in H^0(\mathscr X_N, \mathcal O_{\mathscr X_N}(\divebw_{\bw \w}))$. 
\end{corollary}
We therefore arrive at the definition of the small polygon $N_{\bw \w}$ given in Definition \ref{SNP1} {by the correspondence (\ref{divisorpolygonbijection})}.

\subsection{Points at infinity} \la{sec:paiinf}

In this section, we prove that the points at infinity of $\spectralcurve$ are as described in Section~\ref{sec:cas}. {We use the notations $U_\Sigma \subset \C^{\Sigma(1)}$ and the standard coordinates $(z_\rho)$ on $\C^{\Sigma(1)}$ from Appendix~\ref{A3}. 
		The toric variety $X_N$ is the quotient $U_\Sigma/H$, where $H$ is the kernel of the map $(\C^\times)^{\Sigma(1)} \rightarrow {\rm T}$ sending $(z_\rho)_{\rho \in \Sigma(1)}$ to $(\prod_{\rho} z_\rho^{\langle (1,0) , u_\rho \rangle},\prod_\rho z_\rho^{\langle (0,1) , u_\rho \rangle} )$. There is a canonical map $\pi:U_\Sigma \ra X_N$ given by sending $(z_\rho)$ to $H \cdot (z_\rho^{|E_\rho|})$ which induces the coarse moduli space map $\pi:\mathscr X_N \rightarrow X_N$. The spectral curve $\spectralcurve$ is cut out by the section $P=P(z,w)$ of $\mathcal O_{X_N}(D_N)$. The pullback $\pi^*P$ defines a section of $\mathcal O_{\mathscr X_N}(D_N)$ which is a $G$-invariant section of $\mathcal O_{U_\Sigma}$, so it vanishes on a $G$-invariant subvariety $\spectralcurve_{U_\Sigma}$. Each point at infinity of $\spectralcurve$ corresponds to a $G$-invariant set of points at infinity of $\spectralcurve_{U_\Sigma}$, so we will determine the points at infinity of $\spectralcurve$ from the points at infinity of $\spectralcurve_{U_\Sigma}$. By Lemma~\ref{detpoly}, $\pi^*P = \det \widetilde K$ so the points at infinity of $\spectralcurve_{U_\Sigma}$ are obtained by setting $z_\rho = \det \widetilde K= 0$ for $\rho \in \Sigma(1)$.
}

 From (\ref{eq:tau}) and Proposition~\ref{pro:globalsec}, {for $e=\bw \w$}, we get that $\phi(e)$ corresponds to the $G$-invariant section of $\mathcal O_{U_\Sigma}$ given by
\be \la{phiestack}
\phi(e) = \prod_{\tau \in \Sigma(1)} \prod_{\substack{\gamma \in Z_\tau : \bw \in \gamma \\ \gamma \neq \alpha,\beta}}^{}z_\tau.
\ee
{The divisor $D_\rho$ in $X_N$ corresponds to 
$\{z_\rho=0\} \subset U_\Sigma$.} $\phi(e)$ vanishes on $\{z_\rho=0\}$ precisely when there is a zig-zag path $\alpha \in Z_\rho$ such that $\bw$ is contained in $\alpha$ but $\w$ is not contained in $\alpha$. This implies that when restricted to $\{z_\rho=0\}$, after reordering the black and white vertices, the extended Kasteleyn operator $\widetilde K$ takes a block-upper-triangular form 
$$
\begin{pmatrix}
\restr{\widetilde K}{\alpha_1} &  & & &*\\
 & \restr{\widetilde K}{\alpha_2} &  & &*\\
 & & \ddots & & \vdots \\
 &  &  &\restr{\widetilde K}{\alpha_n} & \ast \\
&& & &\restr{\widetilde K}{\Gamma-\{\alpha_1,\dots,\alpha_n\}}
\end{pmatrix},
$$
where $Z_\rho=\{\alpha_1,\dots,\alpha_n\}$, $\restr{\widetilde K}{\alpha_i}$ is the restriction of $\widetilde K$ to the black and white vertices in $\alpha_i$, and the $*$'s denote some possibly nonzero blocks whereas {any block that has not been indicated is zero. In particular, the nonzero blocks $*$ are only in the last column (in these blocks, $\bw$ is not in $\alpha$ but $\w$ is in $\alpha$). Note that the zig-zag paths $\alpha_1,\dots,\alpha_n$ do not share any vertices because of minimiality since otherwise we would have a parallel bigon, so the blocks $\restr{\widetilde K}{\alpha_1},\dots,\restr{\widetilde K}{\alpha_n}$ do not overlap.}

{If $\alpha \in Z_\rho$ is $\bw_1 \ra \w_1 \ra \bw_2 \ra \cdots \ra \w_d \ra \bw_1$, the determinant of the block $\restr{\widetilde K}{\alpha}$ is }
\begin{align*}
\restr{\det \widetilde K}{\alpha}&= 
 \det\begin{pmatrix} 
  \widetilde{K}_{\w_1 \bw_1}  &&&& \widetilde{K}_{\w_d \bw_1}\\
   \widetilde{K}_{ \w_1 \bw_2}   & \widetilde{K}_{ \w_2 \bw_2} & &  \\
 &  \widetilde{K}_{ \w_2 \bw_3}   \\
&&\ddots&\ddots \\
   &  & &  \widetilde{K}_{ \w_{d-1} \bw_d}     & \widetilde{K}_{ \w_d \bw_d} 
    \end{pmatrix}\\
    &= \prod_{i=1}^d \widetilde{K}_{ \w_i \bw_i} - (-1)^{d} \prod_{i=1}^d \widetilde{K}_{ \w_{i-1} \bw_i}\\
    &=-\prod_{i=1}^d\left(wt(\bw_i \w_{i-1})\epsilon(\bw_i \w_{i-1})\phi(\bw_i \w_{i}) \right) \left(\prod_{i=1}^d \frac{\phi(\bw_i \w_{i-1})}{\phi(\bw_i \w_{i})}-C_\alpha \right),
\end{align*}
where we have used the definition of the Kasteleyn matrix (see (\ref{edgeph}), (\ref{Kastdet}). {Plugging in (\ref{phiestack}) and using the fact that $\alpha$ intersects a zig-zag path $\beta \in Z_\tau$ $\langle [\alpha],u_\tau \rangle$ times, we get
\[
\prod_{i=1}^d \frac{\phi(\bw_i \w_{i-1})}{\phi(\bw_i \w_{i})} = \prod_{\tau \in \Sigma(1)}z_\tau^{-|E_\tau|\langle [\alpha], u_\tau \rangle}.
\]
Therefore, the points at infinity of $\spectralcurve_{U_\Sigma}$ are given by setting $z_\rho=0$ and $\prod_{\tau \in \Sigma(1)}z_\tau^{-|E_\tau|\langle [\alpha], u_\tau \rangle} = C_\alpha$. The point at infinity of $\spectralcurve$ is the point obtained by applying $\pi$ to any of these points. From the definition of $\pi$, we get that $\prod_{\tau \in \Sigma(1)}z_\tau^{-|E_\tau|\langle [\alpha], u_\tau \rangle} = \pi^* \chi^{-[\alpha]}$, so the point at infinity of $\spectralcurve$ is given by 
 \be \la{caspoint}
\chi^{-[\alpha]}= C_\alpha.
\ee
}

\section{Behaviour of the Laurent polynomial  \texorpdfstring{${ Q}_{\bw \w}(z,w)$}{Qbw} at infinity}\la{extensionsection}
We proved in Corollary~\ref{Cor3.5} that the Laurent polynomial ${ Q}_{\bw \w}(z,w)$ lies in the finite dimensional vector space $H^0(\mathscr X_N, \mathcal O_{\mathscr X_N}(\divebw_{\bw \w}))$. We need some additional constraints on ${Q}_{\bw \w}(z,w)$ to determine it. Corollary \ref{cordeg} provides  $g$ linear equations for the coefficients of ${ Q}_{\bw \w}(z,w)$ coming from the vanishing of ${Q}_{\bw \w}(z,w)$ at the $g$ points of the divisor $S_\w$. We obtain additional equations from the behaviour of ${ Q}_{\bw \w}(z,w)$ at the points at infinity of the spectral curve, which we study in this section.

Recall that $X_N$ is the toric surface associated to $N$ compactifying $T$. The restriction of the Kasteleyn operator to the open spectral curve $\spectralcurve^\circ$ is a map of trivial sheaves: 
\[
\restr{K}{\spectralcurve^\circ}: \bigoplus_{ \bw \in B} {\mathcal O}_{\spectralcurve^\circ}  \longrightarrow   \bigoplus_{ w \in W} \mathcal O_{\spectralcurve^\circ}.
\]
{Recall the correspondence between divisors $D$ and invertible sheaves with rational sections $(\mathcal L,s)$: given a divisor $D$, the corresponding invertible sheaf $\mathcal L=\mathcal O_\spectralcurve(D)$ is defined on an open $U$ by
	\[
H^0(U,\mathcal O_\spectralcurve(D)) := \{ t  \in K(\spectralcurve)^\times : \restr{(\div t + D)}{U} \geq 0\} \cup \{0\},	
	\]
with the obvious restriction maps, where $K(\spectralcurve)^\times$ denotes the nonzero rational functions on $\spectralcurve$. The rational section $s$ corresponds to the rational function $1$. On the other hand, given $(\mathcal L,s)$, we obtain $D$ as the divisor $\div s$. Moreover, there is a correspondence between rational functions $t$ on $\spectralcurve$ and rational sections $\overline t$ of $\mathcal L$ given by $t \mapsto \overline t:=s t$. In particular, \be \la{eq:divttbar}
\div \overline t = \div t + \div s = \div t+D, 
\ee
and so $\overline t$ is regular if and only if $\div t + D \geq 0$.
 }

{A similar proof to Proposition~\ref{main::lem} shows that the Kasteleyn matrix $K$, which is a matrix of rational functions on $\spectralcurve$, defines a regular map} $\overline K$ of locally free sheaves on $\spectralcurve$ extending $\restr{K}{\spectralcurve^\circ}$, providing an exact sequence
\be \la{eq:kextc}
0 \ra \mathcal M \ra  \bigoplus_{ \bw \in B} \mathcal O_\spectralcurve \Bigl({\bf d}( \bw)-\sum_{ \alpha \in Z:\bw \in \alpha} \nu(\alpha)\Bigr) \xrightarrow[]{\overline K} \bigoplus_{ \w \in W} \mathcal O_\spectralcurve({\bf d}(\w))  
\ra \mathcal L \ra 0,
\ee
where $\mathcal M$ and $\mathcal L$ are the kernel and cokernel of the map $\overline K$ respectively. {When we say $\overline K$ is regular, we mean that each entry $\overline K_{\w \bw}$ is a regular section of the corresponding $\mathcal Hom$ line bundle \[\mathcal Hom_{\mathcal O_\spectralcurve} \Bigl(\mathcal O_\spectralcurve ({\bf d}( \bw)-\sum_{ \alpha \in Z:\bw \in \alpha} \nu(\alpha)), \mathcal O_\spectralcurve({\bf d}(\w))\Bigr).\]} 
For generic dimer weights, $\Sigma$ is smooth, and $\mathcal M$ and $\mathcal L$ are line bundles (so $\overline Q$ has rank $1$). Let $\overline s_{ \bw}$ and $\overline s_{ \w}$ be sections of 
$$\mathcal M^\vee \otimes \mathcal O_\spectralcurve\Bigl({\bf d}( \bw)-\sum_{\alpha \in Z: \bw \in \alpha} \nu(\alpha)\Bigr)~\text{and}~\mathcal L \otimes \mathcal O_\spectralcurve({\bf d} (\w))^\vee$$
  respectively, given by the  $\bw$-entry of the kernel map and  $\w$-entry of the cokernel map respectively. Since $\overline {Q}$ has rank $1$, we have $\overline{Q}_{\bw \w}=\overline s_\bw \overline s_\w$. Denote by  $S_\bw$ and $S_\w$ the 
   effective divisors on the open spectral curve $\spectralcurve^\circ$ given by vanishing of the $\bw$-row and $\w$-column of $Q$ respectively, {or equivalently, the vanishing of $\overline s_\bw$ and $\overline s_\w$ respectively.} 
\begin{lemma}\label{divbw}
\begin{align*}
    \div_{\spectralcurve}\overline s_\bw &= S_ \bw+\sum_{\alpha \in Z: \bw \notin \alpha} \nu(\alpha),\\
    \div_{\spectralcurve}\overline s_\w &= S_\w,
\end{align*}
\end{lemma}

\begin{proof}
By the definition, ${\restr{(\div_{\spectralcurve}\overline s_ \bw)}{\spectralcurve^\circ}}=S_\bw$ and ${\restr{(\div_{\spectralcurve}\overline s_ \w)}{\spectralcurve^\circ}}=S_\w$, so it only remains to find their orders of vanishing at infinity. 

{Let $U \subset \spectralcurve$ be a neighbourhood of $\nu(\alpha)$ that does not contain any other point at infinity.} Let $u$ be a local parameter {in $U$} that vanishes to order $1$ at $\nu(\alpha)$ and nowhere else. {When restricted to $U$, each of the line bundles in the source and target of $\overline K$ in  (\ref{eq:kextc}) is of the form $\mathcal O_U(k \nu(\alpha))$ for some $k \in \Z$. We trivialize $\mathcal O_U(k \nu(\alpha))$ as follows:
	\begin{align*}
	\mathcal O_U(k \nu(\alpha)) &\xrightarrow[]{\cong} \mathcal O_U\\
	f &\mapsto u^{k} f.
	\end{align*}

}
 Let us order the black and white vertices so that the vertices on $\alpha$ come first. Then {in $U$}, we have 
$$
\overline{K}= \begin{pmatrix} 
K_1 & B \\
u A  & K_2 
\end{pmatrix}+O(u),
$$
where $K_1, K_2$ are the restrictions of $\overline{K}$ to $\alpha$ and $\Gamma - \alpha$ respectively. {Since corank $\overline{K}=1$ and since we know corank $K_1>0$ from the computation of the determinant in Section~\ref{sec:paiinf}, we have corank $K_1=1$ and that $K_2$ is invertible.} Let $v \in \ker K_1$. Then,
$$
\ker\overline K= (v,-u K_2^{-1} A v)+O(u).
$$
{If any entry in $v$ or $K_2^{-1} A v$ is $0$, then it means that some $\overline s_\bw$ is identically $0$. Let $\overline Q$ denote the adjugate matrix of $\overline K$. Since $\overline {Q}$ has rank $1$, we have $\overline{Q}_{\bw \w}=\overline s_\bw \overline s_\w = 0$, so we will have $\overline Q_{\bw \w}=0$ for all $\w \in W$. On the other hand, $Q_{\bw \w}$ is the signed partition function for dimer covers of $\Gamma \setminus \{\bw ,\w\}$, so if we choose $\w$ such that $\bw \w$ is an edge of $\Gamma$ used in a dimer cover, then $Q_{\bw \w} \neq 0$ for generic dimer weights, a contradiction. Therefore, the entries of $\ker \overline K$ are nonzero when $u \neq 0$, so $\overline s_\bw$ has a simple zero at $\nu(\alpha)$ for all $\bw \not\in \alpha$ and has no zeroes or poles for $\bw \in \alpha$. }

Similarly, let $v' \in \ker K_1^*$. We have 
$$
\ker\overline K^*= (v',-(K_2^*)^{-1} B v')+O(u).
$$
For generic dimer weights, none of the entries of $(K_2^*)^{-1} B v'$ can vanish, so $\overline s_\w$ has no zeroes or poles at $\nu(\alpha)$.
\end{proof}

\begin{corollary}\label{divQbw}
$\div_{\spectralcurve} Q_{\bw\w}=S_\bw+S_\w-\restr{D_N}{\spectralcurve}+{\bf d}(\w)-{\bf d}(\bw)+\sum_{\alpha \in Z}\nu(\alpha).$
\end{corollary}
\begin{proof} Let $\overline Q$ denote the adjugate matrix of $\overline K$. Since $\overline {Q}$ has rank $1$, we have $\overline{Q}_{\bw \w}=\overline s_\bw \overline s_\w$, so that
$$
\div_{\spectralcurve}\overline Q_{\bw \w}=S_\bw+S_\w+\sum_{\alpha \in Z: \bw \notin \alpha} \nu(\alpha).
$$
{A computation similar to Corollary~\ref{Cor3.5} shows that \[\overline Q_{\bw \w} \in H^0(\spectralcurve,\mathcal O_\spectralcurve(-\restr{D_N}{\spectralcurve}+{\bf d}(\w)-{\bf d}(\bw)+\sum_{\alpha \in Z:\bw \in \alpha}\nu(\alpha))).\]}
{$Q_{\bw \w}$ is the rational function corresponding to the rational section $\overline Q_{\bw \w}$. Therefore, using (\ref{eq:divttbar}), we have }
\begin{align*}
\div_{\spectralcurve}Q_{\bw\w}&=\div_{\spectralcurve} \overline{Q}_{\bw \w}-\restr{D_N}{\spectralcurve}+{\bf d}(\w)-{\bf d}(\bw)+\sum_{\alpha \in Z:\bw \in \alpha}\nu(\alpha)\\
&=S_\bw+S_\w-\restr{D_N}{\spectralcurve}+{\bf d}(\w)-{\bf d}(\bw)+\sum_{\alpha \in Z}\nu(\alpha). \qedhere
\end{align*}
\end{proof}

\begin{corollary}\label{cordeg}
We have for all $\bw \in B,\w \in W$, $\deg S_\bw=\deg S_\w=g$, where $g$ is the genus of $\spectralcurve$.
\end{corollary}

\begin{proof}
{Let $\omega_*$ denote the canonical divisor of $*$.} We have $\omega_{X_N} =-\sum_{\rho \in \Sigma(1)} D_\rho$ {(\cite{CLS}*{Theorem 8.2.3})}. By the adjunction formula, we get $\omega_{\spectralcurve} = \restr{D_N}{\spectralcurve} -\sum_{\alpha \in Z} \nu(\alpha)$. Since $Q_{\bw\w}$ is a rational function on $\spectralcurve$, we have $\text{deg div}_{\spectralcurve} Q_{\bw\w} =0$. Since $\deg ({\bf d}(\w)-{\bf d}(\bw))=-2$ and $\deg\omega_{\spectralcurve}=2g-2$, we get $\deg(S_\bw+S_\w)=2g$. By symmetry under interchanging $B$ and $W$, we get $\deg S_\bw=\deg S_\w=g$.
\end{proof}
Recall that the number $g$ is also equal to the number of interior lattice points in $N$ for generic $\spectralcurve \in |D_N|$.
\begin{proposition}\label{degl}
The line bundle $\mathcal L $ is isomorphic to $\mathcal O_{\spectralcurve}(S_\w+{\bf d}(\w))$ for any $\w \in W$. It has degree $g-1$.
\end{proposition}
\begin{proof}
By Lemma \ref{divbw}, $\overline s_\w$ is a section of $\mathcal L \otimes \mathcal O_\spectralcurve({\bf d}(\w ))^\vee$ with divisor $S_\w$. Therefore, we must have 
\[
 \mathcal L \otimes \mathcal O_\spectralcurve({\bf d}(\w ))^\vee \cong \mathcal O_\spectralcurve\left(S_\w\right),
\]
which implies $\mathcal L \cong \mathcal O_{\spectralcurve}(S_\w+{\bf d}(\w))$. Since $\deg S_\w=g$ and ${\rm{deg}}~{\bf d}(\w)=-1$, we get $\deg\mathcal L=g-1$.
\end{proof}

\section{Equations for the Laurent polynomial \texorpdfstring{$Q_{{\bw}{\bf w}}$}{Qbw}}\label{laurentsection}

Since $Q_{{\bw}{\bf w}}$ has Newton polygon $N_{\bw {\bf w}}$, we have 
\[
Q_{{\bw}{\bf w}}=\sum_{m \in N_{\bw {\bf w} } \cap {\rm M}}a_m \chi^m,
\]
for some $a_m \in \C$.
We know that $Q_{{\bw}{\bf w}}$ vanishes on $S_{\bf w}$, which gives $g$ linear equations among the $(a_m)_{m \in N_{\bw {\bf w} } \cap {\rm M}}$. However these $g$ linear equations are not usually sufficient to determine $Q_{{\bw}{\bf w}}$, so we need to find some additional equations. These additional equations will come from the vanishing of $Q_{{\bw}{\bf w}}$ at the points at infinity.

\subsection{{Additional linear equations for \texorpdfstring{$Q_{{\bw}{\bf w}}$}{Qbw}}}
 The fact that the Newton polygon of $Q_{{\bw}{\bf w}}$ is the small polygon $N_{{\bw}{\bf w}}$ imposes  certain inequalities on the order of vanishing of $Q_{{\bw}{\bf w}}$ at points at infinity of $\spectralcurve$. Corollary \ref{divQbw} imposes additional constraints that are linear equations in the coefficients of $Q_{{\bw}{\bf w}}$. Inverting this linear system gives $(a_m)_{m \in N_{\bw {\bf w} } \cap {\rm M}}$ and therefore $Q_{{\bw}{\bf w}}$. 

We now give the precise statement. For a $\Q$-divisor $D = \sum_{\rho \in \Sigma(1)} b_\rho D_\rho$, we define a ($\Z$-) divisor $\floor{D}:=\sum_{\rho \in \Sigma(1)} \floor{b_\rho} D_\rho$, where 
$\floor{x}$ is the largest integer $n$ such that $n \leq x$. It is the pushforward of $D$ by the canonical projection $\mathscr X_N \ra X_N$.
\begin{proposition}\label{lemextraeq}
The extra linear equations for $(a_m)_{m \in N_{\bw {\bf w} } \cap {\rm M}}$ from vanishing of $Q_{{\bw}{\bf w}}$ at points at infinity correspond to the points in 
\be \label{extraeq}
-\restr{D_N}{\spectralcurve}+{\bf d}({\bf w})-{\bf d}(\bw)+\sum_{\alpha \in Z} \nu(\alpha) + \restr{\floor{{\divebw_{{\bw }{\bf w}}}}}{\spectralcurve}.
\ee 
\end{proposition}
\begin{proof}
A generic Laurent polynomial $F$ of the form $\sum_{m \in N_{\bw {\bf w} } \cap {\rm M}}a_m \chi^m$ has order of vanishing
$$
\div_{\spectralcurve} \restr{F}{\spectralcurve} \geq -\restr{\floor{\divebw_{{\bw }{\bf w}}}}{\spectralcurve}
$$
at the points at infinity of $\spectralcurve$. From Corollary \ref{divQbw}, we have that $\div_{\spectralcurve} Q_{\bw {\bf w}}=S_\bw+S_{\bf w}-\restr{D_N}{\spectralcurve}+{\bf d}({\bf w})-{\bf d}(\bw)+\sum_{\alpha \in Z} \nu(\alpha).$ The discrepancy provides the extra equations.
\end{proof}

Now we describe these extra linear equations explicitly. Suppose $\alpha \in Z_\rho$ is a zig-zag path that contributes a linear equation. We extend $[\alpha]$ to a basis $([\alpha]=x_1,x_2)$ of $\rm M$ such that $\langle x_2,u_\rho \rangle =1$, so that for any $m \in {\rm M}$, we can write
\[
    \chi^m =  x_1^{b_m} x_2^{c_m}, ~~~~b_m,c_m \in \Z.
    \]
    Let $N_{\bw {\bf w}}^\rho$ be the set of lattice points in $N_{\bw {\bf w}}$ closest 
to the edge $E_{\rho}$ of $N$ i.e., the set of points in $N_{\bw {\bf w}}$ that minimize the functional $\langle *,u_\rho \rangle$. 

 \begin{proposition}\label{caseqn}
 Suppose $Q_{{\bw}{\bf w}}=\sum_{m \in N_{\bw {\bf w} } \cap {\rm M}}a_m \chi^m$ {and suppose $\alpha \in Z_\rho$ is a zig-zag path that contributes a linear equation.}
 Then, the linear equation given by $\alpha$ is: \[
 \sum_{m \in N^\rho_{\bw {\bf w} } \cap {\rm M}}a_m C_\alpha^{-b_m}=0.
\]
 \end{proposition}
 
 \begin{proof}
 The affine open variety in $X_N$ corresponding to the cone $\rho$ is 
 \[
 U_\rho=\spec \C[x_1^{\pm 1 },x_2] \cong \C^\times \times \C,
 \]
 and $D_\rho \cap U_\rho$ is defined by $x_2=0$.
 
A generic curve $\spectralcurve$ meets $D_\rho$ transversely at $\nu(\alpha)$, and therefore we may take $x_2$ as a uniformizer of the local ring $\mathcal O_{\spectralcurve, \nu(\alpha)}$ at $\nu(\alpha)$. For each $m \in N^\rho_{\bw {\bf w} } \cap {\rm M}$, we have
$$
   \chi^m = x_1^{b_m} x_2^p, ~~~~b_\gamma,p \in \Z, 
    $$ 
where $p$ is the same for all of them, and is the coefficient of $\nu(\alpha)$ in $-\restr{[{E_{{\bf b}w}}]}{\spectralcurve}$. Then using $x_1^{-1}=C_\alpha$ at $\nu(\alpha)$, we have
\be \label{qbweqn}
Q_{{\bw}{\bf w}}=  \left(\sum_{m \in N^\rho_{\bw {\bf w} } \cap {\rm M}} a_m C_\alpha^{-b_m} \right)x_2^p+O(x_2^{p+1}).
\ee
Since $\alpha$ contributes a linear equation, (\ref{qbweqn}) must vanish at order $x_2^p$, {so $ \sum_{m \in N^\rho_{\bw {\bf w} } \cap {\rm M}}a_m C_\alpha^{-b_m}=0$.} 
 \end{proof}

\subsection{The system of linear equations \texorpdfstring{$\mathbb V_{{\bw}{\bf w}}$}{Vbw}}

{
	Recall from Section~\ref{sec2} the system of linear equations $\mathbb V_{{\bw}{\bf w}}$. These are linear equations in the variables $(a_m)_{m \in N_{\bw {\bf w} } \cap {\rm M}}$. Recall also that }the matrix $\mathbb V_{{\bw}{\bf w}}$ is defined such that these equations are given by $$\mathbb V_{{\bw}{\bf w}}(a_m)=0.$$ It is not necessarily a square matrix. However, we have:

\begin{proposition} \la{prop:uniq}
For generic spectral data, $Q_{{\bw}{\bf w}}$ is the unique solution of the linear system of equations $\mathbb V_{{\bw}{\bf w}}$ modulo scaling.
\end{proposition}
\begin{remark}
 \begin{enumerate}
 \item {While the definition of $\mathbb V_{{\bw}{\w}}$ makes sense for all $\w \in W$,  Proposition \ref{prop:uniq} only holds when $\w = {\bf w}$ since $(p_i,q_i)_{i=1}^g$ depends on $\bf w$}.
 \item For generic spectral data, the equations (\ref{eq1}) are linearly independent, but the equations (\ref{Cas}) may not be.
  \end{enumerate}
\end{remark}
The rest of this {section} is devoted to the proof of Proposition \ref{prop:uniq}. 
Consider following the exact sequence on $X_N$, obtained by tensoring the closed embedding exact sequence of $i:\spectralcurve \hookrightarrow X_N$ by $\mathcal O_{X_N}(\floor{\divebw_{\bw{\bf w}}})$.
$$
0 \ra \mathcal O_{X_N}(\floor{\divebw_{\bw{\bf w}}}-D_N) \ra \mathcal O_{X_N}(\floor{\divebw_{\bw{\bf w}}}) \ra i_* \restr{\mathcal O_\spectralcurve(\floor{\divebw_{\bw{\bf w}}}}{\spectralcurve}) \ra 0.
$$
The following is a portion of the long exact sequence of cohomology.
\be \la{coh}
0 \ra H^0(X_N,\floor{\divebw_{\bw{\bf w}}}-D_N) \ra H^0(X_N,\floor{\divebw_{\bw{\bf w}}}) \ra H^0(\spectralcurve,\restr{\floor{\divebw_{\bw{\bf w}}}}{\spectralcurve}).
\ee
We need the following technical lemma.
\begin{lemma}\la{cohlem}
The restriction map $H^0(X_N,\floor{\divebw_{\bw{\bf w}}}) \ra H^0(\spectralcurve,\restr{\floor{\divebw_{\bw{\bf w}}}}{\spectralcurve})$ is injective.
\end{lemma}
\begin{proof}
If $\chi^m \in H^0( X_N,\floor{\divebw_{\bw{\bf w}}}-D_N)$, then $\div\chi^m +\floor{\divebw_{\bw{\bf w}}}-D_N \geq 0$. This implies that 
\begin{align}\la{inter}
\div\chi^m +\divebw_{{\bw}{\bf w}}-D_N=\sum_{\rho \in \Sigma(1)}\sum_{\alpha \in Z_\rho} \langle m,u_\rho \rangle  \frac{D_{\rho}}{|E_{\rho}|}- \dd({\bf w})+\dd(\bw)-\sum_{\rho \in \Sigma(1)}\sum_{\alpha  \in Z_\rho:{\bw} \in \alpha }\frac{1}{|E_\rho|}D_\rho \geq 0.
\end{align} 
Let $\gamma$ be a cycle in $\T$ with homology class $m$. The total number of signed intersections of $\gamma$ with all zig-zag paths is zero. {This number is the sum of the coefficients of $\sum_{\rho \in \Sigma(1)}\sum_{\alpha \in Z_\rho} \langle m,u_\rho \rangle  \frac{D_{\rho}}{|E_{\rho}|}$.} Let $\w'$ be any white vertex adjacent to $\bw$. Then we have
$$
-\dd({\bf w})+\dd(\bw)-\sum_{\rho \in \Sigma(1)}\sum_{\alpha  \in Z_\rho:{\bw} \in \alpha }\frac{1}{|E_\rho|}D_\rho = (\dd(\w')-\dd({\bf w}))-\sum_{\rho \in \Sigma(1)}\sum_{\substack{\alpha \in Z_\rho: { \bw} \in \alpha\\  \bw \w' \notin \alpha  }}\frac{1}{|E_\rho|}D_\rho.
$$
{The sum of the coefficients of} $\dd(\w')-\dd({\bf w})$ {is} the signed number of intersections with zig-zag paths of any path in $R$ from ${\bf w}$ to $\w'$, {which is also $0$}. Since the {coefficients of the} last term $-\sum_{\rho \in \Sigma(1)}\sum_{\substack{\alpha \in Z_\rho: { \bw} \in \alpha\\  \bw \w' \notin \alpha  }}\frac{1}{|E_\rho|}D_\rho$ {are} strictly negative, the sum in (\ref{inter}) cannot be non-negative. Therefore, $H^0(X_N,\floor{\divebw_{\bw{\bf w}}}-D_N)=0$, which by (\ref{coh}) means that the map $H^0(X_N,\floor{\divebw_{\bw{\bf w}}}) \ra H^0(\spectralcurve,\restr{\floor{\divebw_{\bw{\bf w}}}}{\spectralcurve})$ is injective.
\end{proof}

\begin{proof}[Proof of Proposition \ref{prop:uniq}]
\begin{enumerate}
    \item Existence: By Theorem 7.3 of \cite{GK12},
   the map  $\kappa_{\Gamma,{\bf w}}$ is dominant. So a generic spectral data is in the image of $\kappa_{\Gamma,{\bf w}}$. For such a spectral data, $Q_{{\bw}{\bf w}}$ satisfies:
   \begin{enumerate}
       \item The system of equations (\ref{eq1}) because, by definition of the spectral transform, $Q_{{\bw}{\bf w}}$ vanishes at the points of the divisor $S=\sum_{i=1}^g (p_i,q_i)$.
       \item The equations (\ref{Cas}) by Proposition \ref{caseqn}.
   \end{enumerate}

   \item Uniqueness: Suppose $V_{{\bw}{\bf w}}$ is a solution of $\mathbb V_{{\bw}{\bf w}}$. {Since $V_{{\bw}{\bf w}}$ has Newton polygon $N_{\bw {\bf w}}$, we have $\div_{\spectralcurve} \restr{F}{\spectralcurve} \geq -\restr{\floor{\divebw_{{\bw }{\bf w}}}}{\spectralcurve}$ as in the proof of Proposition~\ref{lemextraeq}. The additional equations in Proposition~\ref{lemextraeq} then imply that 
}
   $$
   \div_{\spectralcurve}V_{{\bw}{\bf w}} \geq S+D,
   $$
   where $D:=-\restr{D_N}{\spectralcurve}+{\bf d}({\bf w})-{\bf d}(\bw)+\sum_{\alpha \in Z} \nu(\alpha)$ satisfies $\deg D = -2g$.
 Therefore, $\restr{V_{{\bw}{\bf w}}}{\spectralcurve}$ can be identified with a section of $\mathcal O_{\spectralcurve}(-D)$ vanishing at the points of $S$. {Let $\omega_\spectralcurve$ denote the canonical divisor of $\spectralcurve$ as in Section~\ref{extensionsection}.} By the Riemann-Roch theorem,
   $$
   h^0(\spectralcurve,\mathcal O_{\spectralcurve}(-D)-h^1(\spectralcurve,\mathcal O_{\spectralcurve}(-D))=\deg(-D)-g+1=g+1.
   $$
   By Serre duality, $h^1(\spectralcurve,\mathcal O_{\spectralcurve}(-D))=h^0(\spectralcurve, \omega_{\spectralcurve}(D))$, which equals $0$ since $\omega_{\spectralcurve}(D)$ has negative degree $-2$. For generic $S$ that avoids the base locus of $\mathcal O_{\spectralcurve}(-D)$, the requirement that the section of $\mathcal O_{\spectralcurve}(-D)$ corresponding to $V_{{\bw}{\bf w}}$ vanishes at each of the $g$ points of $S$ imposes $g$ independent conditions, and therefore determines $\restr{V_{{\bw}{\bf w}}}{\spectralcurve}$ uniquely up to multiplication by a nonzero complex number. By Lemma \ref{cohlem}, $V_{{\bw}{\bf w}}$ is unique up to multiplication by a nonzero complex number. \qedhere
\end{enumerate}
\end{proof}
\begin{remark}
It is easy to see using Riemann-Roch that the number of equations in $\mathbb V_{{\bw}{\bf w}}$ is equal to $h^0(\spectralcurve,\restr{\floor{\divebw_{\bw{\bf w}}}}{\spectralcurve})-1$. On the other hand, the number of variables is $h^0(X_N,{\floor{\divebw_{\bw{\bf w}}}}{})$. However, the map in Lemma \ref{cohlem} is not necessarily an isomorphism (there may be sections on the curve that are not restrictions of sections on the surface), so we only have the inequality
\[
\#~\text{equations in}~\mathbb V_{{\bw}{\bf w}} \geq \#~\text{variables}-1.
\]
\end{remark}

\appendix
\section{Toric geometry} \la{A}
{
In \ref{toricv} and \ref{A2}, we give a brief background on toric varieties; further details can be found in the books \cite{Fulton} and \cite{CLS}.
}
\subsection{Toric varieties}\la{toricv}

A \textit{toric variety} $X$ over $\C$  is an algebraic variety containing the complex algebraic torus ${\rm T} \cong (\C^\times)^n$  as a Zariski open subset, such that the action of ${\rm T}$ on itself extends to an action of ${\rm T}$ on $X$. 

Let ${\rm M}$ be a lattice,   and let ${\rm M}^\vee:=\text{Hom}_\Z({\rm M},\Z)$ denote the dual lattice. Let  
${\rm T}:={\rm M}^\vee\otimes \C^\times = {\rm Hom}({\rm M}, \C^\times)$  
 be the complex algebraic torus   with the   lattice of characters ${\rm M}$. 
We denote by  $\chi^m:{\rm T} \ra \C^\times $  the character  associated to $m \in {\rm M}$. 
Let $\langle *, *\rangle$ be the pairing between ${\rm M} $ and ${\rm M}^\vee$. 
In our case   ${\rm M} = H_1(\T,\Z){\cong\Z^2}$,  so ${\rm M}^\vee = H^1(\T,\Z){\cong\Z^2}$ and ${\rm T} = H^1(\T,\C^\times){\cong(\C^\times)^2}$. {We have $\chi^{(i,j)}(z,w) = z^iw^j$}.

A \textit{fan} $\Sigma$ is a collection of cones in the real vector space ${\rm M}^\vee_\R:= {\rm M}^\vee \otimes_\Z \R$, which is just the Lie algebra of the real torus ${\rm T}(\R)$,  such that 
\begin{enumerate}
    \item Each face of a cone $\sigma \in \Sigma$ is also in $\Sigma$.
    \item The intersection of two cones $\sigma_1,\sigma_2 \in \Sigma$ is a face of each of them.
\end{enumerate}
Each cone $\sigma \in  \Sigma$ gives rise to an affine toric variety
$$
{\rm U}_\sigma = \spec \C[S_\sigma],
$$
where $S_\sigma = \sigma^\vee \cap {\rm M}$ is a semigroup,   $\sigma^\vee$ is the cone dual to $\sigma$, and $\C[S_\sigma]$  is its semigroup algebra:
$$
\C[S_\sigma]=\left\{\sum_{m \in S_\sigma} c_m \chi^m: c_m \in \C, c_m=0 \text{ for all but finitely many }m \in S_\sigma \right\}.
$$
  If $\tau \subset \sigma,$ then ${\rm U}_\tau$ is an open subset of ${\rm U}_\sigma$. Gluing the affine toric varieties ${\rm U}_{\sigma_1},{\rm U}_{\sigma_2}$ along ${\rm U}_{\sigma_1 \cap \sigma_2}$ for all cones $\sigma_1,\sigma_2 \in \Sigma$, we get the toric variety $X_\Sigma$ associated to $\Sigma$. 
 
In particular,  if $\sigma = \{0\}$, then $S_\sigma={\rm M}$, so $\C[S_\sigma]=\C[{\rm M}]$ and ${\rm U}_\sigma = {\rm T}$. So $X_\Sigma$ contains ${\rm T}$.

We define the action ${\rm T} \times {\rm U}_\sigma \lra {\rm U}_\sigma$ via the dual map of the algebras of   functions: 
\begin{align*}
    \C[S_\sigma] &\longrightarrow \C[{\rm M}] \otimes \C[S_\sigma],  \\
      \chi^m &\longmapsto \chi^m \otimes \chi^m. 
\end{align*}
When $\sigma = \{0\}$, this is the action of ${\rm T}$ on itself. The action of ${\rm T}$ on ${\rm U}_\sigma$ is compatible with the gluing, and therefore gives an action of ${\rm T}$ on $X_\Sigma$.

We denote by $\Sigma(r)$ the set of $r$-dimensional cones of $\Sigma$. There is an inclusion-reversing bijection between ${\rm T}$-orbit closures in $X_\Sigma$ and cones in $\Sigma$. Under this bijection, each ray $\rho \in \Sigma(1)$ corresponds to a ${\rm T}$-invariant divisor $D_\rho$. Let $u_\rho$ be
 the primitive vector generating $\rho$. Then, the (Weil) divisor 
 of the character $\chi^m$ is
\be
\div \chi^m=\sum_{\rho \in \Sigma(1)} \langle m, u_\rho \rangle D_\rho. \la{divchar}
\ee
The following fundamental exact sequence  computes the class group of Weil divisors of $X_\Sigma$:
\begin{align}
0 \ra {\rm M} &\ra \Z^{\Sigma(1)} \ra \text{Cl}(X_\Sigma) \ra 0, \label{cl::gp}\\
m &\mapsto (\langle m, u_\rho \rangle )_{\rho \in \Sigma(1)}.\nonumber 
\end{align}

\subsection{Polygons and projective toric surfaces} \la{A2}
Given a  convex integral polygon $N$ in the  plane  ${\rm M}_\R:={\rm M} \otimes_\Z \R$, 
we   construct  the \textit{normal fan $\Sigma$} of $N$ as follows:
\begin{enumerate}
    \item $\Sigma(0)=\{0\}$. 
    \item For each edge $E_\rho$ of $N$, let $u_\rho\in {\rm M}^\vee$ be the primitive inward normal vector to   $E_\rho$, providing an element of $\Sigma(1)$ given by the ray spanned by  $u_\rho$.
    \item For each vertex $v$ of $N$, we get an element of $\Sigma(2)$ by taking the convex hull of the two rays in $\Sigma(1)$ associated to the two edges incident to $v$ in $N$. 
\end{enumerate} 

The normal fan $\Sigma$ gives rise to a toric surface denoted below by $X_N$. The orbit-cone correspondence  assigns to each edge $E_\rho$ of $N$  a divisor $D_\rho \cong \P^1$. These divisors   intersect according to the combinatorics of $N$. Their union is   the {\it divisor at infinity}  $X_N - {\rm T}$.

In fact the polygon $N$ determines   a pair $(X_N,D_N)$, where  $D_N$ is an ample divisor at infinity:
\[
D_N:=\sum_{\rho \in \Sigma(1)} a_\rho D_\rho, 
\]
where $a_\rho$ is such that the edge $E_\rho $ of $N$ is contained in the line $\{m \in {\rm M} \otimes \R : \langle m,u_\rho \rangle=-a_\rho \}$.

The linear system of hyperplane sections $|D_N|$ has the following properties:
\begin{enumerate}
    \item $H^0(X_N,{\cal O}_{X_N}(D_N)) \cong \bigoplus_{m \in N \cap M } \C \cdot \chi^m$.
    \item The   genus of a generic curve $C$ in $|D_N|$  is   the number of interior lattice points of $N$.
    \item Curves in $|D_N|$ intersect the divisor $D_\rho$ with multiplicity $|E_\rho|$ {(the number of primitive vectors in $E_\rho$)}.
\end{enumerate}

\subsection{Toric stacks} \la{A3}

{We follow \cite{BH09}*{Section 2}}. Given a convex integral polygon $N \subset {\rm M}_\R$, we define a \textit{stacky fan} $\boldsymbol{\Sigma}$ as the following data:
\begin{enumerate}
    \item The normal fan $\Sigma$ of $N$, defined above.
    \item For each ray $\rho \in \Sigma(1),$ the vector $|E_\rho|u_\rho$ generating the ray $\rho.$ 
\end{enumerate}

We define a fan $\widetilde \Sigma \subset \R^{\Sigma(1)}$ as follows: for $\sigma \in \Sigma$, we define $\widetilde \sigma \in \widetilde \Sigma$ by
$$
\widetilde \sigma = \cone (  e_\rho: \rho \in \sigma(1)) \subset \R^{\Sigma(1)},
$$
where $ \{e_\rho\}$ is the standard basis in   $\rho$ in $\R^{\Sigma(1)}$, and $\sigma(1)$ denotes the rays of $\Sigma$ incident to $\sigma$. Then $\widetilde \Sigma$ is the fan generated by the cones $\widetilde \sigma$ and their faces.  

Let $U_\Sigma$ be the toric variety of the fan $\widetilde \Sigma$. It is of the form $\C^{\Sigma(1)}-
(\text{closed codimension $2$ subset})$.

Consider the following map, modifying the map (\ref{cl::gp})  for polygons $N$ with a non-primitive side:
\begin{align*}
    \beta:{\rm M} &\ra \Z^{\Sigma(1)}\\
    m & \mapsto (|E_\rho| \langle m, u_\rho \rangle)_\rho.
\end{align*}
 Applying the functor $\text{Hom}_\Z(*,\C^\times)$, we get a surjective map $(\C^\times)^{\Sigma(1)} \ra {\rm T} $. Denote by $G$ its kernel. So there is an exact sequence
\be \la{eq:gexpl}
 1 \ra G \ra (\C^\times)^{\Sigma(1)} \ra {\rm T} \ra 1. 
\ee
So $G$ is a subgroup of the torus $(\C^\times)^{\Sigma(1)}$ of the toric variety $U_\Sigma$. Therefore, $G$ acts on $U_\Sigma$.

Explicitly, $\lambda=(\lambda_\rho) \in (\C^\times)^{\Sigma(1)}$ is in $G$ if and only if 
\be 
\prod_{\rho \in \Sigma(1)} \lambda_\rho^{|E_\rho|\langle m,u_\rho \rangle}=1
\ee
for all $m \in {\rm M}$. Let $z=(z_\rho) \in \C^{\Sigma(1)}$ denote the standard coordinates on $ \C^{\Sigma(1)}$. The action of $G$ on $U_{\Sigma}$ is $\lambda \cdot z = (\lambda_\rho z_\rho)$.

\bd The \textit{toric stack} $\mathscr X_N$ is the smooth {Deligne-Mumford} stack $\left[U_\Sigma/G \right]$.
\ed

\subsection{Example: a stacky \texorpdfstring{$\mathbb P^2$}{P2}.}

Consider the polygon $N$ given by  the convex-hull of $\{(0,0),(2,0),(0,2)\}$. The rays of its normal fan $\Sigma$ are
 generated by $u_1=(1,0), u_2=(0,1), u_3=(-1,-1)$ with $|E_1|=|E_2|=|E_3|=2$. The fan $\widetilde \Sigma \subset \R^3$ is generated by the cones
$$
\widetilde \sigma_1=\cone ( {e_2}, {e_3}), \quad \widetilde \sigma_2=\cone ( {e_1}, {e_3}),\quad \widetilde \sigma_3=\cone ( {e_1}, {e_2}),
$$
and their faces, where $\{e_i\}$ is the standard basis of $\R^3$. These cones define affine varieties
\[
U_1=\spec \C[X_1^{\pm 1},X_2,X_3], \quad U_2=\spec \C[X_1,X_2^{\pm 1},X_3], \quad U_3=\spec \C[X_1,X_2,X_3^{\pm 1}],
\]
respectively. The face $\widetilde \sigma_{12}:=\widetilde \sigma_{1}\cap \widetilde \sigma_{2}=\cone (e_3)$ defines the affine variety 
$U_{12}=\spec \C[X_1^{\pm 1},X_2^{\pm 1},X_3]$, identified with $U_1 \cap U_2$. Similarly, we define $U_{23}$ and $U_{13}$. Gluing $U_i$ {and $U_j$} along the $U_{ij}$ {for all $i,j$}, we see that the toric variety $U_\Sigma$ of $\widetilde \Sigma$ is $\mathbb C^3-0$. 
The map  $M \ra \Z^{\Sigma(1)}$ is 
\begin{align*}
    \Z^2 &\ra \Z^{3}\\
    (1,0) &\mapsto (2,0,-2)\\
    (0,1) &\mapsto (0,2,-2).
\end{align*}
The group $G$ is the kernel of 
\begin{align*}
    (\C^\times)^3 &\ra (\C^\times)^{2}\\
    (t_1,t_2,t_3) &\mapsto {\left(  \left(\frac{t_1}{t
_3}\right)^2,\left(\frac{t_2}{t
_3}\right)^2
\right).}
\end{align*}

Thus, $G=\{(\pm \lambda,\pm \lambda,\lambda):\lambda \in \C^\times\}$ and it acts on $\C^3 -  0$ by multiplication. The quotient $[\C^3 - 0/G]$ is a stacky $\P^2$.

\subsection{Line bundles and divisors on toric stacks} \la{A5}

A line bundle on the quotient stack $\mathscr X_N=[U_\Sigma/G]$ is the same thing as a $G$-equivariant line bundle on $U_\Sigma$. The Picard group of $U_\Sigma$ is trivial, so line bundles on $\mathscr X_N$ correspond to the various $G$-linearizations of $\mathcal O_{U_\Sigma}$.

\begin{proposition}[Borisov and Hua, 2009 \cite{BH09}*{Proposition 3.3}]
There is an isomorphism, describing the  Picard group of $\mathscr X_N$    via    divisors   $D_\rho$: 
\begin{align*}
    \Z^{\Sigma(1)}/ \beta^* {\rm M} &\cong  {\rm{Pic}}~\mathscr X_N, \\
    (b_\rho)_\rho &\mapsto \mathcal O_{\mathscr X_N}
    \Bigl( \sum_{\rho \in \Sigma(1)} \frac{b_\rho}{|E_\rho|} D_\rho \Bigr).
\end{align*}
The line bundle $\mathcal O_{\mathscr X_N}
    \Bigl( \sum_{\rho \in \Sigma(1)} \frac{b_\rho}{|E_\rho|} D_\rho \Bigr)$ is the trivial line bundle $\mathcal O_{U_\Sigma} = U_\Sigma \times \C$ with the $G$-linearization
\begin{align*}
G \times (U_\Sigma \times \C) &\ra U_\Sigma \times \C\\
\lambda \cdot (z,t) &\mapsto \left(\lambda \cdot z,t \prod_{\rho \in \Sigma(1)} \lambda_\rho^{{b_\rho}} \right).
\end{align*}
\end{proposition}

Let $D=\sum_{\rho \in \Sigma(1)} \frac{b_\rho}{|E_\rho|} D_\rho$ be a divisor at infinity  on  $\mathscr X_N$. We assign to $D$  a polygon $P_D$ in  ${\rm M}_\R$ defined by the intersection of the half planes provided by the coefficients of $D$:
\begin{align}
P_D:= \bigcap_{\rho \in \Sigma(1)}\left\{m \in {\rm M}_\R: \langle m,u_\rho \rangle \geq -\frac{b_\rho}{|E_\rho|}  \right\}. \label{divisorpolygonbijection}
\end{align}
A global section of a line bundle on $\mathscr X_N$ is the same thing as a $G$-invariant global section of $\mathcal O_{U_\Sigma}$. As in the case of toric varieties, global sections of toric line bundles are identified with integral points in the associated polygons: 

\begin{proposition} [Borisov and Hua, 2009 \cite{BH09}*{Proposition 4.1}]\la{pro:globalsec}
We have \[
H^0(\mathscr X_N, \mathcal O_{\mathscr X_N}(D))\cong \bigoplus_{m \in   P_D \cap {\rm M}} \C \cdot \chi^m.\]
The $G$-invariant section of $\mathcal O_{U_\Sigma}$ corresponding to $\chi^m, m \in P_D \cap {\rm M}$, is $\prod_{\rho \in \Sigma(1)} z_\rho^{a_\rho}$, where $a_\rho = |E_\rho|\langle m,u_\rho \rangle +b_\rho$.
\end{proposition}
\begin{proof}
We have $H^0(U_\Sigma,\mathcal O_{U_\Sigma})=\C[z_\rho:\rho \in \Sigma(1)]$. The global section $\prod_{\rho \in \Sigma(1)} z_\rho^{a_\rho}$ is $G$-invariant if and only if 
\[
\prod_{\rho \in \Sigma(1)}\lambda_\rho^{{b_\rho}} \cdot \prod_{\rho \in \Sigma(1)}z_\rho^{a_\rho}=\prod_{\rho \in \Sigma(1)}(z_\rho \lambda_\rho)^{a_\rho}~\text{for all $\rho \in \Sigma(1)$},
\]
which is equivalent to the equations $\prod_{\rho \in \Sigma(1)}\lambda_\rho^{b_\rho-a_\rho}=1$ for all $\rho \in \Sigma(1)$. By exactness of (\ref{eq:gexpl}), this is equivalent to the existence of $m \in {\rm M}$ such that $a_\rho-{b_\rho}={|E_\rho|}\langle m,u_\rho \rangle$ {for all $\rho \in \Sigma(1)$}.
\end{proof}

\section{Combinatorial rules for the linear system of equations \texorpdfstring{$\mathbb V_{\bw \w}$}{Vbw}} \la{B}

In this appendix, we collect some combinatorial rules that facilitate the computation of the small polygons and equations in $\mathbb V_{\bw \w}$.

\subsection{Equivalent description of the small polygons}

Consider the lines  
\be \la{line}
L_{\rho}:= \{ m \in {\rm M}_\R: \langle m,u_\rho \rangle =-b_\rho\} 
\ee
that form the boundary of the small Newton polygon $N_{\bw \w}$.  We give an alternate description of these lines.  Recall that $\widetilde \Gamma$ be the biperiodic graph on the plane given by the lift of $\Gamma$ to the universal cover of $\T$. The zig-zag paths in $\widetilde \Gamma$ for a given $\rho$ divide the plane into an infinite collection of strips $\mathscr S_\rho(d)$  parameterized by $d \in \frac 1 {|E_\rho|} \mathbb Z$ such that 
\[
\mathscr S_\rho(d) \cap V(\widetilde \Gamma)=\{{\rm v} \in V(\widetilde \Gamma): {[D_\rho]}\dd({\rm v})=d\},
\]
{where for a divisor $D$, $[D_\rho] D$ denotes the coefficient of $D_\rho$ in $D$.} We assign to each strip $\mathscr S_\rho(d)$  a line $L_\rho(d)$ in ${\rm M}_\R$ parallel to $E_\rho$, using the following rule illustrated on  Figure \ref{zzlocQ}:

\begin{enumerate}

\item The line associated to a strip  ${\mathscr  S}_\rho(d)$  contains the side $E_\rho$ if and only if either

\begin{itemize}
\item [i)]  The strip ${\mathscr S}_\rho(d)$ is on the right {(when facing in the direction of the path)} of a zig-zag path $\alpha_1 \in Z_\rho$, and   $\alpha_1$ contains $\bw$. 

\item[ii)] The strip ${\mathscr  S}_\rho(d)$ contains   $\bw$, and  $\bw$ is not in a zig-zag path in $Z_\rho$, or
\end{itemize}

\item Moving to the strip to the left shifts the line $1/|E_\rho|$ steps to the left.
\end{enumerate}
We call the strip to the left of the one whose line contains $E_\rho$, and all strips obtained by  its translations by $H_1(\T,\Z)$, \emph{exceptional strips}.

\begin{proposition}\la{propsmalleq}
If we associate lines to strips as above, the boundary of the small Newton polygon $N_{\bw \w}$ is given by the lines $\{L_\rho(d_\rho)\}$, where $d_\rho \in \frac 1 {|E_\rho|}\Z$ is determined by the condition $\w \in \mathscr S_\rho(d_\rho)$, that is, it is the index of the strip containing $\w$ in the direction $\rho$.
\end{proposition}

\begin{proof}
In order for the line $L_\rho$ in (\ref{line}) to contain $E_\rho$, we must have $b_\rho=0,$ where $b_\rho$ is the coefficient of $D_\rho$ in (\ref{br}). We have to consider two cases.
\begin{enumerate}
	\item There is a zig-zag path $\alpha \in Z_\rho$ such that $\bw $ is contained in $\alpha$. We need $[D_\rho](\dd(\bw))=\frac 1 {|E_{\rho}|}+ [D_\rho](\dd(\w))$, which means $\w$ is contained in the strip $\mathscr S$ to the right of the one containing $\bw$, with $\alpha$ separating the two strips.  
    \item No zig-zag path in $Z_\rho$ contains $\bw$. In this case, we need the coefficients  $[D_\rho](\dd(\bw))=[D_\rho](D(\w))$, which means $\w$ is in the strip $\mathscr S$ containing $\bw$.
\end{enumerate}

If $\w_2$ is a white vertex in the strip to the left of the strip containing a white $\w_1$, then   $[D_\rho](\dd(\w_2))=[D_\rho](\dd(\w_1))+\frac{1}{|E_\rho|}$. So if we define $b_\rho(\w_1)$ and $b_\rho(\w_2)$ as in (\ref{br}) with $\w=\w_1$ and $\w=\w_2$ respectively, then $b_\rho(\w_2)=b_\rho(\w_1)+\frac 1 {|E_\rho|}$. Note that the line (\ref{line}) which bounds $N_{\bw \w_2}$ is given by 
\[
L_\rho(\w_2):=\{m \in {\rm M}_\R : \langle m,u_\rho \rangle = b_\rho(\w_2)\}. 
\]
The similar line  which bounds  $N_{\bw \w_1}$ is
\[
L_\rho(\w_1):=\{m \in {\rm M}_\R : \langle m,u_\rho \rangle = b_\rho(\w_1)\},
\]
so the line  $L_\rho(\w_2)$ is obtained from the line $L_\rho(\w_1)$ by shifting $1/{|E_\rho|}$ steps to the left.
\end{proof}

\subsection{The equations in \texorpdfstring{$\mathbb V_{\bw \w}$}{Vbw}}
We describe the equations of type 2 in Section \ref{S2.1.2}.

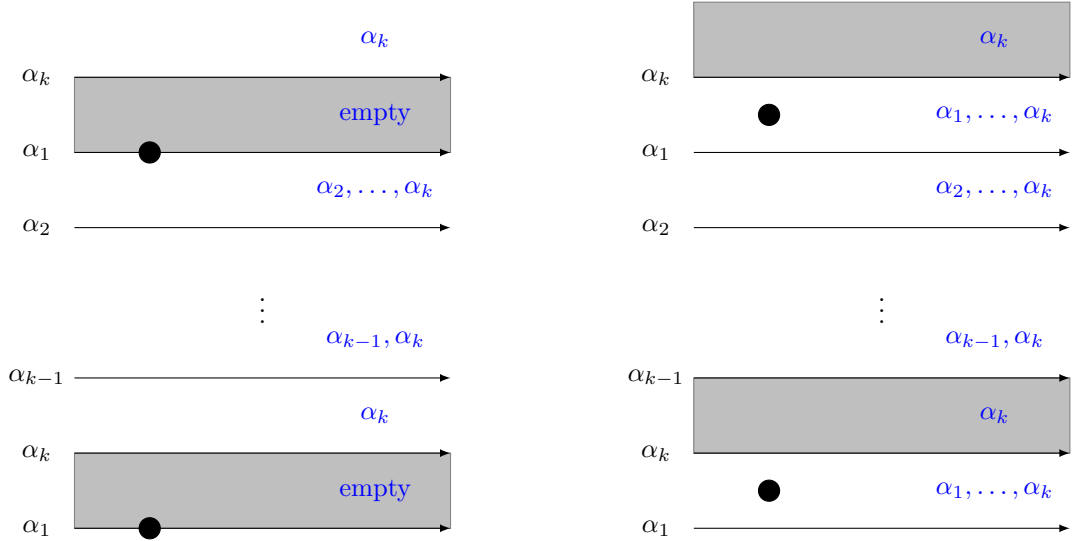
\begin{figure}\label{fignp}
	\centering
	\begin{tikzpicture}
		
		\draw[gray,line width=0.01, fill opacity=0.5, fill=gray] (-1,0) rectangle (4,1);
		\draw[gray,line width=0.01, fill opacity=0.5, fill=gray] (-1,5) rectangle (4,6);
		\draw[fill=black] (0,0) circle (4pt);
		\draw[fill=black] (0,5) circle (4pt);
		\node[] (no) at (-1.5,0){$\alpha_1$};
		\node[] (no) at (-1.5,1){$\alpha_k$};
		\node[] (no) at (-1.5,2){$\alpha_{k-1}$};
		\node[] (no) at (-1.5,4){$\alpha_{2}$};
		\node[] (no) at (-1.5,5){$\alpha_{1}$};
		\node[] (no) at (-1.5,6){$\alpha_k$};
		\draw[->](-1,0) -- (4,0);
		\draw[->](-1,1) -- (4,1);
		\draw[->](-1,2) -- (4,2);
		\draw[->](-1,4) -- (4,4);
		\draw[->](-1,5) -- (4,5);
		\draw[->](-1,6) -- (4,6);
		\node[blue] (no) at (3,0.5){empty};
		\node[blue] (no) at (3,1.5){$\alpha_k$};
		\node[blue] (no) at (3,4.5){$\alpha_2,\dots,\alpha_k$};
			\node[blue] (no) at (3,2.5){$\alpha_{k-1},\alpha_k$};
		\node[blue] (no) at (3,5.5){empty};
		\node[blue] (no) at (3,6.5){$\alpha_k$};
		\node[](no) at (1.5,3){$\vdots$};
	\end{tikzpicture} \hspace{20mm}
	\begin{tikzpicture}
		\draw[fill=black] (0,0.5) circle (4pt);
		\draw[fill=black] (0,5.5) circle (4pt);
		\draw[gray,line width=0.01, fill opacity=0.5, fill=gray] (-1,1) rectangle (4,2);
		\draw[gray,line width=0.01, fill opacity=0.5, fill=gray] (-1,6) rectangle (4,7);
		\node[] (no) at (-1.5,0){$\alpha_1$};
		\node[] (no) at (-1.5,1){$\alpha_k$};
		\node[] (no) at (-1.5,2){$\alpha_{k-1}$};
		\node[] (no) at (-1.5,4){$\alpha_{2}$};
		\node[] (no) at (-1.5,5){$\alpha_{1}$};
		\node[] (no) at (-1.5,6){$\alpha_k$};
		\draw[->](-1,0) -- (4,0);
		\draw[->](-1,1) -- (4,1);
		\draw[->](-1,2) -- (4,2);
		\draw[->](-1,4) -- (4,4);
		\draw[->](-1,5) -- (4,5);
		\draw[->](-1,6) -- (4,6);
		\node[blue] (no) at (3,0.5){$\alpha_1,\dots,\alpha_k$};
		\node[blue] (no) at (3,1.5){$\alpha_k$};
			\node[blue] (no) at (3,2.5){$\alpha_{k-1},\alpha_k$};
		\node[blue] (no) at (3,4.5){$\alpha_2,\dots,\alpha_k$};
		\node[blue] (no) at (3,5.5){$\alpha_1,\dots,\alpha_k$};
		\node[blue] (no) at (3,6.5){$\alpha_k$};
		\node[](no) at (1.5,3){$\vdots$};
	\end{tikzpicture} 
	\caption{ 
		The lifts of zig-zag paths $\alpha_1,\dots,\alpha_k$ in $Z_\rho$ divide the plane into strips. The side $L_{\rho}$ of the  small polygon $N_{\bw \w}$ and the columns of the matrix  $\mathbb V_{\bw \w}$ are determined by the strips containing $\bw$ and $\w$. The black vertex ${\bw}$ is the black dot. {On the left panel, $\bw$ is on a zig-zag path, and on the right, it is between two zig-zag paths}. Written inside each strip in blue is the subset of $Z_\rho$ that gives rise to equations in $\mathbb V_{\bw \w}$ if $\w$ is contained in that strip.  Exceptional strips are shaded. } \label{zzlocQ}  \end{figure} 

Let $\rho \in \Sigma(1)$ be a ray and let $Z_\rho=\{\alpha_1,\dots,\alpha_k\}$, where $\alpha_1,\dots,\alpha_k$ are labeled in cyclic order. Their lifts  to the universal cover of the torus divides it
 into strips, see Figure \ref{zzlocQ}. We denote  by ${\mathscr S}_i$ the strip immediately to the right of $\alpha_i$.
 
 \begin{proposition}
 The set of extra linear equations is described as follows:
 \begin{enumerate}
        \item  One of these zig-zag paths contains  ${\bw}$. We can assume it is $\alpha_1$. Then the subset of $Z_\rho$ that contributes an equation to $\mathbb V_{\bw \w}$ is:
             \begin{align}
           \text{empty} &\text{ if }\w \in {\mathscr S}_k; \nonumber\\
            \alpha_{i+1}, \dots, \alpha_k  &\text{ if }\w \in {\mathscr S}_i \text{ for some }i \neq k.\label{colcase1}
        \end{align}
        \item The vertex ${\bw}$ is not in any of  zig-zag paths  in $Z_\rho$. Then the   subset of $Z_\rho$ is
        \begin{align}
            \alpha_1,\dots, \alpha_k &\text{ if }\w \in {\mathscr S}_k; \nonumber \\
            \alpha_{i+1},\dots \alpha_k  &\text{ if }\w \in {\mathscr S}_i, \text{ for some }i \neq k.\label{colcase2}
        \end{align}
    \end{enumerate}
 \end{proposition}
 \begin{proof}
Plugging \[
\divebw_{\bw \w}=D_N-\dd(\w)+\dd(\bw)-\sum_{\rho \in \Sigma(1)}\sum_{\alpha \in Z_\rho : \bw \in \alpha } \frac{1}{|E_{\rho}|}D_{\rho}.
\]  into (\ref{extraeq}), we first observe that (\ref{extraeq}) does not change if we replace $\w$ by its any translate  on the universal cover because ${\bf d}(\w)-{\bf d}({ \bw})$ changes by the same amount as $\restr{[{E_{{\bw}\w}}]}{\spectralcurve}$ but with the opposite sign. Therefore we may assume that among all its possible translates in the universal cover, the strip $\mathscr S_i$ is the one that is immediately to the right of ${\bw}$. Then we have
\begin{align*}
\restr{({\bf d}(\w)-{\bf d}({\bw}))}{\spectralcurve \cap D_\rho}&=-\nu(\alpha_1)- \cdots -\nu{(\alpha_i)}\end{align*}
and the coefficient of $D_\rho$ in ${(\dd(\w)-\dd({\bw}))}$ is $ -\frac{i}{k} $.

Now we distinguish two cases:
\begin{enumerate}
    \item Suppose $\bw$ is contained in $\alpha_1$, so that the coefficient of $D_\rho$ in $ \sum_{\rho \in \Sigma(1)}\sum_{\alpha \in Z_\rho : \bw \in \alpha } \frac{1}{|E_{\rho}|}D_{\rho}$ is $\frac 1 k . $ Then we have 
    \[
    \restr{\floor{-\dd(\w)+\dd(\bw)-\sum_{\rho \in \Sigma(1)}\sum_{\alpha \in Z_\rho : \bw \in \alpha } \frac{1}{|E_{\rho}|}D_{\rho}}}{ \spectralcurve \cap D_\rho }=0,
 \]
    so that (\ref{extraeq}) is 
    $\nu({\alpha_{i+1}})+\cdots+\nu({\alpha_k})$, which proves  (\ref{colcase1}). 
    \item Suppose $\bw$ is not contained in any of the $\alpha_i$, so that the coefficient of $D_\rho$ in  $\sum_{\rho \in \Sigma(1)}\sum_{\alpha \in Z_\rho : \bw \in \alpha } \frac{1}{|E_{\rho}|}D_{\rho}$ is $0.$ We have 
    \[
    \restr{\floor{-\dd(\w)+\dd(\bw)-\sum_{\rho \in \Sigma(1)}\sum_{\alpha \in Z_\rho : \bw \in \alpha } \frac{1}{|E_{\rho}|}D_{\rho}}}{ \spectralcurve \cap D_\rho }=\begin{cases}0 &\text{ if }i \neq k,\\
    \sum_{j=1}^k \nu({\alpha_j}) &\text{ if }i=k.
    \end{cases}
    \]
    This gives 
    \begin{align*}
    (\ref{extraeq})=\begin{cases}
    \nu({\alpha_{i+1}})+\cdots +\nu(\alpha_k) &\text{ if }i \neq k,\\
    \sum_{j=1}^k \nu({\alpha_j}) &\text{ if }i=k.\end{cases}\end{align*} 
  We obtain (\ref{colcase2}).
\end{enumerate}
\end{proof}

\bibliography{references}

\providecommand{\MR}[1]{}
\begin{bibdiv}
\begin{biblist}

\bib{AGR}{article}{
      author={Affolter, Niklas},
      author={George, Terrence},
      author={Ramassamy, Sanjay},
       title={Cross-ratio dynamics and the dimer cluster integrable system},
        date={2021},
      eprint={arXiv:2108.12692},
}

\bib{BCdT}{article}{
      author={Boutillier, Cédric},
      author={Cimasoni, David},
      author={de~Tilière, Béatrice},
       title={Elliptic dimers on minimal graphs and genus 1 harnack curves},
        date={2020},
      eprint={arXiv:2007.14699},
}

\bib{BCdT1}{misc}{
      author={Boutillier, Cédric},
      author={Cimasoni, David},
      author={de~Tilière, Béatrice},
       title={Minimal bipartite dimers and higher genus harnack curves},
        date={2021},
}

\bib{Beau}{incollection}{
      author={Beauville, Arnaud},
       title={Theta functions, old and new},
        date={2013},
   booktitle={Open problems and surveys of contemporary mathematics},
      series={Surv. Mod. Math.},
      volume={6},
   publisher={Int. Press, Somerville, MA},
       pages={99\ndash 132},
      review={\MR{3204388}},
}

\bib{BH09}{article}{
      author={Borisov, Lev},
      author={Hua, Zheng},
       title={On the conjecture of {K}ing for smooth toric {D}eligne-{M}umford
  stacks},
        date={2009},
        ISSN={0001-8708},
     journal={Adv. Math.},
      volume={221},
      number={1},
       pages={277\ndash 301},
         url={https://doi.org/10.1016/j.aim.2008.11.017},
      review={\MR{2509327}},
}

\bib{CKP}{article}{
      author={Cohn, Henry},
      author={Kenyon, Richard},
      author={Propp, James},
       title={A variational principle for domino tilings},
        date={2001},
        ISSN={0894-0347},
     journal={J. Amer. Math. Soc.},
      volume={14},
      number={2},
       pages={297\ndash 346},
         url={https://doi.org/10.1090/S0894-0347-00-00355-6},
      review={\MR{1815214}},
}

\bib{CLS}{book}{
      author={Cox, David~A.},
      author={Little, John~B.},
      author={Schenck, Henry~K.},
       title={Toric varieties},
      series={Graduate Studies in Mathematics},
   publisher={American Mathematical Society, Providence, RI},
        date={2011},
      volume={124},
        ISBN={978-0-8218-4819-7},
         url={https://doi.org/10.1090/gsm/124},
      review={\MR{2810322}},
}

\bib{EKLP1}{article}{
      author={Elkies, Noam},
      author={Kuperberg, Greg},
      author={Larsen, Michael},
      author={Propp, James},
       title={Alternating-sign matrices and domino tilings. {I}},
        date={1992},
        ISSN={0925-9899},
     journal={J. Algebraic Combin.},
      volume={1},
      number={2},
       pages={111\ndash 132},
         url={https://doi.org/10.1023/A:1022420103267},
      review={\MR{1226347}},
}

\bib{EKLP2}{article}{
      author={Elkies, Noam},
      author={Kuperberg, Greg},
      author={Larsen, Michael},
      author={Propp, James},
       title={Alternating-sign matrices and domino tilings. {II}},
        date={1992},
        ISSN={0925-9899},
     journal={J. Algebraic Combin.},
      volume={1},
      number={3},
       pages={219\ndash 234},
         url={https://doi.org/10.1023/A:1022483817303},
      review={\MR{1194076}},
}

\bib{FM}{incollection}{
      author={Fock, Vladimir~V.},
      author={Marshakov, Andrey},
       title={Loop groups, clusters, dimers and integrable systems},
        date={2016},
   booktitle={Geometry and quantization of moduli spaces},
      series={Adv. Courses Math. CRM Barcelona},
   publisher={Birkh\"{a}user/Springer, Cham},
       pages={1\ndash 66},
      review={\MR{3675462}},
}

\bib{Fock}{article}{
      author={Fock, V.~V.},
       title={Inverse spectral problem for {GK} integrable systems},
        date={2015},
      eprint={arXiv:1503.00289},
}

\bib{Fulton}{book}{
      author={Fulton, William},
       title={Introduction to toric varieties},
      series={Annals of Mathematics Studies},
   publisher={Princeton University Press, Princeton, NJ},
        date={1993},
      volume={131},
        ISBN={0-691-00049-2},
         url={https://doi.org/10.1515/9781400882526},
        note={The William H. Roever Lectures in Geometry},
      review={\MR{1234037}},
}

\bib{GK12}{article}{
      author={Goncharov, Alexander~B.},
      author={Kenyon, Richard},
       title={Dimers and cluster integrable systems},
        date={2013},
        ISSN={0012-9593},
     journal={Ann. Sci. \'{E}c. Norm. Sup\'{e}r. (4)},
      volume={46},
      number={5},
       pages={747\ndash 813},
         url={https://doi.org/10.24033/asens.2201},
      review={\MR{3185352}},
}

\bib{HK}{article}{
      author={Hanany, Amihay},
      author={Kennaway, Kristian~D.},
       title={{Dimer models and toric diagrams}},
        date={2005},
      eprint={hep-th/0503149},
}

\bib{Kast61}{article}{
      author={Kasteleyn, P.~W.},
       title={{The statistics of dimers on a lattice: I. The number of dimer
  arrangements on a quadratic lattice}},
        date={1961},
     journal={Physica},
      volume={27},
       pages={1209\ndash 1225},
}

\bib{Kast63}{article}{
      author={Kasteleyn, P.~W.},
       title={Dimer statistics and phase transitions},
        date={1963},
        ISSN={0022-2488},
     journal={J. Mathematical Phys.},
      volume={4},
       pages={287\ndash 293},
         url={https://doi.org/10.1063/1.1703953},
      review={\MR{153427}},
}

\bib{K00}{article}{
      author={Kenyon, Richard},
       title={Conformal invariance of domino tiling},
        date={2000},
        ISSN={0091-1798},
     journal={Ann. Probab.},
      volume={28},
      number={2},
       pages={759\ndash 795},
         url={https://doi.org/10.1214/aop/1019160260},
      review={\MR{1782431}},
}

\bib{Kenyon.isoradial}{article}{
      author={Kenyon, R.},
       title={The {L}aplacian and {D}irac operators on critical planar graphs},
        date={2002},
        ISSN={0020-9910},
     journal={Invent. Math.},
      volume={150},
      number={2},
       pages={409\ndash 439},
         url={https://doi.org/10.1007/s00222-002-0249-4},
      review={\MR{1933589}},
}

\bib{K97}{article}{
      author={Kenyon, Richard},
       title={Local statistics of lattice dimers},
        date={1997},
        ISSN={0246-0203},
     journal={Ann. Inst. H. Poincar\'{e} Probab. Statist.},
      volume={33},
      number={5},
       pages={591\ndash 618},
         url={https://doi.org/10.1016/S0246-0203(97)80106-9},
      review={\MR{1473567}},
}

\bib{KO}{article}{
      author={Kenyon, Richard},
      author={Okounkov, Andrei},
       title={Planar dimers and {H}arnack curves},
        date={2006},
        ISSN={0012-7094},
     journal={Duke Math. J.},
      volume={131},
      number={3},
       pages={499\ndash 524},
         url={https://doi.org/10.1215/S0012-7094-06-13134-4},
      review={\MR{2219249}},
}

\bib{KOS}{article}{
      author={Kenyon, Richard},
      author={Okounkov, Andrei},
      author={Sheffield, Scott},
       title={Dimers and amoebae},
        date={2006},
        ISSN={0003-486X},
     journal={Ann. of Math. (2)},
      volume={163},
      number={3},
       pages={1019\ndash 1056},
         url={https://doi.org/10.4007/annals.2006.163.1019},
      review={\MR{2215138}},
}

\bib{TF61}{article}{
      author={Temperley, H. N.~V.},
      author={Fisher, Michael~E.},
       title={Dimer problem in statistical mechanics---an exact result},
        date={1961},
        ISSN={0031-8086},
     journal={Philos. Mag. (8)},
      volume={6},
       pages={1061\ndash 1063},
      review={\MR{136398}},
}

\bib{TWZ18}{article}{
      author={Treumann, David},
      author={Williams, Harold},
      author={Zaslow, Eric},
       title={Kasteleyn operators from mirror symmetry},
        date={2019},
        ISSN={1022-1824},
     journal={Selecta Math. (N.S.)},
      volume={25},
      number={4},
       pages={Paper No. 60, 36},
         url={https://doi.org/10.1007/s00029-019-0506-7},
      review={\MR{4016520}},
}

\end{biblist}
\end{bibdiv}
\Addresses

\end{document}